\documentclass{article}

\usepackage[utf8]{inputenc}
\usepackage[margin=1.2in]{geometry}
\usepackage{amsmath,amssymb,amsfonts,amsthm,amscd}
\usepackage{mathtools}
\mathtoolsset{showonlyrefs,showmanualtags}
\usepackage{bm}

\usepackage{booktabs}
\usepackage{array}
\usepackage{geometry}

\usepackage{booktabs}
\usepackage{float}
\usepackage{placeins}
\usepackage{xcolor}
\usepackage{graphicx}
\usepackage{wrapfig}
\usepackage{subcaption}

\usepackage[color=green!20]{todonotes}

\usepackage{hyperref}
\hypersetup{
  colorlinks=true,
  linkcolor=red,
  citecolor=blue,
  filecolor=magenta,
  urlcolor=cyan
}

\newtheorem{theorem}{Theorem}[section]

\newtheorem{corollary}{Corollary}[section]

\newtheorem{definition}{Definition}[section]
\newtheorem{example}{Example}[section]

\newtheorem{lemma}{Lemma}[section]

\newtheorem{proposition}{Proposition}[section]
\newtheorem{remark}{Remark}[section]

\newtheorem{assumption}{Assumption}[section]

\DeclareMathOperator{\grad}{grad}
\DeclareMathOperator{\curl}{curl}
\DeclareMathOperator{\divergence}{div}

\newcommand{\f}[1]{\mathfrak{#1}}

\newcommand{\e}[1]{\mathcal{#1}}

\newcommand{\aip}[2][]{\langle #2 \rangle_{#1}}

\newcommand{\keyterm}[2][]{\textbf{#2}}
\newcommand{\R} {\mathbb R}

\DeclareMathOperator{\range}{range}

\newcommand{\HH}{\mathfrak H}

\newcommand{\Piop}{\boldsymbol\Pi}
\newcommand{\ip}[2]{\langle #1,#2\rangle}
\newcommand{\norm}[1]{\lVert #1\rVert}

\title{Optimal Control in Hilbert Complex Spaces with Finite Element Exterior Calculus}
\author{Farid Bozorgnia, Michael Holst, Anshu Kumar, \\ Vyacheslav Kungurtsev, Martin Licht, Gantumur Tsogtgerel}

\begin{document}

\maketitle

\begin{abstract}
We develop a framework for PDE-constrained optimal control on Hilbert complexes and its structure-preserving discretization by finite element exterior calculus (FEEC), with emphasis on problems where nontrivial topology creates physically meaningful global modes. We separate the gauge-fixed local state from harmonic components representing global circulation or flux, and introduce finite-dimensional topological actuators whose reachable modes are determined jointly by domain topology and actuator rank. We establish well-posedness, weak optimality conditions, reachability results, and FEEC error estimates preserving the harmonic structure. Numerical experiments on contractible and multiply connected domains demonstrate circulation and cavity-flux control, rank-dependent reachability, mesh-independent optimization, and topology-induced singular convergence. We further extend the framework to time-dependent Maxwell control, where harmonic electric circulation and magnetic flux become dynamical states exhibiting explicit conservation and reachability properties.
\end{abstract}

\section{Introduction}

Finite Element Exterior Calculus (FEEC) is a general framework for the construction and error analysis of mixed finite element methods.  It formulates finite element methods in the language of differential forms and Hilbert complexes, thereby exposing the algebraic, metric, and topological structures shared by many classical vector-valued finite element methods.  This manuscript develops a corresponding framework for PDE-constrained optimal control, with particular emphasis on situations in which nontrivial cohomology creates physically meaningful global circulation or flux modes in addition to the local PDE response.

The central modeling point is that the gauge-fixed mixed Hodge--Laplace state and the physical field need not coincide on a multiply connected domain.  We write
\begin{equation}\label{eq:intro-split}
    y=u+h, \qquad u\perp \HH^k, \qquad h\in\HH^k,
\end{equation}
where $u$ is the local state selected by the mixed Hodge--Laplace problem and $h$ is the physical harmonic component.  A period map
\begin{equation}\label{eq:intro-period-map}
    \Piop:\HH^k\longrightarrow\R^{b_k}
\end{equation}
provides finite-dimensional coordinates for that global state: at degree $k=1$ its components are circulations around noncontractible loops, while at degree $k=2$ they are fluxes through nontrivial two-cycles.  We introduce a finite-dimensional topological actuator $a$ through
\begin{equation}\label{eq:intro-top-control}
    \Piop h=Ga+c_0.
\end{equation}
Thus the Betti number $b_k=\dim\HH^k$ fixes the dimension of the topological state, whereas $\operatorname{rank}(G)$ fixes the dimension of the reachable topological subspace.  The harmonic multiplier $p$ appearing in the standard mixed Hodge--Laplace formulation has a different role: it is the harmonic compatibility component of the total forcing and is not used as a surrogate physical state or tracking target.

Alongside this topology-aware formulation, we retain the standard distributed-control setting.  Suppose that $V\subset Z$ is a dense continuous embedding of Hilbert spaces and that $A:V\to V'$ is linear.  For a nonempty closed convex admissible set $Z_d\subset Z$, a prototypical stationary problem is
\begin{equation}
    \min_{y\in V,\,z\in Z_d}J(y,z),\qquad
    J(y,z)=\frac12\|Oy-y_d\|_Z^2+\frac\alpha2\|z\|_Z^2,
\end{equation}
subject to
\begin{equation}
    Ay=f+z,
\end{equation}
with $\alpha>0$.  Here $O$ is a bounded observation operator and $f$ is fixed forcing.  If the state equation is well posed with solution map $S$, the fixed forcing makes the state map affine, $y=S(f+z)$, and one obtains the usual reduced problem in $z$.  The topology-aware Hodge--Laplace problem developed below supplements this familiar local control by the finite-dimensional state and actuator in \eqref{eq:intro-split}--\eqref{eq:intro-top-control}.

For time-dependent systems the same distinction persists.  A generic evolution problem has the form
\begin{equation}
    \partial_t y+Ay=f+z,\qquad y(0)=y_0,
\end{equation}
with running and terminal tracking terms.  In Maxwell's equations, however, the harmonic coordinates are not merely static kernel components: they become dynamical global states.  Harmonic electric circulation satisfies a finite-dimensional balance law driven by the harmonic projection of current and by optional loop actuators, whereas harmonic magnetic cavity flux is conserved in the absence of a physically distinct global or boundary flux actuator.  The resulting control problem therefore consists of a local FEEC Maxwell evolution coupled to a low-dimensional topology-dependent evolution system.

A prototypical topologically trivial example is the Poisson problem over a domain $\Omega$, with $Z=L^2(\Omega)$, $V=H^1_0(\Omega)$, and $A=-\Delta$.  More generally, the de Rham Hilbert complex
\[
    H^1\xrightarrow{\grad}H(\curl)\xrightarrow{\curl}H(\divergence)\xrightarrow{\divergence}L^2
\]
contains the mixed operators underlying the curl--curl, div--div, and Maxwell examples studied here.  Hilbert-complex structure is particularly useful when standard coercive formulations become inconvenient or unstable on nonconvex domains or domains with nontrivial topology.  FEEC transfers this structure to compatible finite-dimensional subcomplexes, retaining stability, cohomology, and commuting-projection properties even in the presence of re-entrant corners and other geometric complications~\cite{arnold2010finite,schwedes2017mesh}.

The contributions of the paper are the following.  First, we formulate stationary optimal control on a closed Hilbert complex with the physical decomposition \eqref{eq:intro-split}, distributed control of the local mixed PDE, and finite-dimensional control of the period coordinates.  We prove existence and uniqueness for positive regularization, derive the optimality system entirely in weak form, and characterize topology-dependent reachability and the reduced Hessian.  Second, we show how continuous and discrete period functionals can be represented without requiring traces of arbitrary finite element fields on lower-dimensional cycles, and we formulate FEEC discretizations that preserve the harmonic dimension and the physical period coordinates.  Third, the stationary experiments separate geometry from topology: a contractible L-shaped domain provides a local convergence and mesh-independence benchmark; a solid torus gives one circulation mode; a two-hole domain gives a non-orthogonal two-dimensional harmonic space and makes actuator rank and alignment observable; and a spherical shell realizes the same construction at degree two through cavity flux.  The comparison between the L-shaped and two-hole domains further shows that a re-entrant corner need not activate its singular mode, whereas a nontrivial cohomology class can force the harmonic field into the singular regularity class and thereby determine the observed convergence rate.  Fourth, we extend the framework to time-dependent Maxwell control, where the same cohomological coordinates satisfy finite-dimensional evolution laws.  The Maxwell numerical program mirrors the stationary study with a re-entrant-corner baseline, one- and two-mode circulation-control experiments, and magnetic-flux conservation and actuation on a spherical shell.

Throughout we restrict attention to linear state equations and positive control regularization.  Semilinear problems, nonlinear constitutive laws, and application-specific derivations of the finite-dimensional actuator matrices are left to subsequent work.

\subsection{Related Work}
 To the best of our knowledge, this is the first systematic PDE-constrained optimal-control
formulation at the level of closed Hilbert complexes in which nontrivial cohomology is represented
as an explicit physical state in period coordinates, distinct from the harmonic compatibility
multiplier of the mixed Hodge–Laplace problem, and controlled through a finite-dimensional
topological actuator.
For a general overview of PDE-constrained optimization, including the formalism of optimal control in function spaces, solution properties, and algorithms, see~\cite{hinze2008optimization}.

Some recent contemporary works regarding the computational procedures associated with FEEC and their properties are in~\cite{li2020improved,licht2017complexes,gawlik2021local,licht2024symmetry}. 

Maxwell's equations for modeling electromagnetism can be studied from this perspective. Works that consider optimal control of an electromagnetic system of this form,
in a range of settings and applications, include~\cite{yousept2017optimal,yousept2012optimal,nicaise2017optimal,brizitskii2010inverse}. Recent research studying the well-posedness, regularity and structure preservation in the control of electromagnetic systems given by Maxwell's equations appears in~\cite{tran2022optimal,antil2026structure,arnold2010finite}.

The elasticity tensor presents additional structure amenable to analysis within a Hilbert complex.
Optimal control in this space tends to be highly specialized to particular scenarios, even in simpler models; see, e.g.,~\cite{schiela2020optimal} and the monograph~\cite{komkov2006optimal}.



\section{Background on the Hodge--Laplace Equation}

This section recalls the algebraic and analytic structures that underlie FEEC.  The purpose is not merely to collect notation, but to
explain why cohomology, harmonic forms, and the Hodge decomposition enter the
well-posedness and discretization of the Hodge--Laplace equation.  The
presentation follows the extensive review developed in
\cite{BrLe92,arnold2010finite} and used throughout the associated literature.

\subsection{Complexes, cycles, cocycles, and cohomology}

At the highest level of abstraction, the notions of complexes, chains and cochains provide the generic algebraic structure that defines the rest of the manuscript. We begin by presenting the structure in its most general form to elucidate its potential. 

A \keyterm{cochain complex} is a
sequence of vector spaces $\{X^k\}$ and linear maps $\{d^k\}$ with the structure
\[	 X^0
	\xrightarrow{d^0}
	X^1
	\xrightarrow{d^1}
	\cdots
	\xrightarrow{d^{n-1}}
	X^n
	\xrightarrow{d^{n}}
	 0
\]
such that
\[
d^{k+1} d^k = 0
	\qquad \text{for all } k.
\]
The independence of this property from the level $k$ allows one to abbreviate it as $d^2=0$. 

Here, and in the settings considered in this manuscript, the chain is finite, i.e. $n<\infty$, however this does not have to be the case.

 The word ``cochain'' indicates
that the maps increase degree.  By contrast, a \keyterm{chain complex} has maps
that decrease degree,
\[
	0
	\xleftarrow{\partial_0}
 C_0
	\xleftarrow{\partial_1}
	C_1
	\xleftarrow{\partial_2}
	\cdots
	\xleftarrow{\partial_n}
	C_n
\]
with $\partial_{k-1}\partial_k=0$.  In a chain complex one speaks of cycles and
boundaries.  In a cochain complex one speaks of cocycles and coboundaries.  The
distinction is mostly one of direction: chains are tested by cochains, and
differential forms naturally form a cochain complex.

For a cochain complex, the $k$-cocycles and
$k$-coboundaries are
\[	\f Z^k
	:=
	\ker d^k,
	\qquad
	\f B^k
	:=
	\operatorname{range} d^{k-1}.
\]
The identity $d^2=0$ implies
\[
\f B^k \subset \f Z^k.
\]
The quotient
\[	\f H^k
	:=
	\f Z^k/\f B^k
\]
is the $k$th cohomology group of the complex.  Thus cohomology measures the
failure of closed objects to be exact.  An element of $\f Z^k$ is closed; it belongs
to $\f B^k$ precisely when it is the differential of an object of one lower degree.

\subsection{Hilbert complexes}

The preceding discussion is algebraic and topological.  Partial differential
equations require spaces with norms, inner products, adjoints, and completeness. Hilbert complexes combine this chain structure together with function spaces that are particularly appropriate for analyzing partial differential equations. 

\begin{definition}[Hilbert complex]\label{def:hilbert-complex-expanded}
	A Hilbert complex $(W,d)$ consists of Hilbert spaces $W^k$ and densely
	defined closed linear operators
	\[
    	d^k : V^k \subset W^k \longrightarrow W^{k+1},
	\]
    with domains $V^k=\operatorname{dom}(d^k)$, such that $d^{k+1}d^k=0$.
\end{definition}

The domain $V^k$ is itself a Hilbert space when equipped with the graph inner
product
\[
\langle u,v\rangle_{V^k}
	:=
	\langle u,v\rangle_{W^k}
	+
	\langle d^k u,d^k v\rangle_{W^{k+1}}.
\]
The resulting complex $(V,d)$ is called the domain complex.  


A Hilbert complex is called \keyterm{closed} when every $\f B^k$ is closed in
$W^k$.  This condition is essential for the Poincare inequality, the Hodge
decomposition, and the well-posedness of the Hodge--Laplace problem.


\paragraph{Adjoints and the dual complex}

The Hilbert space structure allows one to introduce adjoints.  For each
$d^{k-1}:V^{k-1}\subset W^{k-1}\to W^k$, the adjoint is the operator
\[
	d_k^* : V_k^* \subset W^k \longrightarrow W^{k-1}
\]
defined by the relation
\[
	\langle d_k^* u,v\rangle_{W^{k-1}}
	=
	\langle u,d^{k-1}v\rangle_{W^k}
	\qquad
	\text{for all } v\in V^{k-1}.
\]
The domain $V_k^*$ consists of those $u\in W^k$ for which the functional
$v\mapsto \langle u,d^{k-1}v\rangle$ is bounded on $W^{k-1}$.

The adjoints form a chain complex, usually called the dual complex,
\[
	0
	\longleftarrow
	W^0
	\xleftarrow{d_1^*}
	W^1
	\xleftarrow{d_2^*}
	\cdots
	\xleftarrow{d_n^*}
	W^n
	\longleftarrow
	0.
\]
We write
\[	\f Z_k^*
	:=
	\ker d_k^*,
	\qquad
	\f B_k^*
	:=
	d_{k+1}^* V_{k+1}^*.
\]

\subsection{Harmonic forms and the Hodge decomposition}

\begin{definition}[Morphisms of Hilbert complexes]
	Let $(W,d)$ and $(W',d')$ be two Hilbert complexes with domain complexes $(V,d)$ and $(V',d)$, respectively.
	$f : W\to W'$ is called a \keyterm{morphism of Hilbert complexes} if we have a sequence of bounded linear maps
	$f^{k} : W^{k} \to W^{\prime k}$ such that $f(V^{k})\subseteq V^{\prime k}$ and $d^{\prime k} \circ f^{k} = f^{k+1} \circ d^{k}$ (i.e., they \keyterm{commute} with the differentials).
\end{definition}

The central analytic objects in the Hilbert complex are the harmonic forms.

\begin{definition}[Harmonic forms]\label{def:harmonic-forms-expanded}
	The space of harmonic $k$-forms is
	\[	\f H^k
		:=
		\f Z^k \cap \f Z_k^*
		=
		\left\{
			u\in V^k\cap V_k^*
			:
			d^k u=0,\;
			d_k^*u=0
		\right\}.
	\]
\end{definition}

Thus harmonic forms are simultaneously closed and coclosed.  They are the
analytic representatives of cohomology classes.  More precisely, for a closed
Hilbert complex, every cohomology class contains a unique harmonic
representative.  Equivalently,
\[	\f H^k \simeq \f Z^k/\f B^k.
\]

For a subcomplex $(S,d)$ of $(W,d)$, the inclusion maps $i^{k}: S^{k} \hookrightarrow W^{k}$ are morphisms of Hilbert complexes.
With the above, one can show the following \keyterm[Hodge decomposition!in Hilbert complexes]{weak Hodge decomposition}.

\begin{theorem}
[Hodge decomposition]\label{thm:hodge-decomposition-expanded}
	Let $(W,d)$ be a closed Hilbert complex.  Then
\[
		W^k
		=
		\f B^k
		\oplus
		\f H^k
		\oplus
		(\f Z^k)^{\perp_{W^k}},
\]
where the sum is orthogonal in $W^k$.  Equivalently,
	\[
    W^k
		=
		\f B^k
		\oplus
		\f H^k
		\oplus
		\f B_k^*.
	\]
    Moreover, at the level of the domain complex,
\[		V^k
		=
		\f B^k
		\oplus
		\f H^k
		\oplus
		\left((\f Z^k)^{\perp_{W^k}}\cap V^k\right).
	\]
\end{theorem}

Consequently every $u\in W^k$ has a unique orthogonal decomposition
\[	u
	=
	d^{k-1}\phi
	+
	h
	+
	d_{k+1}^*\psi,
\]
where $h\in \f H^k$.  The three summands are called the exact, harmonic, and
coexact components, respectively.  This is the analytic version of the
topological fact that closed forms split into exact forms and representatives of
cohomology classes.

The Hodge decomposition also implies a Poincare inequality on the component
orthogonal to the kernel.  Namely, if $v\in (\f Z^k)^\perp\cap V^k$, then
\begin{equation}
	\|v\|_{V^k}
	\leq
	C_P \|d^k v\|_{W^{k+1}}.
\end{equation}
This inequality is one of the basic stability estimates for the mixed
Hodge--Laplace problem.

\subsection{The Hodge Laplacian}

The Hodge Laplacian combines the complex differential and its adjoint.

\begin{definition}[Hodge Laplacian]\label{def:hodge-laplacian-expanded}
	The Hodge Laplacian acting on $k$-forms is
\[		L
		=
		d d^*
		+
		d^* d.
	\]
    Its domain is
	\[	D_L
		=
		\left\{
			u\in V^k\cap V_k^*
			:
			d^k u\in V_{k+1}^*,
			\;
			d_k^*u\in V^{k-1}
		\right\}.
	\]
\end{definition}

The weak form of the Hodge--Laplace equation is: given $f\in W^k$, find
$u\in V^k\cap V_k^*$ such that
\[	\langle d u,d v\rangle
	+
	\langle d^*u,d^*v\rangle
	=
	\langle f,v\rangle
	\qquad
	\forall v\in V^k\cap V_k^*.
\]
The left-hand side is nonnegative, and
\[	\langle L u,u\rangle
	=
	\|du\|^2+\|d^*u\|^2.
\]
Therefore, harmonic forms are precisely the nullspace of the Hodge Laplacian, i.e. $\ker L
	=
	\f H^k$. This
fact explains why the equation $Lu=f$ is not invertible on all of $W^k$ when
$\f H^k$ is nontrivial.  The right-hand side must be compatible with the
harmonic nullspace, or else the equation must be augmented by a harmonic
component. The mixed form decomposition of the Hodge Laplace across the Hodge decomposition precisely addresses this challenge.

\paragraph{The mixed Hodge--Laplace problem}

The limitation of the primal formulation is that it requires approximation in the intersection
$V^k\cap V_k^*$, and conforming finite element subspaces of this intersection
are generally unavailable or inconvenient.  The remedy is to introduce the
auxiliary variable, denoted $\sigma
	=
	d^*u$, and to solve a mixed system posed across each component of the Hodge decomposition of the Hilbert complex.

The mixed Hodge--Laplace problem is: given $g\in W^k$, find
\[
	(\sigma,u,p)
	\in
	V^{k-1}\times V^k\times \f H^k
\]
such that
\begin{equation}\label{eqn:mixed-hodge-laplacian}
\begin{array}{rlrr}
	\langle \sigma,\tau\rangle
	-
	\langle u,d\tau\rangle
	&=
	0
	&&
	\forall \tau\in V^{k-1},
	\\
	\langle d\sigma,v\rangle
	+
	\langle du,dv\rangle
	+
	\langle p,v\rangle
	&=
	\langle g,v\rangle
	&&
	\forall v\in V^k,
	\\
	\langle u,q\rangle
	&=
	0
	&&
	\forall q\in \f H^k.
    \end{array}
	\end{equation}
The first equation is the weak statement of $\sigma=d^*u$.  The second equation
is the Hodge--Laplace equation with a harmonic correction.  The third equation
selects the solution orthogonal to the harmonic nullspace.

The variable $p$ is not an additional physical state in most applications. Testing the second equation with $q\in\f H^k$ shows that
\[
  p=P_{\f H}g,
\]
the harmonic component of the total data. It is introduced to make the gauge-fixed problem well-defined. This is the same harmonic structure that later appears in the finite-dimensional problem: the discrete harmonic space must approximate the continuous harmonic space in a way compatible with the cohomology of the domain.

The bilinear form associated with the mixed Hodge-Laplacian system is
\begin{align*}
	B(\sigma,u,p;\tau,v,q)
	:=
	&\,
	\langle \sigma,\tau\rangle
	-
	\langle u,d\tau\rangle
	+
	\langle d\sigma,v\rangle
	+
	\langle du,dv\rangle
	\\
	&\,
	+
	\langle p,v\rangle
	-
	\langle u,q\rangle .
\end{align*}
For a closed Hilbert complex this mixed problem satisfies an inf-sup condition,
and hence is well posed.  This is the continuous stability result that finite
element exterior calculus seeks to reproduce at the discrete level.

\subsection{de Rham Complex}\label{sec:simplicial}

In Euclidean three-dimensional vector notation, the Hodge Laplacian recovers
familiar elliptic operators.  For $0$-forms,
\[
	L u
	=
	-\divergence \grad u.
\]
For vector fields corresponding to $1$-forms,
\[	L u
	=
	-\grad \divergence u
	+
	\curl \curl u,
\]
up to the sign convention used for the scalar Laplacian.  For vector fields
corresponding to $2$-forms,
\[	L u
	=
	\curl \curl u
	-
	\grad \divergence u,
\]
again modulo the proxy convention.  The Hilbert complex
analysis is well defined as long as the sign convention carries through to the $dd^*$ and
$d^*d$ parts.

A popular and illustrative structure is the \textbf{de Rham complex}, arranged as,
\begin{gather*}
	\begin{CD}
		0
		\to
		\Lambda^0(\Omega)
		@>{d^0}>>
		\Lambda^1(\Omega)
		@>{d^1}>>
		\dots
		@>{d^{n-1}}>>
		\Lambda^n(\Omega)
		\to
		0
	\end{CD}
\end{gather*}
where $\Lambda^k(\Omega)$ is the space of smooth differential k-forms on the smooth manifold $\Omega$. In the specific case of a three-dimensional domain $\Omega\subseteq \mathbb{R}^3$, this is just the oft-studied differential complex
\[	\begin{CD}
		0
		\to
		C^{\infty}(\Omega)
		@>{\grad}>>
		C^{\infty}(\Omega)^3
		@>{\curl}>>
		C^{\infty}(\Omega)^3
		@>{\divergence}>>
		C^{\infty}(\Omega)
		\to
		0
	\end{CD}
\]
In order to transition to Hilbert complexes, we take the closures of the spaces as well as the operators in the respective $L^2$ norms.
For example, the $L^2$ closure of $\grad : C^\infty(\Omega) \rightarrow C^\infty(\Omega)^3$ is the operator $\grad : H^1(\Omega) \subseteq L^2(\Omega) \rightarrow L^2(\Omega)^3$.
This approach leads to the Hilbert complex
\begin{equation}\label{eq:rnderhamv}
	\begin{CD}
		0
		\to
		H^{1}(\Omega)
		@>{\grad}>>
		H(\curl;\Omega)
		@>{\curl}>>
		H(\divergence;\Omega)
		@>{\divergence}>>
		L^{2}(\Omega)
		\to
		0
	\end{CD}
\end{equation}

For the de Rham complex in three dimensions, the adjoint complex is the vector
calculus sequence
\begin{equation}\label{eq:adjoint-vector-derham}
	0
	\longleftarrow
	L^2(\Omega)
	\xleftarrow{-\divergence}
	\mathring H(\divergence;\Omega)
	\xleftarrow{\curl}
	\mathring H(\curl;\Omega)
	\xleftarrow{-\grad}
	\mathring H^1(\Omega)
	\longleftarrow
	0,
\end{equation}
where the rings indicate the closure of compactly supported smooth fields in the
corresponding graph norms. That is, $\mathring{H}^1(\Omega)$ denotes the closure of $C^{\infty}_0(\Omega)$ in $H^1(\Omega)$, and similarly for the other spaces The precise signs and boundary conditions depend on
the convention used to identify differential forms with vector fields, but the
main point is that $d$ and $d^*$ encode complementary information.  The operator
$d$ is tied to the complex and hence to topology; the operator $d^*$ depends on
the Hilbert space inner product and hence on metric and boundary data.


At each degree one has closed forms, exact forms, and harmonic forms.  In vector
notation these are
\begin{align*}
	\f B^0 &= \{0\},
&
	\f Z^0 &= \{u\in H^1(\Omega):\grad u=0\},\\
	\f B^1 &= \grad H^1(\Omega),
&
	\f Z^1 &= \{u\in H(\curl;\Omega):\curl u=0\},\\
	\f B^2 &= \curl H(\curl;\Omega),
&
	\f Z^2 &= \{u\in H(\divergence;\Omega):\divergence u=0\},\\
	\f B^3 &= \divergence H(\divergence;\Omega),
&
	\f Z^3 &= L^2(\Omega).
\end{align*}

The harmonic representatives are obtained by imposing the corresponding
coclosed condition. Recall that the cohomology groups are the quotients $\f H^k(\Omega)=\f Z^k/\f B^k$. Thus $\f H^1(\Omega)$ measures curl-free fields which are not gradients, while
$\f H^2(\Omega)$ measures divergence-free fields which are not curls.
 For the absolute de Rham complex above, they may be written
formally as
\begin{align*}
	\f H^0
	&=
	\{u\in H^1(\Omega): \grad u=0\},\\
	\f H^1
	&=
	\{u\in H(\curl;\Omega)\cap H_0(\divergence;\Omega):
		\curl u=0,\ \divergence u=0\},\\
	\f H^2
	&=
	\{u\in H(\divergence;\Omega)\cap H_0(\curl;\Omega):
		\divergence u=0,\ \curl u=0\},\\
	\f H^3
	&=
	\{u\in H^1_0(\Omega): \grad u=0\}.
\end{align*}
Here $H_0(\divergence;\Omega)$ denotes vanishing normal trace and
$H_0(\curl;\Omega)$ denotes vanishing tangential trace. The zero trace properties appear on account of the domain of the adjoint operator $d^*$.

\subsection{Simplicial complexes, cochain morphisms, and boundary conditions}
\label{subsec:simplicial-morphisms-boundary}

The de Rham complex is an infinite-dimensional cochain complex built from
smooth or weak differential forms.  Its finite-dimensional topological analogue
is obtained from a triangulation. Notably this triangulation presents the potential mesh that could be used as part of a finite element discretization. This Section explores how the topology of the triangulation relates to its cohomology, which ultimately connects, via the morphism to the domain complex, the dimension of the Harmonic forms of the Hilbert complex to the Betti numbers of the domain.

Let $\e T$ be a simplicial triangulation of
$\overline{\Omega}$ and let $\Delta_k(\e T)$ denote the set of oriented
$k$-simplices.  The vector space of simplicial $k$-chains is
\[	C_k(\e T)
	:=
	\left\{
		\sum_{\sigma\in\Delta_k(\e T)} a_\sigma \sigma
		:
		a_\sigma\in\R
	\right\}.
\]
The boundary operator
\[	\partial_k:C_k(\e T)\longrightarrow C_{k-1}(\e T)
\]
is defined on an oriented simplex by concatenating its oriented faces. 
Observing that its range lies in the set of loops, thus without boundaries themselves, it satisfies
\[	\partial_{k-1}\partial_k=0,
\]
and therefore the spaces $C_k(\e T)$ form a chain complex.  The associated
cochain complex is obtained by duality:
\[
	C^k(\e T)
	:=
	\operatorname{Hom}(C_k(\e T),\R).
\]
Its coboundary operator
\[	\delta^k:C^k(\e T)\longrightarrow C^{k+1}(\e T)
\]
is defined by
\[	(\delta^k c)(\sigma)
	=
	c(\partial_{k+1}\sigma),
	\qquad
	c\in C^k(\e T),\quad \sigma\in C_{k+1}(\e T).
\]
Since $\partial^2=0$, we have $\delta^{k+1}\delta^k=0$.  Thus the simplicial
cochains form a cochain complex in the same algebraic sense as the de Rham
complex.  The difference is that a simplicial cochain records one number per
oriented simplex, whereas a differential form is a field defined throughout
$\Omega$.

The bridge between these two complexes is given by integration.  For a smooth
$k$-form $\omega\in \Lambda^k(\Omega)$, define the de Rham map $\e R^k \omega \in C^k(\e T)$ by
\[	(\e R^k \omega)(\sigma)
	:=
	\int_{\sigma} \operatorname{tr}_{\sigma}\omega,
	\qquad
	\sigma\in\Delta_k(\e T),
\]
where $\operatorname{tr}_{\sigma}$ denotes the trace of the differential form on
the simplex $\sigma$.  Stokes' theorem implies the commutation relation
\[	\e R^{k+1} d\omega
	=
	\delta^k \e R^k\omega.
\]
Equivalently, the following diagram commutes:
\[
\begin{CD}
	\Lambda^k(\Omega) @>{d}>> \Lambda^{k+1}(\Omega)\\
	@V{\e R^k}VV                       @VV{\e R^{k+1}}V\\
	C^k(\e T) @>{\delta^k}>> C^{k+1}(\e T).
\end{CD}
\]
Thus $\e R$ is a cochain morphism from the de Rham complex to the simplicial
cochain complex.  This is the analytic expression of the fact that exterior
differentiation and taking the boundary of a simplex are dual operations.

There is also a morphism in the opposite direction, the Whitney map.  The
Whitney map
\[
\e W^k:C^k(\e T)\longrightarrow \Lambda^k(\Omega)
\]
assigns to a simplicial cochain a piecewise polynomial differential form.  If
$\sigma=[x_0,\ldots,x_k]$ is an oriented $k$-simplex and
$\lambda_i$ denotes the barycentric coordinate associated with $x_i$, the
Whitney form associated with $\sigma$ is
\[
\phi_\sigma
	=
	k!
	\sum_{i=0}^k
	(-1)^i
	\lambda_i\,
	d\lambda_0\wedge\cdots
	\wedge \widehat{d\lambda_i}
	\wedge\cdots
	\wedge d\lambda_k.
\]
Extending linearly gives $\e W^k$.  With the standard orientation conventions,
Whitney forms satisfy
\[	d\,\e W^k c
	=
	\e W^{k+1}\delta^k c,
	\qquad
	c\in C^k(\e T),
\]
so $\e W$ is also a cochain morphism.  Moreover,
\[
\e R^k \e W^k = I
	\quad\text{on } C^k(\e T).
\]
This identity says that integrating a Whitney form over the simplices recovers
the original cochain degrees of freedom.

Finite element exterior calculus may be viewed as placing finite element spaces
between these two complexes.  The Whitney forms give the lowest-order finite
element de Rham complex.  Higher-order FEEC spaces enlarge the local polynomial
spaces but preserve the two structural features above: they form subcomplexes of
the de Rham complex, and there exist bounded cochain projections onto them.
These commuting projections are the mechanism by which the topology and Hodge
decomposition of the continuous problem are transferred to the discrete problem;
see \cite{arnold2010finite}.

Boundary conditions are naturally expressed in the same language.  Let
$\Gamma\subseteq \partial\Omega$ be a union of boundary faces and let
$\e T_\Gamma$ be the induced subcomplex.  The relative simplicial cochain
complex may be represented as
\[
C^k(\e T,\e T_\Gamma)
	=
	\{c\in C^k(\e T): c(\sigma)=0
	\text{ for every }\sigma\in\Delta_k(\e T_\Gamma)\}.
\]
Equivalently, this is the quotient of all cochains by those supported on
$\e T_\Gamma$.  On the de Rham side, the corresponding relative complex is
obtained by imposing vanishing trace on $\Gamma$:
\[	\operatorname{tr}_{\Gamma}\omega=0.
\]
The de Rham map and Whitney map restrict to morphisms between the relative
de Rham complex and the relative simplicial cochain complex.  The cohomology is
then a relative cohomology, $\f H^k(\Omega,\Gamma)$, rather than the absolute cohomology $\f H^k(\Omega)$.

This distinction is important for PDE boundary conditions.  Essential boundary
conditions are imposed by choosing the domain complex.  In the de Rham complex,
a relative boundary condition on $\Gamma$ means that the trace of the
$k$-form vanishes on $\Gamma$.  In three-dimensional vector notation this gives
\begin{align*}
	k=0:\quad &u=0 \quad \text{on }\Gamma,\\
	k=1:\quad &n\times u=0 \quad \text{on }\Gamma,\\
	k=2:\quad &u\cdot n=0 \quad \text{on }\Gamma.
\end{align*}
Natural boundary conditions are not imposed in the trial space; rather, they
arise from the adjoint complex through integration by parts.  Thus the choice
between absolute and relative complexes is not merely a topological convention:
it determines the admissible traces, the adjoint operators, the harmonic space,
and hence the nullspace of the Hodge Laplacian.

Consequently, when the mixed Hodge--Laplace problem contains a nontrivial harmonic component $p\in \f H^k$, the space $\f H^k$ is the harmonic space associated with the chosen boundary
complex.  With no essential trace imposed, its dimension is the absolute Betti
number, that is $\dim \f H^k = b_k(\Omega)$.  With relative trace imposed on $\Gamma$, its dimension is
the relative Betti number $b_k(\Omega,\Gamma)$.  The same principle holds after
discretization: the discrete harmonic space is computed from the simplicial or
finite element subcomplex satisfying the same boundary conditions.
\subsection{Existence and uniqueness of the forward PDE solution}

The following inequality is an important result crucial to the stability of solutions to the mixed abstract Hodge Laplacian problem, as well as to its numerical approximation:
\begin{theorem}[Abstract Poincar\'e Inequality] If $(V,d)$ is a closed, bounded Hilbert complex, then there exists a constant $c_{P} >0$ such that for all $v \in \f Z^{k\perp}$,
	\begin{gather}
		\|v\|_{V} \leq c_{P} \|d^{k}v\|_{V}.
	\end{gather}
\end{theorem}
In the case that $(V,d)$ is the domain complex associated with a closed Hilbert complex $(W,d)$, $(V,d)$ is again closed, and the additional graph inner product term vanishes: $\|d^{k} v\|_{V} = \|d^{k}v\|$.

We now develop some basic existence and uniqueness results for linear PDE constraints that arise in the optimal control problems considered here.
First, we consider the abstract version of the Hodge Laplacian and the associated abstract linear Hodge Laplace problem, and we review some existence and uniqueness results for this case.
We then consider a class of abstract semilinear problems that will allow us to consider optimal control problems with some types of nonlinear constraints.

The form $B$ is \emph{not} coercive but rather, for a closed Hilbert complex, satisfies an \keyterm{inf-sup condition} \cite{arnold2010finite,Babuska.I1971}:
there exists $\gamma > 0$ (the \keyterm{stability constant}) such that
\begin{gather*}
	\inf_{(\sigma,u,p)\neq 0} \sup_{(\tau,v,q)\neq 0} \frac{B(\sigma,u,p;\tau,v,q)}{\| (\sigma,u,p) \|_{\f X} \|(\tau, v, q)\|_{\f X}} =: \gamma > 0.
\end{gather*}
where we have defined a standard norm on products: $\|(\sigma,u,p)\|_{\f X} := \|\sigma\|_{V} + \|u\|_{V} + \|p\|$.
This is sufficient to guarantee well-posedness \cite{Babuska.I1971}. To summarize:
\begin{theorem}[Arnold, Falk, and Winther~\cite{arnold2010finite}, Theorem 3.1]\label{thm:cont-prob-wellposed}
	The mixed variational problem \eqref{eqn:mixed-hodge-laplacian} on a closed Hilbert complex $(W,d)$ with domain $(V,d)$ is well posed:
	the bilinear form $B$ satisfies the inf-sup condition with constant $\gamma$, so for any $F \in  (\f X^{k})^{*}$,
	there exists a unique solution $(\sigma,u,p)$ to \eqref{eqn:mixed-hodge-laplacian}, i.e., $B(\sigma,u,p;\tau,v,q) = F(\tau,v,q)$ for all $(\tau,v,q)\in\f X^{k}$, and moreover,
	\begin{gather*}
		\|(\sigma,u,p)\|_{\f X} \leq  \gamma^{-1} \|F\|_{\f X^{*}}.
	\end{gather*}
	The \keyterm{stability constant} $\gamma^{-1}$ depends only on the Poincar\'e constant.
\end{theorem}
Note that the general theory (e.g., \cite{Babuska.I1971,Ev98}) guarantees that a unique solution exists for \emph{any} bounded linear functional $F \in (\f X^{k})^{*}$,
which, in this case of product spaces, means that the problem is still well posed when there are other nonzero linear functionals on the right-hand sides of \eqref{eqn:mixed-hodge-laplacian} besides $\aip{f,v}$.

\subsection{Examples and Applications - Stationary Problems}

\paragraph{Hodge Laplace Problems}

We present two concrete examples in this setting.
To begin with, consider magnetotelluric modeling, for instance~\cite{CASTILLOREYES2022105030}. The system presents the prime model operator as the curl curl of an electric field, $\nabla\times\nabla\times E$, after simplification and substitution into Maxwell's equations. If we neglect the additional zero order term appearing in the literature of the operator of the form $\kappa E$, where $\kappa$ includes some physical constants, while considering the incorporation of externally controlled current injection defined by $z$, see, e.g.~\cite{li2020improved}, we obtain the strong form system,
\[
	\nabla\times\nabla\times E = f+z
\]
Now we can apply the reformulation of this equation into mixed form to obtain the canonical curl curl system. This system corresponds to the $\begin{CD}H(\curl;\Omega)
		@>{\curl}>>
		H(\divergence;\Omega)\end{CD}$ portion of the de Rham Hilbert complex~\eqref{eq:rnderhamv}, taken in mixed Hodge-Laplacian form.

Find $y\in H(\curl;\Omega)$, $\sigma\in H^1(\Omega)$ and $p\in\mathfrak H^1$ such that:
\begin{align}\label{eq:curlcurl}
	\begin{array}{rll}
		\langle \sigma, \tau \rangle - \langle y, \grad \tau \rangle
		 & = 0
		 & \forall \tau \in H^{1}
		\\
		\langle \grad\sigma,v \rangle + \langle \curl y, \curl v \rangle + \langle p,v \rangle
		 & = \langle f+z,v \rangle
		 & \forall v \in H(\curl)
		\\
		\langle y, q \rangle
		 & = 0
		 & \forall q \in \mathfrak H^1
	\end{array}
\end{align}
Recall that, in this case,
\[\f H^1:=\{p\in H(\curl;\Omega)\cap H_0(\divergence;\Omega):
		\curl p=0,\ \divergence p=0\}
        \]

An alternative mixed formulation of the same curl-curl system is given as finding $y\in H(\curl)$, $\sigma\in H^1$ and $p\in\f H^1$ satisfying:
\begin{align}
	\begin{array}{rll}
		\langle \curl y, \curl v \rangle
		+
		\langle \grad \sigma, v \rangle
		+
		\langle p,v \rangle
		 & = \langle f+z,v \rangle
		 & \forall v \in H(\curl)
		\\
		\langle y, \grad \tau \rangle
		 & = 0
		 & \forall \tau \in H^{1}
		\\
		\langle y, q \rangle
		 & = 0
		 & \forall q \in \mathfrak H^1
	\end{array}
\end{align}
It can be shown that the boundary conditions,
\[
	y\cdot n = 0,\quad (\curl y)\times n = 0,\,\text{on }\partial \Omega
\]
are natural in this formulation.

At level $k=2$, the corresponding de Rham Hodge Laplacian mixed formulation is to find $y\in H(\divergence)$, $\sigma\in H(\curl)$ and $p\in\f H^2$ satisfying:
\begin{align}
	\begin{array}{rll}
		\langle \sigma, \tau \rangle - \langle y, \curl \tau \rangle
		 & = 0
		 & \forall \tau \in H(\curl)
		\\
		\langle \curl\sigma,v \rangle + \langle \divergence y, \divergence v \rangle + \langle p,v \rangle
		 & = \langle z,v \rangle
		 & \forall v \in H(\divergence)
		\\
		\langle y, q \rangle
		 & = 0
		 & \forall q \in \mathfrak H^2
	\end{array}
\end{align}
The associated boundary conditions are,
\[
y\times n = 0,\,\quad \divergence y = 0,\,\quad \partial\Omega
\]

\paragraph{Other First Order Linear Systems}
We shall now present two other examples whose structure can be expressed in a Hodge complex, but not the de Rham complex.
First, the elasticity complex is associated to the structure:
\begin{gather*}
	0\to H^1(\Omega) \overset{J}{\to} H(\divergence;\Omega;\mathbb{S})\overset{\divergence}{\to} L^2(\Omega;\mathbb{R}^2)\to 0
\end{gather*}
where the second order differential operator $J$
is defined as $Jq=O(\grad\grad q)O^T$ with $O=\begin{pmatrix} 0 & 1 \\ -1 & 0 \end{pmatrix}$ and
$\sigma \in\mathbb{S}$ indicates the space of symmetric second-order tensors.

The elasticity problem can be written as,
\begin{align*}
	\begin{array}{rll}
		\langle A \sigma, \tau \rangle +\langle \divergence \tau,u\rangle
		 & = 0
		 & \forall \tau \in H(\divergence,\Omega;\mathbb{S})
		\\
		\langle \divergence \sigma,v \rangle
		 & = \langle z,v \rangle
		 & \forall v \in L^2(\divergence)
	\end{array}
\end{align*}
See Section 7 in~\cite{arnold2010finite}.

Consider also the Hodge-Dirac system, which has only one state variable. This problem appears in the study of Clifford algebras, quaternions, and applications in Quantum optics~\cite{leopardi2016abstract}.
\begin{align}
	\begin{array}{rll}
		\langle du,v \rangle + \langle u, d v \rangle + \langle p,v \rangle
		 & = \langle z,v \rangle
		 & \forall v \in V
		\\
		\langle u, q \rangle
		 & = 0
		 & \forall q \in \mathfrak H
	\end{array}
\end{align}

\subsection{Examples and Applications - Time-Dependent Problems}

\paragraph{Heat Equation}

We can generically consider the presence of an additional first order time dependent term to an otherwise elliptic system turning the equations to become parabolic.
Depending on the engineering context, control can be distributed, as a source term for the entire domain, or applied to the boundary.

To use a canonical textbook example, we consider the Heat equation. The Hodge heat equation is the straightforward parabolic extension of the Hodge Laplacian,
\[
y_t+(\delta d+d \delta)y = f\text{ on }\Omega,\qquad y(0)=y_0,\qquad \mathop{tr}\star dy=0\text{ on }\partial\Omega
\]

The Heat Equation in Hodge Laplace mixed weak form introduces $\sigma=\delta u$ and is defined as seeking $(\sigma,u,p):[0,T]\to H\Lambda^{k-1}\times H\Lambda^k\times \f H^k$ that satisfies: 
\begin{align}
	\begin{array}{rll}
		\langle \sigma, \tau \rangle - \langle d\tau, y \rangle
		 & = 0
		 & \forall \tau \in H\Lambda^{k-1}(\Omega),\, t\in[0,T]
		\\
		\langle y_t , v \rangle + \langle d\sigma, v \rangle+\langle dy,dv\rangle
		 & = \langle f+z, v\rangle
		 & \forall v \in H\Lambda^k(\Omega),\,t\in[0,T]
		\\
		\langle y, q \rangle
		 & = 0
		 & \forall q \in \mathfrak H^k
	\end{array}
\end{align}
The distinction from the semi-elliptic Hodge Laplacian is that these equations must now hold for all time $t$.

\paragraph{Maxwell complex and weighted harmonic spaces}

Let
\[
  V^0 \xrightarrow{d^0} V^1 \xrightarrow{d^1} V^2
  \xrightarrow{d^2} V^3
\]
be the absolute de Rham Hilbert complex on a bounded Lipschitz domain
$\Omega\subset\R^3$, with the corresponding adjoint complex carrying
the complementary natural traces. In vector proxies,
$d^0=\grad$, $d^1=\curl$, and $d^2=\divergence$.
Let $\varepsilon$ and $\mu$ be uniformly positive definite material
tensors. We write
\[
  \ip{u}{v}_{\varepsilon}
  :=\int_\Omega \varepsilon u\cdot v\,dx,
  \qquad
  \ip{b}{c}_{\mu^{-1}}
  :=\int_\Omega \mu^{-1}b\cdot c\,dx.
\]

The corresponding weighted harmonic spaces at degrees one and two are
denoted by
\[
  \HH^1_{\varepsilon}\subset V^1,
  \qquad
  \HH^2_{\mu^{-1}}\subset V^2,
\]
with dimensions $b_1$ and $b_2$, respectively. All harmonic spaces,
Betti numbers, period coordinates, and reconstruction maps below are
understood with respect to this same absolute complex.

Choose period maps
\begin{equation}
  \Piop_E:\HH^1_{\varepsilon}\longrightarrow\R^{b_1},
  \qquad
  \Piop_B:\HH^2_{\mu^{-1}}\longrightarrow\R^{b_2},
\end{equation}
that are isomorphisms. With noncontractible loops $\gamma_i$ and
independent two-cycles $\Sigma_j$, the canonical choices are
\begin{equation}
  (\Piop_E h_E)_i=\oint_{\gamma_i} h_E\cdot d\ell,
  \qquad
  (\Piop_B h_B)_j=\int_{\Sigma_j} h_B\cdot n\,dS.
\end{equation}
Thus $\Piop_E$ measures global electric circulation and $\Piop_B$
measures global magnetic flux. Let
\[
  H_E:\R^{b_1}\to\HH^1_{\varepsilon},
  \qquad
  H_B:\R^{b_2}\to\HH^2_{\mu^{-1}}
\]
be the inverse reconstruction maps,
$H_E=\Piop_E^{-1}$ and $H_B=\Piop_B^{-1}$.

\paragraph{Maxwell's Equations without a Harmonic Component}

Maxwell's equations provide a canonical time-dependent example of the
de Rham complex and admit a formulation compatible with the FEEC
framework. In particular, the evolution equations form a first-order
hyperbolic system whose differential operators follow the de Rham
sequence introduced above.

Following the exterior-calculus formulation in
\cite[Section~8.6]{arnold2010finite}, we write the classical Maxwell
system in vector proxies as
\[
\begin{array}{ll}
\varepsilon E_t-\curl(\mu^{-1}B)=-j_z,
&
B_t+\curl E=0,
\\[1mm]
\divergence B=0,
&
\divergence(\varepsilon E)=\rho.
\end{array}
\]
Here $E$ is the electric field, $B$ is the magnetic flux-density
two-form/vector proxy, $j_z$ is the total electric current density, and
$\rho$ is the charge density. The permittivity $\varepsilon$ and
permeability $\mu$ are assumed uniformly positive definite and may vary
spatially.

We write
\[
  j_z=j+\mathcal B z,
\]
where $j$ is a prescribed current density and
$\mathcal B:Z\to(V^1)'$ is the distributed-current control operator.
With this convention, the strong term $\curl(\mu^{-1}B)$ is represented
weakly through the pairing of $\mu^{-1}B$ with $\curl v$. Charge
conservation is expressed by
\[
  \dot\rho+\divergence j_z=0,
\]
with the divergence understood distributionally whenever the current
does not possess sufficient strong regularity. In exterior-calculus
language, the current and charge are represented at degrees two and
three, respectively, while the material coefficients enter through the
weighted inner products associated with $\varepsilon$ and $\mu^{-1}$.

Maxwell's equations traverses three levels of the de Rham complex. To this end, let,
\[
V^0:=H^1(\Omega),\,V^1:=H(\curl;\Omega),\,V^2:=H(\divergence;\Omega),\,V^3:=L^2(\Omega)
\]

With this defining compact notation, a mixed vector-proxy formulation introduces an auxiliary
degree-zero variable $\sigma$ and seeks
\[
  (\sigma,E,B):[0,T]\to V^0\times V^1\times V^2
\]
such that
\begin{align}
\begin{array}{rll}
\langle \sigma_t,\tau\rangle_{\alpha}
-\langle E,\grad\tau\rangle_{\varepsilon}
&=\langle\rho,\tau\rangle,
&\forall\tau\in V^0,\quad t\in[0,T],
\\[1mm]
\langle E_t,v\rangle_{\varepsilon}
+\langle\grad\sigma,v\rangle_{\varepsilon}
-\langle\mu^{-1}B,\curl v\rangle
&=-\langle j_z,v\rangle,
&\forall v\in V^1,\quad t\in[0,T],
\\[1mm]
\langle B_t,w\rangle_{\mu^{-1}}
+\langle\curl E,w\rangle_{\mu^{-1}}
&=0,
&\forall w\in V^2,\quad t\in[0,T],
\\[1mm]
\langle\dot\rho,\eta\rangle
&=-\langle\divergence j_z,\eta\rangle,
&\forall\eta\in V^3,\quad t\in[0,T].
\end{array}
\end{align}
All pairings above are understood as the corresponding weighted
$L^2$ inner products or natural duality pairings.

Since the topology-aware formulation below uses the absolute de Rham
complex fixed above, we do not additionally impose the relative/PEC
essential conditions
\[
E\times n=0,
\qquad
B\cdot n=0.
\]
The boundary terms are instead interpreted through the natural
conditions associated with the corresponding adjoint complex. If a
nonhomogeneous magnetic boundary flux is prescribed, one may introduce
a lifting
\[
  B=B_0+B_{\rm lift},
\]
where $B_0$ belongs to the homogeneous space of the selected complex
and $B_{\rm lift}$ carries the prescribed boundary data.

Note that this corresponds to,
\begin{align*}
\begin{array}{rll}
\langle \sigma_t,\tau\rangle_{\alpha}
-\langle E,d\tau\rangle_{\varepsilon}
&=\langle\rho,\tau\rangle,
&\forall\tau\in V^0,\quad t\in[0,T],
\\[1mm]
\langle E_t,v\rangle_{\varepsilon}
+\langle d\sigma,v\rangle_{\varepsilon}
-\langle B,dv\rangle_{\mu^{-1}}
&=-\langle j_z,v\rangle,
&\forall v\in V^1,\quad t\in[0,T],
\\[1mm]
\langle B_t,w\rangle_{\mu^{-1}}
+\langle dE,w\rangle_{\mu^{-1}}
&=0,
&\forall w\in V^2,\quad t\in[0,T],
\\[1mm]
\langle\dot\rho,\eta\rangle
&=-\langle dj_z,\eta\rangle,
&\forall\eta\in V^3,\quad t\in[0,T].
\end{array}
\end{align*}
Thus the vector-proxy operators $\grad$, $\curl$, and $\divergence$
are represented uniformly by the exterior derivative $d$ at successive
degrees of the complex.

\paragraph{Split physical fields}

We now separate the local Maxwell response from the physical harmonic
components. We decompose
\begin{equation}
  E=e+h_E=e+H_Ec_E,
  \qquad
  B=b+h_B=b+H_Bc_B,
\end{equation}
where
\[
  e(t)\perp_{\varepsilon}\HH^1_{\varepsilon},
  \qquad
  b(t)\perp_{\mu^{-1}}\HH^2_{\mu^{-1}},
\]
and
\[
  c_E(t)=\Piop_Eh_E(t)\in\R^{b_1},
  \qquad
  c_B(t)=\Piop_Bh_B(t)\in\R^{b_2}.
\]
The variables $e$ and $b$ are the PDE-resolved local fields, whereas
$c_E$ and $c_B$ are the physical topological states. In particular,
these coordinates are distinct from harmonic Lagrange multipliers
introduced to regularize a singular mixed operator.

Let $P_E^\perp$ and $P_B^\perp$ denote the weighted orthogonal
projections onto the complements of the corresponding harmonic spaces.
Using the total current $j_z=j+\mathcal Bz$ defined above, the local
Maxwell evolution is posed on the harmonic complements:
\begin{subequations}
\label{eq:maxwell-local-state}
\begin{align}
  \ip{\varepsilon\dot e}{v}
  -\ip{\mu^{-1}b}{dv}
  &=-\ip{P_E^\perp j_z}{v}
  &&\forall
  v\in V^1\cap(\HH^1_{\varepsilon})^{\perp_\varepsilon},
  \\
  \ip{\mu^{-1}\dot b}{w}
  +\ip{\mu^{-1}de}{w}
  &=0
  &&\forall
  w\in V^2\cap(\HH^2_{\mu^{-1}})^{\perp_{\mu^{-1}}}.
\end{align}
\end{subequations}
Here and below, time derivatives are understood in the appropriate
Bochner-space sense. The initial data satisfy
\begin{equation}
  e(0)=e_0,
  \qquad
  b(0)=b_0.
\end{equation}

The harmonic projection of the Amp\`ere--Maxwell equation is not lost
under this decomposition; rather, it becomes a finite-dimensional
evolution equation for $c_E$. To write this equation invariantly,
define the harmonic mass matrices
\begin{equation}
  M_E^H:=H_E^*\varepsilon H_E\in\R^{b_1\times b_1},
  \qquad
  M_B^H:=H_B^*\mu^{-1}H_B\in\R^{b_2\times b_2},
\end{equation}
which are symmetric positive definite. Define also the harmonic-current
coordinate maps
\begin{equation}
  C_Ez:=H_E^*\mathcal Bz,
  \qquad
  j_H(t):=H_E^*j(t).
\end{equation}
Projecting the electric equation onto $\HH^1_{\varepsilon}$ yields the
physical balance law
\begin{equation}
  M_E^H\dot c_E
  =
  -j_H-C_Ez+G_Ea_E+g_E.
  \label{eq:maxwell-top-E-ode}
\end{equation}
The optional actuator $a_E(t)\in\R^{m_E}$ permits direct control of
global electromotive or loop-current modes through
$G_E\in\R^{b_1\times m_E}$, while $g_E$ represents prescribed
topological forcing. If no separate topological electric actuator is
present, one sets
\[
  G_Ea_E+g_E=0,
\]
and the harmonic projection of the distributed current alone drives
$c_E$.

In the absence of magnetic current or boundary injection, the harmonic
projection of Faraday's law gives conservation of the magnetic
cohomology:
\begin{equation}
  M_B^H\dot c_B=0.
  \label{eq:maxwell-magnetic-flux-conservation}
\end{equation}
Consequently, a nonzero cavity flux is conserved by the local Maxwell
evolution. To make this flux an actuated control variable, one must
introduce a physically distinct topological magnetic actuator,
boundary-flux injection, or equivalent source. We represent such
finite-dimensional actuation by
\begin{equation}
  M_B^H\dot c_B=G_Ba_B+g_B,
  \label{eq:maxwell-top-B-ode}
\end{equation}
where $a_B(t)\in\R^{m_B}$. If the physical model contains no such
actuator, \eqref{eq:maxwell-top-B-ode} is omitted and
\eqref{eq:maxwell-magnetic-flux-conservation} is retained.

\begin{remark}[Why topology is genuinely dynamical]
Equations
\eqref{eq:maxwell-top-E-ode}--\eqref{eq:maxwell-top-B-ode}
are not auxiliary constraints introduced merely for numerical
invertibility. They govern physical circulation and flux states. The
dimensions of these dynamical subsystems are $b_1$ and $b_2$, while
their reachable directions are determined by the ranks of $G_E$,
$G_B$, and the harmonic projection of the distributed-current operator
$C_E$.
\end{remark}

\paragraph{Gauss laws and charge conservation}

The Maxwell evolution is supplemented by the Gauss constraints
\begin{equation}
  dB=0,
  \qquad
  d(\varepsilon E)=\rho,
  \label{eq:maxwell-gauss-laws}
\end{equation}
interpreted weakly with the natural boundary conditions of the selected
absolute complex. In vector proxies these relations correspond to
\[
  \divergence B=0,
  \qquad
  \divergence(\varepsilon E)=\rho.
\]
If the initial fields satisfy \eqref{eq:maxwell-gauss-laws} and the
current and charge satisfy the continuity equation
\begin{equation}
  \dot\rho+d j_z=0,
\end{equation}
then the Gauss constraints are propagated by the Maxwell evolution.
A compatible FEEC semidiscretization inherits the complex identity
$d^2=0$ and therefore preserves the corresponding discrete constraint
structure. If a four-field mixed Maxwell realization with an auxiliary
scalar enforcing Gauss' law is used instead of
\eqref{eq:maxwell-local-state}, this does not alter the physical
harmonic states or their finite-dimensional period equations.

\section{Domain Topology and the Hilbert Complex}

\subsection{Betti numbers of a triangulation}
\label{sec:betti}
 
The examples in this section are posed on domains whose
meshes are generated rather than assumed, so the topology of the
computational domain is an output of the construction. Since
$\dim\mathcal{H}^{k}_{h}=b_{k}$ is the fact the harmonic block rests
on, we compute the Betti numbers of the triangulation directly, at a
cost linear in the size of the mesh and with no linear algebra.
 
Writing $n_{V},n_{E},n_{F},n_{T}$ for the numbers of vertices, edges,
faces and tetrahedra, the Euler--Poincar\'e theorem gives
\begin{equation}\label{eq:euler}
  \chi(\mathcal{T})=n_{V}-n_{E}+n_{F}-n_{T}
  =b_{0}-b_{1}+b_{2}-b_{3}.
\end{equation}
The identity follows from the rank–nullity theorem on the chain complex of
Section~\ref{sec:simplicial}. With
$r_{k}=\operatorname{rank}\partial_{k}$ and
$z_{k}=\dim\ker\partial_{k}$ one has $n_{k}=r_{k}+z_{k}$ and
$b_{k}=z_{k}-r_{k+1}$, so the alternating sums of the $n_{k}$ and of
the $b_{k}$ differ by $\sum_{k}(-1)^{k}(r_{k}+r_{k+1})$, which
telescopes to zero because $\partial_{0}$ and $\partial_{n+1}$ vanish.
No manifold hypothesis is used here.
 
Equation \eqref{eq:euler} is one relation among four unknowns. For a
compact orientable three-manifold with nonempty boundary $b_{3}=0$,
since no closed three-cycle carries a fundamental class. And for a
compact three-manifold $\Omega\subset\mathbb{R}^{3}$, Alexander duality
gives
\begin{equation}\label{eq:alexander}
  b_{2}=n_{\partial}-b_{0},
\end{equation}
with $n_{\partial}$ the number of connected components of
$\partial\Omega$. To see \eqref{eq:alexander}, recall that
$\tilde H_{0}(S^{3}\setminus\Omega)\cong\tilde H^{2}(\Omega)$ and form
the bipartite graph whose vertices are the components of $\Omega$ and
of $S^{3}\setminus\Omega$, with one edge per component of
$\partial\Omega$. Every closed orientable surface in $S^{3}$ separates,
because $H_{2}(S^{3})=0$, so no edge lies on a cycle; a connected
acyclic graph is a tree, whence $n_{\partial}=b_{0}+c-1$ with $c$ the
number of complement components, and $b_{2}=c-1$. The formula returns
zero for a ball and for a solid torus, and one for a spherical shell.
 
Both $b_{0}$ and $n_{\partial}$ are connected-component counts, the
first on the graph of tetrahedra sharing a face and the second on the
graph of boundary facets glued along shared edges, so
\[
  b_{1}=b_{0}+b_{2}-\chi(\mathcal{T}).
\]
 
Since \eqref{eq:alexander} presupposes a manifold, the computation is
accompanied by a manifold test: every facet must lie in at most two
tetrahedra, and the star of every vertex and of every edge must be
connected. The second condition is not implied by the first. Two
tetrahedra meeting at a single vertex satisfy the facet condition and
do not form a manifold, and the same configuration is fatal for the
discretization, since N\'ed\'elec degrees of freedom live on edges and
a vertex-connected mesh decouples the discrete
$H(\operatorname{curl})$ spaces of the two pieces.
\subsection{Standing Domain Examples}

 \paragraph{L Shape}\label{ex:lshape}
Let $\Omega\subset\mathbb{R}^{3}$ be the L-shaped domain obtained by
removing the cube $[0,\tfrac12]^{3}$ from the unit cube $[0,1]^{3}$;
see Figure~\ref{fig:n1}. The domain is contractible, so
$b_{0}=1$ and $b_{1}=b_{2}=b_{3}=0$. The harmonic space
$\mathcal{H}^{1}$ is trivial.

\begin{figure}[h!]
  \centering
  \includegraphics[width=0.92\linewidth]{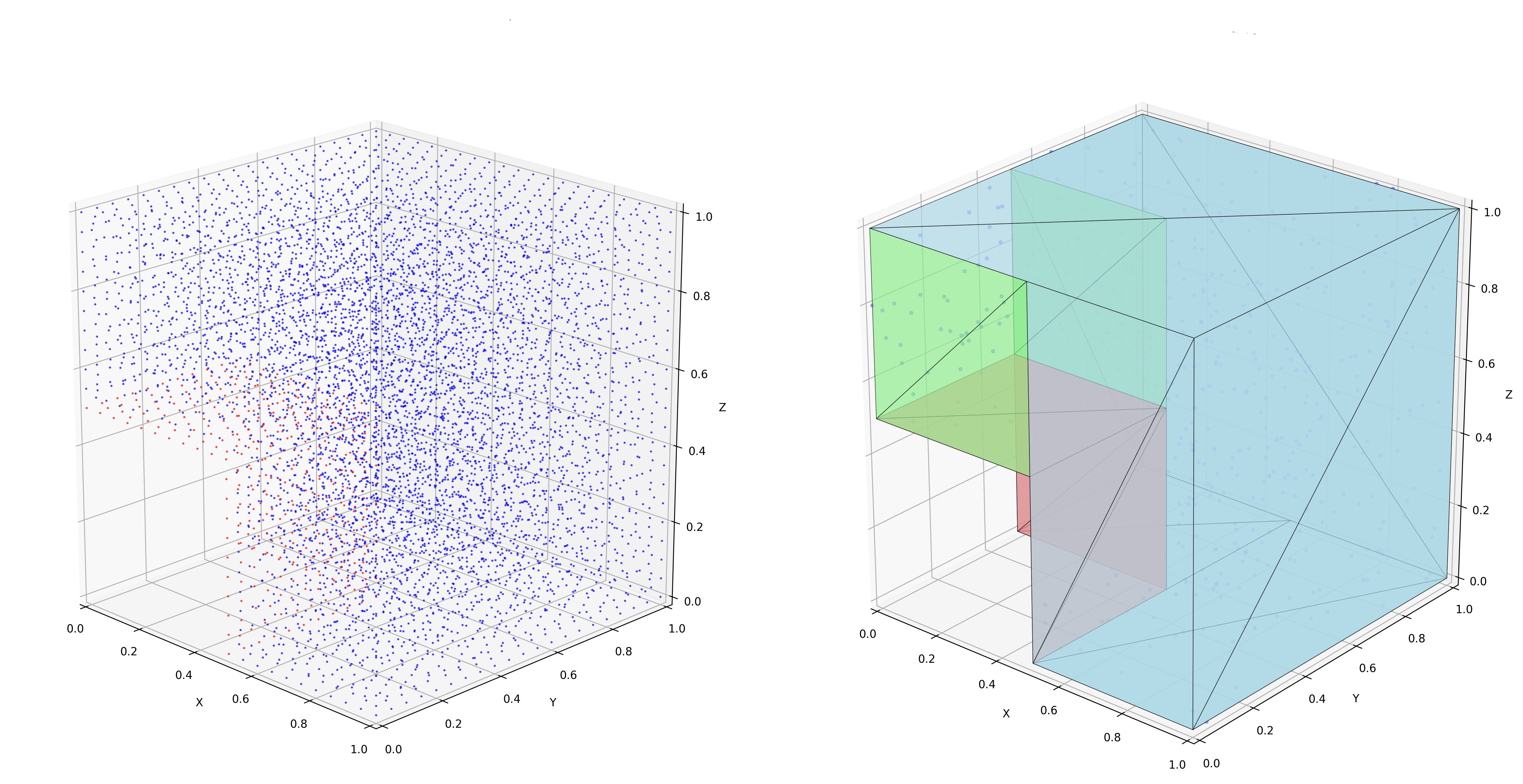}
  \caption{The L-shaped domain and generated mesh.}
  \label{fig:n1}
\end{figure}

\paragraph{Torus}\label{ex:torus}

Let $\Omega$ be the solid torus obtained by rotating a disk of radius
$\rho=0.15$, centered at distance $R=0.30$ from the vertical axis
through $(0.5,0.5)$, about that axis. Its Betti numbers are
$(b_{0},b_{1},b_{2},b_{3})=(1,1,0,0)$, so $\dim\mathcal{H}^{1}=1$ and
there is a single harmonic $1$-form circulating around the central
hole. See Figure~\ref{fig:torus}
 
 \begin{figure}[ht]
\centering
\includegraphics[width=.62\textwidth]{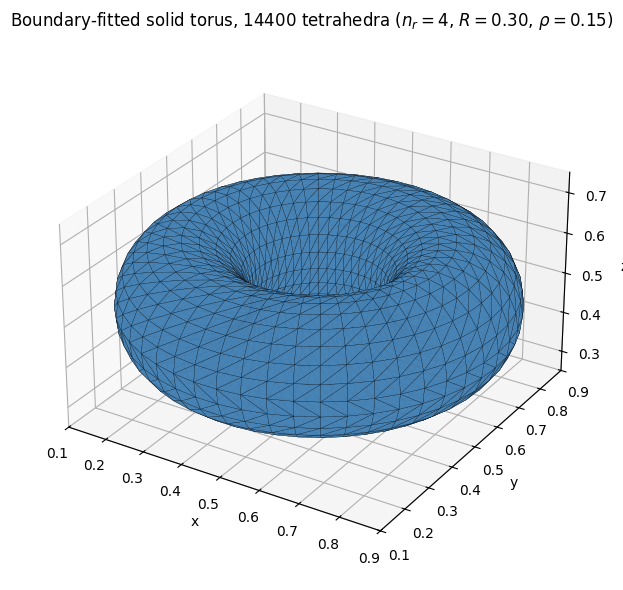}
\caption{Solid torus mesh. The single noncontractible loop gives one harmonic circulation mode and one scalar period coordinate.}\label{fig:torus}
\end{figure}

\paragraph{Double Torus}\label{ex:dtorus}

Let
\begin{equation}\label{eq:fig8-domain}
  \Omega=\bigl([0,1]^{2}\setminus(S_{1}\cup S_{2})\bigr)\times[0,\tfrac14],
  \qquad
  S_{i}=[c_{i}-\tfrac18,c_{i}+\tfrac18]\times[\tfrac38,\tfrac58],
\end{equation}
with $c_{1}=\tfrac14$ and $c_{2}=\tfrac34$: a slab with two square holes
punched through it. The domain retracts to a wedge of two circles, so
\[
  b_{0}=1,\qquad b_{1}=2,\qquad b_{2}=0,\qquad b_{3}=0,
\]
and $\dim\mathcal{H}^{1}=2$. Any connected genus-two solid, such as a
union of two overlapping solid tori, has the same homotopy type and the
same harmonic dimension. We use \eqref{eq:fig8-domain} because every
corner is a multiple of $\tfrac18$, so for $N$ a multiple of eight the
domain is exactly a union of cells of the uniform background
triangulation of the slab and the discrete domain equals $\Omega$ at
every level of refinement. The boundary of a union of interpenetrating
tori contains a quartic intersection curve that no grid represents
exactly, and carving it produces the volume oscillation documented for the torus in Section~\ref{ex:torus}. See Figure \ref{fig:doubletorus}.

\begin{figure}[ht]
\centering
\includegraphics[width=.7\textwidth]{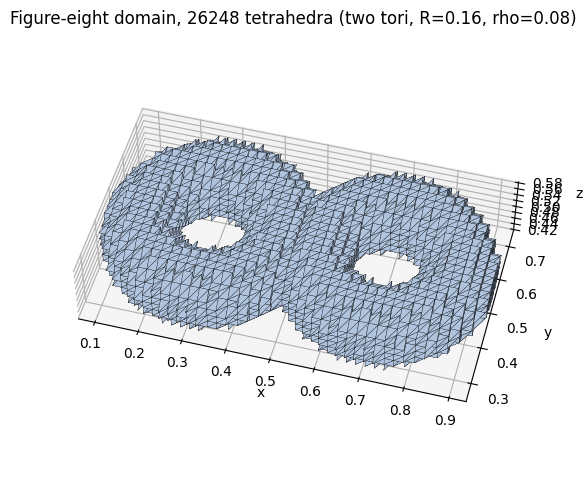}
\caption{Figure-eight domain. The two independent tunnels generate a two-dimensional period vector.}\label{fig:doubletorus}
\end{figure}

\paragraph{Ball with Cavity}\label{ex:Ball}

Let
\begin{equation}
  \Omega
  =
  \left\{
  x\in\R^3:
  r_{\rm in}<|x-c|<r_{\rm out}
  \right\},
  \qquad
  c=(1/2,1/2,1/2),
  \label{eq:shell-domain}
\end{equation}
with $0<r_{\rm in}<r_{\rm out}$.  Thus $\Omega$ is a spherical shell,
or equivalently a ball with a spherical cavity removed.  It is homotopy
equivalent to $S^2$, and hence
\[
  b_0=1,\qquad b_1=0,\qquad b_2=1,\qquad b_3=0.
\]
Consequently, the degree-two harmonic space is one dimensional,
\[
  \dim\HH^2=1.
\]
The relevant global observable is not circulation around a tunnel, as
in the torus, but flux through a closed surface surrounding the cavity.  The mesh is boundary-fitted: an
  icosahedron subdivided four times and projected to the sphere, then
  extruded radially through five layers, each prism split into three
  tetrahedra. Every boundary vertex lies exactly on one of the two
  spheres, so the volume deficit is one-signed and second order,
  $0.216\%$ at this resolution. The cavity wall, in dark red, is the
  boundary component that produces $b_{2}=1$,  see Figure \ref{fig:shell}.
\begin{figure}[htbp]
  \centering
  \includegraphics[width=0.65\textwidth]{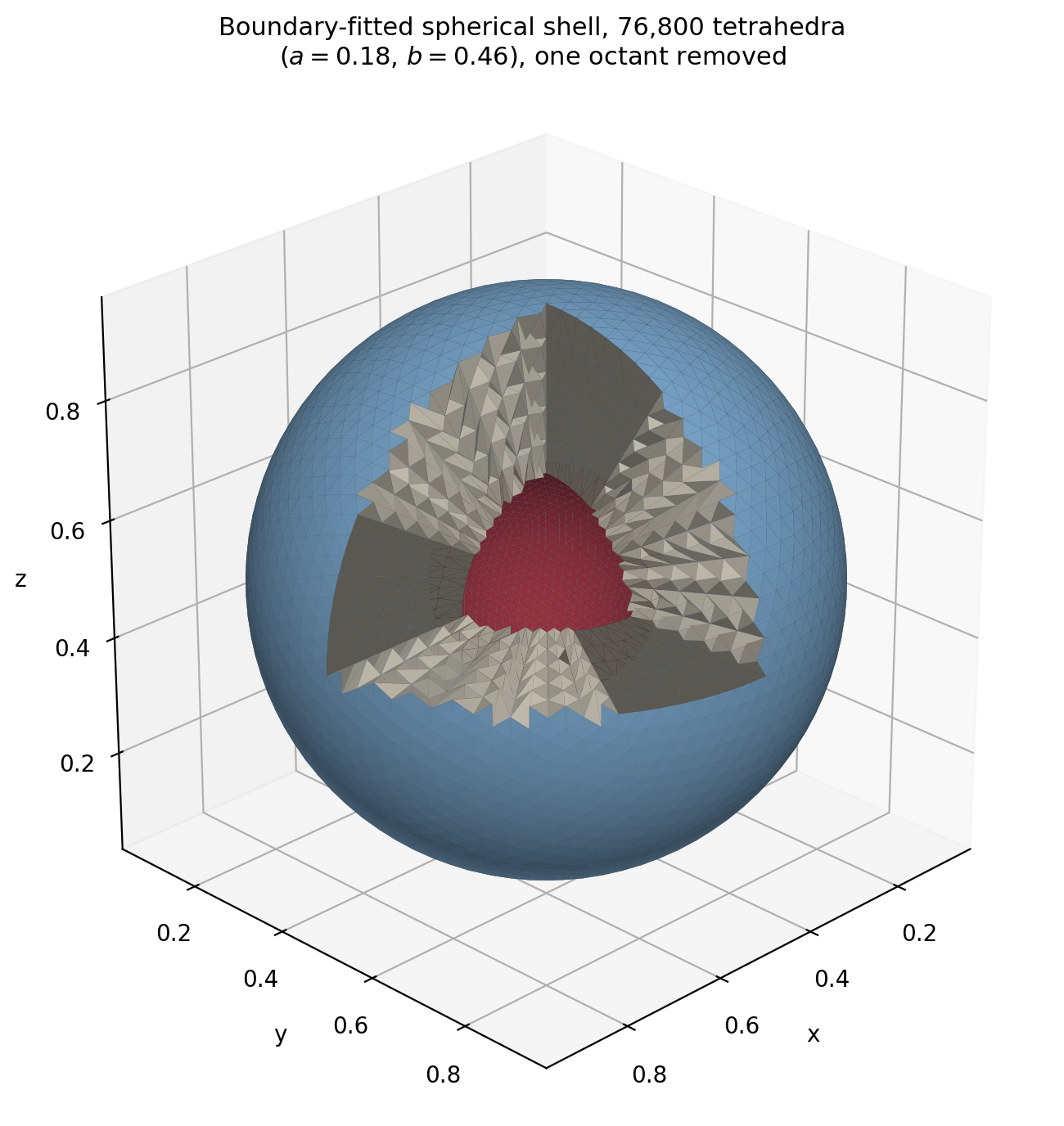}
  \caption{The spherical shell \eqref{eq:shell-domain} with one octant   removed to expose the interior.}
  \label{fig:shell}
\end{figure}

\subsection{Period and flux maps}\label{Pmaps}
Let $b_k=\dim\HH^k<\infty$. Choose a basis of homology cycles dual to $\HH^k$. The resulting period map is a linear isomorphism
\begin{equation}\label{eq:period-map}
 \Piop:\HH^k\to\R^{b_k}.
\end{equation}
For the de Rham complex in three dimensions:
\begin{itemize}
\item at degree $k=1$, if $\gamma_1,\dots,\gamma_{b_1}$ are noncontractible loops,
\[
 (\Piop h)_i=\oint_{\gamma_i} h\cdot d\ell;
\]
\item at degree $k=2$, if $\Sigma_1,\dots,\Sigma_{b_2}$ represent independent two-cycles,
\[
 (\Piop h)_i=\int_{\Sigma_i} h\cdot n\,dS.
\]
\end{itemize}
The map may be normalized by selecting harmonic basis functions $h_1,\dots,h_{b_k}$ satisfying $\Piop h_i=e_i$. Then
\begin{equation}
 h=\sum_{i=1}^{b_k}c_i h_i,
 \qquad c=\Piop h.
\end{equation}

\begin{remark}
The period map is defined only on the 
finite-dimensional harmonic space in the abstract theory. This avoids trace regularity issues that may arise if line or surface integrals are applied to arbitrary elements of $V^k$. In applications, the same functional may extend to a larger space, but that extension is not needed for the analysis below.
\end{remark}

The definition \eqref{eq:period-map} uses integration over cycles. That is
inconvenient in two ways. A representative of each homology class must be
exhibited, which on a generated mesh is awkward, and the traces of a general
element of $V^k$ onto a curve or a surface need not be defined. Both
difficulties disappear once the period is written as an inner product. The
following says that any coclosed field induces a functional on cohomology, and
that only its harmonic part matters.

\begin{theorem}[Period functionals]\label{thm:period-pairing}
Let $(W,d)$ be a closed Hilbert complex and let $J\in\f Z^{*}_{k}=\ker d^{*}_{k}$.
Then
\begin{equation}\label{eq:period-pairing}
 \Piop_J(v):=\ip{v}{J}_{W^k}
\end{equation}
vanishes on $\f B^{k}$ and therefore descends to a linear functional on
$\f Z^{k}/\f B^{k}$. For every $v\in\f Z^{k}$,
\begin{equation}\label{eq:period-both-harmonic}
 \Piop_J(v)=\ip{P_{\HH}v}{P_{\HH}J},
\end{equation}
so $\Piop_J$ depends on $J$ only through $P_{\HH}J$. Conversely every linear
functional on $\HH^{k}$ has the form \eqref{eq:period-pairing}. Given
$J_1,\dots,J_{b_k}\in\f Z^{*}_{k}$, the induced map
$\Piop=(\Piop_{J_1},\dots,\Piop_{J_{b_k}}):\HH^{k}\to\R^{b_k}$ is an
isomorphism if and only if $P_{\HH}J_1,\dots,P_{\HH}J_{b_k}$ are linearly
independent.
\end{theorem}

\begin{proof}
Since the complex is closed, $\f B^{k\perp_W}=\f Z^{*}_{k}$, so
$\ip{d\phi}{J}=0$ for every $\phi\in V^{k-1}$ and $\Piop_J$ annihilates
$\f B^{k}$. If $v\in\f Z^{k}$ then $v=P_{\HH}v+b$ with $b\in\f B^{k}$ by
Theorem \ref{thm:hodge-decomposition-expanded}, and the first assertion gives
$\ip{v}{J}=\ip{P_{\HH}v}{J}$. Writing $J=P_{\HH}J+(I-P_{\HH})J$ and using that
the second summand is orthogonal to $\HH^{k}$ gives
\eqref{eq:period-both-harmonic}. For the converse, a linear functional on the
finite-dimensional space $\HH^{k}$ is represented by some $w\in\HH^{k}$, and
$\HH^{k}\subseteq\f Z^{*}_{k}$. The criterion for $\Piop$ follows from
\eqref{eq:period-both-harmonic}, since $\Piop$ is injective on $\HH^{k}$
precisely when the $P_{\HH}J_i$ span it.
\end{proof}

\begin{remark}
Formula \eqref{eq:period-pairing} is defined on all of $W^{k}$ and requires no
trace of $v$ on a lower-dimensional set. The restriction to $\HH^{k}$ in
\eqref{eq:period-map} is therefore a normalization convention and not a
regularity constraint.
\end{remark}
In three dimensions the adjoint complex \eqref{eq:adjoint-vector-derham} identifies
$\f Z^{*}_{k}$ with closed fields of complementary degree carrying a vanishing
trace. The condition $J\in\f Z^{*}_{k}$ therefore splits into one differential
and one boundary condition, and it is this splitting that makes the two cases
below look different at first sight.

\begin{corollary}[de Rham realization]\label{cor:period-derham}
Let $\Omega\subset\R^{3}$ be a bounded Lipschitz domain and let the absolute de
Rham complex \eqref{eq:rnderhamv} carry its adjoint complex
\eqref{eq:adjoint-vector-derham}.
\begin{itemize}
\item At $k=1$, every $J\in\mathring H(\divergence;\Omega)$ with
$\divergence J=0$ lies in $\f Z^{*}_{1}$, so $v\mapsto\ip{v}{J}$ annihilates
$\grad H^{1}(\Omega)$.
\item At $k=2$, every $J\in\mathring H(\curl;\Omega)$ satisfying $\curl J=0$ lies in $\f Z^{*}_{2}$. Hence the functional $v\mapsto\ip{v}{J}$ annihilates the exact space $\curl H(\curl;\Omega)$.
\end{itemize}
In both cases $\ip{\cdot}{J}$ is a period functional in the sense of
Theorem~\ref{thm:period-pairing}.
\end{corollary}

\begin{proof}
At $k=1$ the operator $d^{*}_{1}$ is $-\divergence$ with domain
$\mathring H(\divergence;\Omega)$, and at $k=2$ the operator $d^{*}_{2}$ is
$\curl$ with domain $\mathring H(\curl;\Omega)$. In each case the stated
hypotheses say exactly that $J$ lies in the domain and in the kernel.
\end{proof}

\begin{example}[Solid torus]\label{ex:period-torus}
Let $\Omega$ be the solid torus of Section~\ref{ex:torus} and let
$J=\hat e_{\varphi}/(\pi\rho^{2})$. Then $\divergence J=0$ and $J\cdot n=0$ on
$\partial\Omega$, so Corollary~\ref{cor:period-derham} applies at $k=1$. Since
$b_1=1$, the functional $\ip{\cdot}{J}$ and the circulation
$\oint_{\gamma}\cdot\,d\ell$ are two functionals on a one-dimensional space and
are therefore proportional. Evaluating both on the exact generator
$h_1=\hat e_{\varphi}/(2\pi r)$, in torus coordinates $r=R+s\cos\phi$ with
$dV=r\,s\,ds\,d\phi\,d\theta$,
\[
 \ip{h_1}{J}=\frac{1}{2\pi^{2}\rho^{2}}\int_{\Omega}\frac{dV}{r}
 =\frac{1}{2\pi^{2}\rho^{2}}\cdot 2\pi^{2}\rho^{2}=1
 =\oint_{\gamma}h_1\cdot d\ell ,
\]
so the two agree on $\HH^{1}$. No decomposition of $\Omega$ into field lines is
needed.
\end{example}

\begin{example}[Extruded domain]\label{ex:period-extruded}
Let $D\subset\R^{2}$ be bounded Lipschitz with boundary components
$\Gamma_0,\dots,\Gamma_b$, let $\Omega=D\times[0,H]$, and let $a\in H^{1}(D)$
take the constant value $a_i$ on $\Gamma_i$. Put $J=\curl(a\hat e_z)
=\grad a\times\hat e_z$. Then $\divergence J=0$, and $J\cdot n=0$ on
$\partial\Omega$: on the faces $z=0$ and $z=H$ because $J$ is horizontal, and
on the lateral boundary because constancy of $a$ on each $\Gamma_i$ makes
$\grad a$ parallel to the planar normal, so that $J$ is tangential. Thus
$J\in\f Z^{*}_{1}$ and Corollary~\ref{cor:period-derham} applies. Integrating
$\nabla\cdot(v\times a\hat e_z)=a\hat e_z\cdot\curl v-v\cdot J$ over $\Omega$
for closed $v$, and orienting the loops by $t=\hat e_z\times n$,
\begin{equation}
 \frac1H\ip{v}{J}=-\sum_i a_i\oint_{\Gamma_i}v\cdot d\ell .
\end{equation}
Harmonicity of $a$ is not used. It serves the different purpose of placing $J$
itself in $\HH^{1}$, since $\curl J=-\Delta a\,\hat e_z$.
\end{example}

\begin{example}[Multiply bounded domain]\label{ex:period-cavity}
Let $\partial\Omega$ have components $\Gamma_0,\dots,\Gamma_b$ and let
$\psi\in H^{1}(\Omega)$ take the constant value $\psi_i$ on $\Gamma_i$. Then
$n\times\grad\psi=0$, so $J=\grad\psi\in\mathring H(\curl;\Omega)$ with
$\curl J=0$, and Corollary~\ref{cor:period-derham} applies at $k=2$.
Integration by parts against a divergence-free $v$ gives
\begin{equation}
 \ip{v}{\grad\psi}=\sum_i\psi_i\int_{\Gamma_i}v\cdot n\,dS .
\end{equation}
Taking $\psi$ equal to one on a cavity wall and zero on the remaining
components leaves a single term, the flux through that cavity.
\end{example}

\paragraph{Distributed and topological controls}

Let $z\in Z$ be a distributed control and $B:Z\to W^k$ a bounded control operator. Let $f\in W^k$ be a fixed forcing. The local mixed state solves the mixed Hodge Laplace problem with
\begin{equation}
 g=f+Bz.
\end{equation}
Separately, let $a\in\R^m$ be a topological actuator and impose
\begin{equation}
 \Piop h=Ga+c_0,
\end{equation}
where $G\in\R^{b_k\times m}$ and $c_0\in\R^{b_k}$. Since $\Piop$ is an isomorphism, this equation uniquely determines
\begin{equation}
 h(a)=\Piop^{-1}(Ga+c_0).
\end{equation}
The reachable topological period vectors form the affine space $c_0+\range(G)$. Full topological controllability holds precisely when $G$ has row rank $b_k$.

\begin{example}[Torus circulation]
For a solid torus and $k=1$, $b_1=1$. The scalar $\Piop h$ is the circulation around a generator of the central hole. A scalar actuator may represent an imposed electromotive force, loop voltage, or linked coil current, and $G$ reduces to a nonzero scalar gain.
\end{example}

\begin{example}[Cavity flux]
For a spherical shell and $k=2$, $b_2=1$. The scalar $\Piop h$ is the flux through the cavity. The topological actuator can represent a coil or imposed flux source that changes this global mode without being described by a local Hodge--Laplace forcing alone.
\end{example}

\section{Optimal Control Problem - Stationary Hodge Laplacian}

We would like to study an optimal control problem associated with the Hodge-Laplace equation.
There are different formulations, but they are not necessarily amenable for a Galerkin discretization.
In this section, we study the corresponding optimal control problem using the mixed Hodge--Laplace formulation.

The gauge-fixed variable $u$ in~\eqref{eqn:mixed-hodge-laplacian} does not contain a harmonic component.  We therefore use throughout the physical decomposition already introduced in \eqref{eq:intro-split}: the local PDE determines $(\sigma,u,p)$, while global constitutive, boundary, or actuator relations determine $h$.  This separation is essential because adding a target directly to $p$ would track the harmonic projection of the forcing rather than a physical global state.

We review this in several steps in order to motivate our saddle-point formulation. Let us temporarily assume that the space of harmonic forms is trivial and that no constraints are put on the control variable. Most naively, we could ask for the minimizer of
\begin{gather} 
	\min_{ \substack{ u \in \operatorname{dom}(\Delta) \\ z \in W^k } }
	\frac{1}{2} \| O u - y_d \|^{2} + \frac \alpha 2 \| z \|^{2}
	\quad
	\text{ subject to }
	\quad
	\Delta u = z.
\end{gather}
We easily see that this formulation is not amenable to a discretization. For example, if the Hodge-Laplacian happens to encode the Poisson problem, then this would require $u \in H^2$, which is not suitable for a simple finite element discretization.
For that reason, we could attempt a weak elliptic formulation such as
\begin{align*}
\min_{\substack{u \in V^k \cap V_k^\ast \\ z \in W^k}}
\quad &
\frac{1}{2}\|Ou-y_d\|^2+\frac{\alpha}{2}\|z\|^2,
\\
\text{s.t.}\quad &
\langle du,dv\rangle+\langle d^\ast u,d^\ast v\rangle
=\langle z,v\rangle
\quad
\forall v\in V^k\cap V_k^\ast .
\end{align*}
That approach is successful for problems such as Poisson, where the weak formulation has a natural Galerkin discretization.
However, it does not work for the general Hodge-Laplace equation
because we generally do not have a convenient conforming discretization of the space $V^k \cap V^\ast_k$. This is the same difficulty encountered in the discretization of the unconstrained Hodge--Laplace equation.



We study an optimal control problem based on the mixed formulation of the Hodge--Laplace equation. To remain faithful to physical engineering practice, we allow a secondary observation $R\sigma$ of the mixed auxiliary state. The compatibility multiplier $p\in\HH^k$ is not assigned a tracking target: it balances the harmonic component of the forcing and is generally not a physical state. By contrast, the separately reconstructed harmonic field $h\in\HH^k$ represents physically observable global circulation or flux and may therefore enter both the physical-state observation $O(u+h)$ and the period target $\Piop h$. Throughout the manuscript we assume that $O$ and $R$ are bounded linear observation operators. 

The purpose of the topology-dependent extension is to ensure that cohomology is not merely carried passively by the forward solver.  The standard mixed formulation fixes the local representative $u\perp\HH^k$ and introduces $p\in\HH^k$ only to balance the harmonic part of the forcing.  The physical harmonic field $h$, by contrast, carries the global circulation or flux modes that the local Hodge--Laplace operator cannot select.  We therefore use the period coordinates \eqref{eq:intro-period-map} and the actuator relation \eqref{eq:intro-top-control}.  At degree $k=1$ these coordinates are circulations along noncontractible loops; at degree $k=2$ they are fluxes through nontrivial surfaces.  The matrix $G$ specifies which of the $b_k=\dim\HH^k$ global modes can be actuated, and $c_0$ is the imposed background period vector.

The distributed control $z$ acts on the local PDE, while $a$ acts on the global harmonic modes. We consider objectives of the form
\begin{equation}\label{eq:intro-objective}
\begin{aligned}
J(\sigma,u,h,z,a)
={}&\frac{w_y}{2}\norm{O(u+h)-y_d}_{Y}^{2}
 +\frac{w_\sigma}{2}\norm{R\sigma-r_d}_{Y_\sigma}^{2} \\
&+\frac{w_\Pi}{2}\norm{\Piop h-\pi_d}_{\R^{b_k}}^{2}
 +\frac{\alpha}{2}\norm{z}_{Z}^{2}
 +\frac{\alpha_{\rm top}}{2}\lvert a\rvert^{2},
\end{aligned}
\end{equation}
The first two terms track local physical quantities, the third tracks global circulation or flux, and the final terms regularize the distributed and topological controls.

We assume throughout that
\begin{equation}\label{eq:positive-reg}
 \alpha>0,\qquad \alpha_{\rm top}>0,
\end{equation}
and $w_y,w_\sigma,w_\Pi\ge0$.

The tracking term $O(u+h)$ acts on the physical state. The period term $\Piop h$ directly measures global circulation or flux. The compatibility multiplier $p=P_{\HH}(f+Bz)$ is absent from the objective.

This formulation has four principal consequences. First, the mixed Hodge--Laplace multiplier $p$ retains its correct role as the harmonic projection of the total forcing, while the physical harmonic state is represented by $h$. Second, topology becomes an actual part of the control model: the number of independent global modes is $b_k$, and their reachability is determined by $\range(G)$. Third, the first-order system remains entirely in weak form, without imposing unjustified strong-domain assumptions on adjoint variables. Fourth, a compatible FEEC discretization preserves the local complex, the harmonic dimension, and the period coordinates of the topological subsystem.


Our constrained minimization problem then
asks for the minimizer of
\begin{align}\label{eq:pdeconlaplace}
    \min\limits_{\substack{\sigma\in V^{k-1},\,u\in V^k,\,p,h\in\HH^k\\ z\in Z_{ad},\,a\in A_{ad}}}
    J(\sigma,u,h,z,a)
\end{align}
subject to the constraints
\begin{align}\label{eq:pdeconlaplace:constraints}
	\begin{array}{rll}
		0
		=
		 &
		\langle \sigma, \tau \rangle - \langle u, d \tau \rangle
		 &
		\forall \tau \in V^{k-1}
		\\
		\langle f+Bz,v \rangle
		=
		 &
		\langle d\sigma,v \rangle + \langle du,dv \rangle + \langle p,v \rangle
		 &
		\forall v \in V^k
		\\
		0 =
		 &
		\langle u, q \rangle
		 &
		\forall q \in \mathfrak H^k
		\\
        0 = & \Piop h-Ga-c_0 \\
        z & \in Z_{ad}, & a\in A_{ad}
	\end{array}
\end{align}
Generally speaking, the constraints formalize that $(\sigma,u,p) \in V^{k-1} \times V^k \times \mathfrak{H}^k)$ solves a stationary partial differential equation whose right-hand side is given by $z \in W^k$.
In fact, we often require $z$ to belong to a constrained subset $Z_{ad}\subset W^k$, while the finite-dimensional actuator belongs to a nonempty closed convex set $A_{ad}\subset\R^m$.
The optimal control problem can be interpreted as the attempt to match the state variables to a given target variable while ensuring that they solve a partial differential equation with a not too large right-hand side.
Note that the target variables are not required to solve the same partial differential equation. Even if the target states solve the original PDE, it is not guaranteed that the minimizer of the optimal control problem will reproduce that solution.

The tracking weights $w_y,w_\sigma,w_\Pi$ may vanish individually, depending on which physical quantities are observed.  The two control regularizations are treated differently: throughout the analysis we assume $\alpha>0$ and $\alpha_{\rm top}>0$ as in \eqref{eq:positive-reg}.  These terms provide coercivity and strong convexity in the distributed and topological controls, respectively.  We do not treat the unregularized cases $\alpha=0$ or $\alpha_{\rm top}=0$ in this manuscript.

\subsection{Existence and Uniqueness of the Optimal Control Solution}

With a convex objective function, a linear PDE equation system, and a constraint confining the search space to a closed convex set,
it is standard to expect that a solution exists. Indeed, that is the case, and for completeness we include a proof.

Let
\[
(\sigma(z),u(z),p(z))=S(f+Bz),\qquad h(a)=\Piop^{-1}(Ga+c_0).
\]
The reduced objective is
\begin{equation}
 \widehat J(z,a)=J(\sigma(z),u(z),h(a),z,a).
\end{equation}
The local and topological state maps are linear-affine and bounded. They may remain coupled in the objective through the observation $O(u+h)$ even though their state equations are separate.

\begin{theorem}[Existence and uniqueness]\label{thm:existence}
Assume that $(W,d)$ is a closed Hilbert complex, $B$, $O$, and $R$ are
bounded, $\Piop:\HH^k\to\R^{b_k}$ is an isomorphism, $Z_{\rm ad}\subset Z$
and $A_{\rm ad}\subset\R^m$ are nonempty closed convex sets, and
\eqref{eq:positive-reg} holds. Then the optimal control problem
\eqref{eq:pdeconlaplace}--\eqref{eq:pdeconlaplace:constraints}
has a unique optimal control pair
$(z^*,a^*)\in Z_{\rm ad}\times A_{\rm ad}$.
The associated states $(\sigma^*,u^*,p^*,h^*)$ are uniquely determined.
\end{theorem}

\begin{proof}
The reduced objective is continuous, weakly lower semicontinuous, and coercive because
\[
 \widehat J(z,a)\ge \frac{\alpha}{2}\norm{z}_Z^2+\frac{\alpha_{\rm top}}{2}|a|^2.
\]
The direct method gives existence on the nonempty closed convex admissible set. The two positive regularization terms make $\widehat J$ strongly convex in $(z,a)$, hence the minimizer is unique. The bounded state maps then give unique associated states.
\end{proof}

\begin{proposition}[Topological reachability]\label{prop:reachability}
A target period $\pi_d$ is exactly reachable by the topological state equation if and only if
\[
 \pi_d-c_0\in\range(G).
\]
The number of independently controllable topological modes is $\operatorname{rank}(G)\le b_k$. Full control of all cohomology modes is possible if and only if $\operatorname{rank}(G)=b_k$.
\end{proposition}

\begin{remark}
This proposition is the minimum point at which topology changes the control problem rather than merely the forward linear algebra. The Betti number determines the dimension of the global state, while the actuator matrix determines which of those modes can be reached.
\end{remark}

\begin{proposition}[Compatibility multiplier and harmonic-inert control]
\label{prop:p-inert}
For every feasible tuple in \eqref{eq:pdeconlaplace:constraints},
\[
 p=P_{\HH}(f+Bz).
\]
Hence the compatibility multiplier is independent of the physical harmonic state
$h$ and of the topological actuator $a$.

Assume in addition that the admissible distributed-control set is
\emph{harmonic-invariant}: for every $z\in Z_{\rm ad}$ there exists
$\widetilde z\in Z_{\rm ad}$ such that
\[
 (I-P_{\HH})B\widetilde z=(I-P_{\HH})Bz,\qquad
 P_{\HH}B\widetilde z=0,\qquad
 \norm{\widetilde z}_{Z}\le \norm{z}_{Z},
\]
with strict inequality whenever $P_{\HH}Bz\neq0$. Then every optimal control
satisfies
\[
 P_{\HH}Bz^{\star}=0,\qquad p^{\star}=P_{\HH}f.
\]
In particular, if $P_{\HH}f=0$, then $p^{\star}=0$. This hypothesis holds, for
example, for $B=I$ with the unconstrained control space
$Z_{\rm ad}=W^{k}$, and also when the control bounds are inactive and the
harmonic component can be removed without violating admissibility.
\end{proposition}

\begin{proof}
Testing the second equation of \eqref{eq:pdeconlaplace:constraints} with
$q\in\HH^{k}$ gives
\[
 \ip{f+Bz}{q}=\ip{p}{q},
\]
because $q$ is harmonic, so it is orthogonal to $d\sigma$ and satisfies
$dq=0$. Since $p\in\HH^{k}$, this proves
$p=P_{\HH}(f+Bz)$.

For the second assertion, replacing $z$ by $\widetilde z$ leaves the
nonharmonic forcing $(I-P_{\HH})(f+Bz)$ unchanged. Hence the gauge-fixed local
states $\sigma$ and $u$ are unchanged, while $h$ depends only on $a$. The
objective therefore changes only through the regularization term
$\alpha\norm{z}_{Z}^{2}/2$. Since $\alpha>0$, optimality rules out
$P_{\HH}Bz^{\star}\neq0$.
\end{proof}
\subsection{Optimality Conditions}

We derive the optimality system entirely in weak form. This is important because the formal strong adjoint operators would require additional domain and regularity assumptions that are neither needed nor generally available in FEEC spaces.

Introduce multipliers
\[
 \lambda\in V^{k-1},\qquad \mu\in V^k,\qquad \nu\in\HH^k,
 \qquad \xi\in\R^{b_k}.
\]
Using the sign convention
\begin{equation*}
\begin{array}{rl}
 L={}&J
-\big(\ip{\sigma}{\lambda}-\ip{u}{d\lambda}\big)
-\big(\ip{d\sigma}{\mu}+\ip{du}{d\mu}+\ip{p}{\mu}-\ip{f+Bz}{\mu}\big)\\
&-\ip{u}{\nu}+\xi\cdot(\Piop h-Ga-c_0),
\end{array}
\end{equation*}
stationarity with respect to the state variables gives the adjoint equations.

\begin{theorem}[Weak optimality system]\label{thm:optimality}
A feasible point $(\sigma,u,p,h,z,a)$ is optimal if and only if there exist $(\lambda,\mu,\nu,\xi)$ such that, in addition to the forward PDE system,
\begin{subequations}
\begin{align}
 w_\sigma\ip{R\sigma-r_d}{R\tau}
 -\ip{\lambda}{\tau}-\ip{\mu}{d\tau}&=0
 &&\forall\tau\in V^{k-1},\\
 w_y\ip{O(u+h)-y_d}{Ov}
 +\ip{v}{d\lambda}-\ip{d\mu}{dv}-\ip{v}{\nu}&=0
 &&\forall v\in V^k,\\
 \ip{\mu}{q}&=0
 &&\forall q\in\HH^k,\label{eq:adj3}\\
 w_y\ip{O(u+h)-y_d}{O\eta}
 +w_\Pi(\Piop h-\pi_d)\cdot\Piop\eta
 +\xi\cdot\Piop\eta&=0
 &&\forall\eta\in\HH^k,\label{eq:adj4}
\end{align}
\end{subequations}
and the variational inequalities
\begin{subequations}\label{eq:control-vi}
\begin{align}
 \ip{\alpha z+B^*\mu}{\widetilde z-z}_Z&\ge0
 &&\forall\widetilde z\in Z_{\rm ad},\label{eq:zvi}\\
 (\alpha_{\rm top}a-G^T\xi)\cdot(\widetilde a-a)&\ge0
 &&\forall\widetilde a\in A_{\rm ad}.\label{eq:avi}
\end{align}
\end{subequations}
Because the problem is convex, these conditions are both necessary and sufficient.
\end{theorem}

\begin{remark}[Meaning of the harmonic adjoint equations]
Equation \eqref{eq:adj3} states $\mu\perp\HH^k$; it does not imply $\mu=0$. Equation \eqref{eq:adj4} is a finite-dimensional adjoint balance on the physical harmonic state $h$. These two equations have distinct roles and should not be conflated.
\end{remark}

If the controls are unconstrained, \eqref{eq:control-vi} reduces to
\begin{equation}
 z=-\alpha^{-1}B^*\mu,
 \qquad
 a=\alpha_{\rm top}^{-1}G^T\xi.
\end{equation}

\paragraph{Finite-dimensional form of the topological adjoint}
Choose a period-normalized basis $\{h_i\}_{i=1}^{b_k}$ with $\Piop h_i=e_i$, and write $h=Hc$, where $c=\Piop h$ and $H:\R^{b_k}\to\HH^k$ is the harmonic reconstruction map. Then $c=Ga+c_0$. Equation \eqref{eq:adj4} becomes
\begin{equation}\label{eq:top-adjoint-matrix}
 w_y H^*O^*(O(u+Hc)-y_d)+w_\Pi(c-\pi_d)+\xi=0
 \quad\text{in }\R^{b_k}.
\end{equation}
This explicitly exhibits the topology-dependent adjoint subsystem.

\begin{proposition}[Topological balance law]\label{prop:balance}
Assume that the topological control is unconstrained, or that its bounds are
inactive. With the period-normalized reconstruction map $H$ introduced above,
set
\[
 Q_H:=H^{*}O^{*}OH,
 \qquad
 r(z):=H^{*}O^{*}\bigl(y_d-Ou(z)\bigr).
\]
Then the optimal topological control satisfies
\[
 \Bigl[\alpha_{\rm top}I+G^{T}\bigl(w_yQ_H+w_{\Pi}I\bigr)G\Bigr]a
 =
 G^{T}\Bigl[
   w_y r(z)+w_{\Pi}\pi_d
   -\bigl(w_yQ_H+w_{\Pi}I\bigr)c_0
 \Bigr].
\]
Thus the topological and distributed controls are coupled only through
$r(z)$, equivalently through the cross term
$H^{*}O^{*}Ou(z)$. If $O=I$ and $u\perp\HH^{k}$, then
$r(z)=H^{*}y_d$ and the balance law reduces to
\[
 \Bigl[\alpha_{\rm top}I+G^{T}\bigl(w_yM+w_{\Pi}I\bigr)G\Bigr]a
 =
 G^{T}\Bigl[
   w_y d+w_{\Pi}\pi_d
   -\bigl(w_yM+w_{\Pi}I\bigr)c_0
 \Bigr],
 \qquad
 M:=H^{*}H,\quad d:=H^{*}y_d.
\]
\end{proposition}

\begin{proof}
Because the topological control is unconstrained, or its bounds are inactive,
\eqref{eq:avi} gives
\[
 \alpha_{\rm top}a=G^{T}\xi.
\]
From \eqref{eq:top-adjoint-matrix}, with $c=Ga+c_0$,
\[
 \xi
 =
 -w_yH^{*}O^{*}\bigl(Ou+OHc-y_d\bigr)
 -w_{\Pi}(c-\pi_d).
\]
Substitution into $\alpha_{\rm top}a=G^{T}\xi$ and collection of the terms
containing $a$ gives the first balance law. The second follows from $O=I$ and
$H^{*}u=0$.
\end{proof}
\paragraph{Reduced Hessian and mesh-independent structure}
For unconstrained controls, define the local state-observation operator
\[
 K_z z=\big(\sqrt{w_\sigma}R\sigma(z),\sqrt{w_y}Ou(z)\big)
\]
and the topological observation operator
\[
 K_a a=\big(\sqrt{w_y}OH\,Ga,\sqrt{w_\Pi}Ga\big).
\]
The full reduced Hessian has the block form
\begin{equation}\label{eq:reduced-hessian}
 \nabla^2\widehat J=
 \begin{pmatrix}
 \alpha I+K_z^*K_z & K_z^*K_a\\
 K_a^*K_z & \alpha_{\rm top}I+K_a^*K_a
 \end{pmatrix}.
\end{equation}
Hence
\begin{equation}\label{eq:hessian-lower}
 \ip{\nabla^2\widehat J(\delta z,\delta a)}{(\delta z,\delta a)}
 \ge \alpha\norm{\delta z}^2+\alpha_{\rm top}|\delta a|^2.
\end{equation}
The topology-dependent block has dimension at most $b_k$ and can therefore be treated exactly or through a low-rank Schur complement.

\subsection{Spherical shell: cavity-flux control}\label{sec:shell}

\medskip
\noindent\textbf{The degree-two mixed state.}
At degree two, the de Rham complex gives
\[
  H(\curl;\Omega)\xrightarrow{\ \curl\ }
  H(\divergence;\Omega)\xrightarrow{\ \divergence\ }L^2(\Omega),
\]
and we seek the gauge-fixed local state
$(\sigma,u,p)\in H(\curl;\Omega)\times H(\divergence;\Omega)\times\HH^2$
satisfying
\begin{subequations}\label{eq:shell-mixed}
\begin{align}
  \ip{\sigma}{\tau}-\ip{u}{\curl\tau} &= 0
    &&\forall\tau\in H(\curl;\Omega),\\
  \ip{\curl\sigma}{v}+\ip{\divergence u}{\divergence v}+\ip{p}{v}
    &= \ip{f+Bz}{v}
    &&\forall v\in H(\divergence;\Omega),\\
  \ip{u}{q} &= 0 &&\forall q\in\HH^2 .
\end{align}
\end{subequations}
Here $z$ is the distributed control, $f$ a fixed forcing, and $p$ the
harmonic compatibility component of the total forcing. As in the
abstract formulation, $p$ is \emph{not} the physical harmonic state.
The physical degree-two field is instead
\begin{equation}\label{eq:shell-split}
  y=u+h,\qquad h\in\HH^2 ,
\end{equation}
the mixed equations determining the orthogonal representative $u$ while
the additional variable $h$ carries the global cavity flux.
 
\subsubsection{The cavity-flux period map}
\label{sec:shell-period}
 
Let $\Sigma\subset\Omega$ be any smooth closed oriented surface that
surrounds the cavity once, for instance
$\Sigma=\{x:|x-x_0|=r_\Sigma\}$ with
$r_{\rm in}<r_\Sigma<r_{\rm out}$. For $h\in\HH^2$ define
\begin{equation}\label{eq:shell-period-map}
  \Piop h:=\int_\Sigma h\cdot n\,dS .
\end{equation}
Because $h$ is closed at degree two, that is $\divergence h=0$, the
value of \eqref{eq:shell-period-map} does not depend on the
representative $\Sigma$: if $\Sigma_1$ and $\Sigma_2$ both surround the
cavity once, the region $D$ between them lies in $\Omega$ and the
divergence theorem gives
\[
  \int_{\Sigma_1}h\cdot n\,dS-\int_{\Sigma_2}h\cdot n\,dS
  =\int_D\divergence h\,dx=0 .
\]
Thus $\Piop h$ is a genuine topological period.
 
A convenient continuous representative of the nontrivial class is the
radial field
\begin{equation}\label{eq:shell-radial-harmonic}
  \widetilde h(x)=\frac{x-x_0}{|x-x_0|^3},
\end{equation}
which is divergence-free and curl-free on $\Omega$ and has flux $4\pi$
through any surface enclosing the cavity. Normalizing the harmonic
basis $h_2\in\HH^2$ by
\begin{equation}\label{eq:shell-unit-flux}
  \Piop h_2=\int_\Sigma h_2\cdot n\,dS=1
\end{equation}
therefore determines it in closed form,
\begin{equation}\label{eq:shell-h2-closed}
  h_2=\frac{x-x_0}{4\pi|x-x_0|^3},
\end{equation}
exactly and not to any approximation, since $\dim\HH^2=1$ and
\eqref{eq:shell-radial-harmonic} already lies in $\HH^2$. Two
consequences will be used below. Every physical harmonic state has the
unique representation
\begin{equation}\label{eq:shell-harmonic-coordinate}
  h=c\,h_2,\qquad c=\Piop h ,
\end{equation}
with the scalar $c$ precisely the signed cavity flux. And the quantity
\begin{equation}\label{eq:shell-m-closed}
  m:=\norm{h_2}^2
   =\frac{1}{16\pi^2}\int_\Omega\frac{dV}{|x-x_0|^4}
   =\frac{1}{4\pi}\Bigl(\frac{1}{r_{\rm in}}-\frac{1}{r_{\rm out}}\Bigr),
\end{equation}
which enters the optimality law, is known analytically; for
$r_{\rm in}=0.18$ and $r_{\rm out}=0.46$ it equals
$2.69102561\times10^{-1}$. Note that $m\neq1$: the normalization
\eqref{eq:shell-unit-flux} fixes the period, not the $L^2$ norm, and
the two cannot be imposed together.
 
The definition \eqref{eq:shell-period-map} requires a surface. The
following removes that requirement, which matters at the discrete level
where exhibiting a mesh-resolved $\Sigma_h$ is awkward.
 
\begin{lemma}\label{lem:flux-period}
Let $\Omega\subset\R^3$ be a bounded Lipschitz domain whose boundary
has components $\Gamma_0,\dots,\Gamma_b$, and let $\psi\in H^1(\Omega)$
take the constant value $\psi_i$ on $\Gamma_i$. Then for every
$v\in H(\divergence;\Omega)$ with $\divergence v=0$,
\begin{equation}\label{eq:flux-period}
  \int_\Omega v\cdot\operatorname{grad}\psi
  =\sum_i\psi_i\int_{\Gamma_i}v\cdot n\,dS .
\end{equation}
\end{lemma}
 
\begin{proof}
Integrate by parts,
$\int_\Omega v\cdot\operatorname{grad}\psi
=-\int_\Omega\psi\,\divergence v+\int_{\partial\Omega}\psi\,v\cdot n$,
and the first term vanishes.
\end{proof}
 
Taking $\psi$ equal to one on the cavity wall and zero on the outer
sphere reduces \eqref{eq:flux-period} to a single term, so
\begin{equation}\label{eq:shell-period-potential}
  \Piop(v)=-\ip{v}{\operatorname{grad}\psi}
\end{equation}
agrees with \eqref{eq:shell-period-map} for divergence-free $v$.
Neither harmonicity of $\psi$ nor the shape of $\Omega$ is used, only
that $\psi$ is constant on each boundary component.
 
\subsubsection{Topological actuation}
\label{sec:shell-actuation}
 
Let $a\in\R$ denote a scalar actuator associated with the global
degree-two mode. We impose
\begin{equation}\label{eq:shell-topological-control}
  \Piop h=g\,a+c_0,\qquad g\neq0,
\end{equation}
where $c_0$ is a prescribed background flux and $g$ the actuator gain;
equivalently
\begin{equation}\label{eq:shell-h-of-a}
  h(a)=(g\,a+c_0)\,h_2 .
\end{equation}
Relation \eqref{eq:shell-h-of-a} is distinct from the distributed
forcing $Bz$. The
distributed control changes the local mixed state, whereas $a$ changes
the cohomological component of the physical field.
 
For a generic Hodge--Laplace model,
\eqref{eq:shell-topological-control} may be read directly as an imposed
flux actuator. In an electromagnetic interpretation, a nonzero
cavity-flux coordinate should be understood as a global flux
established by boundary data, an excluded source, linked circuitry, or
an imposed background field; it is not a magnetic-monopole source
inside the modeled material domain.
 
\subsubsection{Optimal-control problem}
\label{sec:shell-ocp}
 
We prescribe a desired physical state $y_d$, a desired auxiliary
quantity $r_d$, and a desired cavity flux $\Phi_d\in\R$. The shell
problem is
\begin{equation}\label{eq:shell-ocp}
\begin{aligned}
  \min_{\sigma,u,p,h,z,a}\quad
  J_{\rm sh}(\sigma,u,h,z,a)
  &:=\frac{w_y}{2}\norm{O(u+h)-y_d}_Y^2
    +\frac{w_\sigma}{2}\norm{R\sigma-r_d}_{Y_\sigma}^2\\
  &\quad+\frac{w_\Pi}{2}\bigl|\Piop h-\Phi_d\bigr|^2
    +\frac{\alpha}{2}\norm{z}_Z^2
    +\frac{\alpha_{\rm top}}{2}|a|^2,
\end{aligned}
\end{equation}
subject to \eqref{eq:shell-mixed}, \eqref{eq:shell-split} and
\eqref{eq:shell-topological-control}, with $\alpha>0$ and
$\alpha_{\rm top}>0$. Problem \eqref{eq:shell-ocp} is the degree-two
instance of the abstract objective~\eqref{eq:intro-objective}. No target is imposed on the compatibility
variable $p$. The topological term has the direct interpretation
\[
  \frac{w_\Pi}{2}\Bigl|\int_\Sigma h\cdot n\,dS-\Phi_d\Bigr|^2 ,
\]
and since the gauge-fixed state satisfies $u\perp\HH^2$, the variable
$h$ is the part of the physical field carrying the nontrivial global
flux class.
 
Propositions~\ref{prop:p-inert} and \ref{prop:balance} were proved for
an abstract Hilbert complex and therefore apply verbatim here.  Under the
harmonic-invariant distributed-control hypothesis of Proposition~\ref{prop:p-inert},
$p^\star=P_{\HH^2}f$; hence with $f=0$ the compatibility multiplier vanishes
at the optimum whatever the flux actuator does.  This is why assigning a
tracking target to $p$ cannot model cavity flux.
 
\subsubsection{A decoupled manufactured target}
\label{sec:shell-manufactured}
 
For verification it is convenient to choose the desired state in
Hodge-decomposed form,
\begin{equation}\label{eq:shell-manufactured-target}
  y_d=u_d+\gamma\,h_2,\qquad u_d\perp\HH^2,
\end{equation}
and to take $O=I$. Writing $h=c\,h_2$ as in
\eqref{eq:shell-harmonic-coordinate},
\[
  \norm{u+c\,h_2-y_d}^2=\norm{u-u_d}^2+|c-\gamma|^2\,m ,
\]
because $u-u_d\perp h_2$, with $m=\norm{h_2}^2$ given by
\eqref{eq:shell-m-closed}. The topological part of the reduced
objective is therefore
\begin{equation}\label{eq:shell-reduced-top}
  \widehat J_{\rm top}(a)
  =\frac{w_y m}{2}\,|g a+c_0-\gamma|^2
   +\frac{w_\Pi}{2}\,|g a+c_0-\Phi_d|^2
   +\frac{\alpha_{\rm top}}{2}\,|a|^2 ,
\end{equation}
and stationarity of \eqref{eq:shell-reduced-top} in $a$ gives the exact
scalar topological optimum
\begin{equation}\label{eq:shell-exact-a}
  a^\star=\frac{g\bigl[\,w_y m\,\gamma+w_\Pi\Phi_d
      -(w_y m+w_\Pi)c_0\,\bigr]}
     {\alpha_{\rm top}+(w_y m+w_\Pi)g^2},
\end{equation}
and the achieved optimal cavity flux is
\begin{equation}\label{eq:shell-exact-flux}
  c^\star=g\,a^\star+c_0 .
\end{equation}
The factor $m$ is essential and is easily lost: the normalization
\eqref{eq:shell-unit-flux} fixes $\Piop h_2=1$ and leaves
$\norm{h_2}^2=m\neq1$, so $w_y$ and $w_\Pi$ do not enter symmetrically.
Setting $\gamma=\Phi_d$ recovers the fully aligned case in which the
tracking term and the period term ask for the same flux, and
\eqref{eq:shell-exact-a} then reduces to
$g(w_ym+w_\Pi)(\Phi_d-c_0)/[\alpha_{\rm top}+(w_ym+w_\Pi)g^2]$.
Keeping $\gamma$ and $\Phi_d$ separate allows the two drives to be
isolated experimentally, which is done in
Section~\ref{sec:shell-numerics}.
 
Equations \eqref{eq:shell-exact-a} and \eqref{eq:shell-exact-flux} are the specialization to $b_k=1$ of Proposition~\ref{prop:balance}: with
$d:=\ip{y_d}{h_2}=\gamma m$ from
\eqref{eq:shell-manufactured-target}, the general law returns
\eqref{eq:shell-exact-flux}. Because $m$ is known in closed form by
\eqref{eq:shell-m-closed}, this provides a verification against an
analytic value and not merely an internal consistency check.  When $O$ couples harmonic and nonharmonic components, the topological and
distributed controls remain coupled through the reduced Hessian, and the
scalar law acquires the cross term $H^{*}O^{*}Ou(z)$ appearing in
Proposition~\ref{prop:balance}.

8

\section{Time-dependent Maxwell Control}
\label{sec:maxwell-topological}

The stationary construction extends naturally to evolution problems, but Maxwell's equations add an important feature: the harmonic components are themselves dynamical global degrees of freedom.  The curl operators determine the non-harmonic parts of the electric and magnetic fields, whereas the periods of the harmonic parts evolve through finite-dimensional balance laws.  This section formulates the resulting optimal-control problem without assigning a target to any harmonic compatibility multiplier.

\subsection{Control Problem Formulation}
\label{sec:maxwell-ocp}

Let $O_E:W^1\to Y_E$ and $O_B:W^2\to Y_B$ be bounded observation operators.  Let $E_d$, $B_d$, $c_{E,d}$, and $c_{B,d}$ be time-dependent targets.  We minimize
\begin{equation}
\begin{aligned}
  J_T={}&\frac12\int_0^T
  \Big[
    w_E\norm{O_E(e+H_Ec_E)-E_d}_{Y_E}^2
   +w_B\norm{O_B(b+H_Bc_B)-B_d}_{Y_B}^2\\
  &\hspace{17mm}
   +w_{\Pi,E}|c_E-c_{E,d}|^2
   +w_{\Pi,B}|c_B-c_{B,d}|^2
   +\alpha\norm{z}_Z^2
   +\alpha_E|a_E|^2
   +\alpha_B|a_B|^2
  \Big]dt\\
  &+\frac{\beta_E}{2}\norm{O_{E,T}(e(T)+H_Ec_E(T))-E_T}^2
   +\frac{\beta_B}{2}\norm{O_{B,T}(b(T)+H_Bc_B(T))-B_T}^2\\
  &+\frac{\beta_{\Pi,E}}{2}|c_E(T)-c_{E,T}|^2
   +\frac{\beta_{\Pi,B}}{2}|c_B(T)-c_{B,T}|^2,
  \label{eq:maxwell-objective}
\end{aligned}
\end{equation}
subject to \eqref{eq:maxwell-local-state}, \eqref{eq:maxwell-top-E-ode}, and either \eqref{eq:maxwell-magnetic-flux-conservation} or \eqref{eq:maxwell-top-B-ode}, together with the initial conditions
\begin{equation}
  c_E(0)=c_{E,0},\qquad c_B(0)=c_{B,0}.
\end{equation}
The admissible controls satisfy
\[
  z\in Z_{\rm ad}\subset L^2(0,T;Z),\qquad
  a_E\in A_{E,{\rm ad}}\subset L^2(0,T;\R^{m_E}),\qquad
  a_B\in A_{B,{\rm ad}}\subset L^2(0,T;\R^{m_B}),
\]
where the admissible sets are nonempty, closed, and convex.  We assume throughout
\begin{equation}
  \alpha>0,\qquad \alpha_E>0,\qquad \alpha_B>0
  \label{eq:maxwell-positive-reg}
\end{equation}
for every control that is present in the model.

\begin{proposition}[Existence and uniqueness for the Maxwell control problem]
\label{prop:maxwell-existence}
Assume the Maxwell evolution generates a bounded linear state map from the controls to $(e,b,c_E,c_B)$, the observation operators are bounded, and \eqref{eq:maxwell-positive-reg} holds.  Then the reduced functional associated with \eqref{eq:maxwell-objective} is strongly convex and coercive on the admissible control set.  Consequently the problem has a unique optimal control and a unique associated physical state.
\end{proposition}

\begin{proof}
The state equations are linear and the state map is bounded.  All tracking terms are convex quadratic functionals of the controls through this map.  Positive control regularization gives coercivity and strict convexity.  Weak lower semicontinuity and the direct method yield existence, and strict convexity yields uniqueness.
\end{proof}

\begin{proposition}[Topological reachability over a finite time interval]
\label{prop:maxwell-reachability}
For the magnetic topological subsystem with fixed initial flux $c_{B,0}$,
\[
  M_B^H\dot c_B=G_Ba_B+g_B,
\]
the terminal period $c_{B,T}$ is reachable if and only if
\begin{equation}
  M_B^H(c_{B,T}-c_{B,0})-\int_0^T g_B(t)\,dt
  \in\range(G_B).
\end{equation}
The analogous statement holds for $c_E$ after including the prescribed and distributed harmonic-current contributions in \eqref{eq:maxwell-top-E-ode}.  Thus the Betti numbers determine the dimensions of the topological state spaces, while the actuator ranks determine the number of independently reachable global modes.
\end{proposition}

\subsection{Weak adjoint system}
\label{sec:maxwell-adjoint}

The optimality conditions should again be kept in weak form.  Let $\lambda_E(t)$ and $\lambda_B(t)$ be adjoints for the local Maxwell equations and let $\xi_E(t)\in\R^{b_1}$ and $\xi_B(t)\in\R^{b_2}$ be adjoints for the topological period equations.  The local adjoints evolve backward in time.  For almost every $t\in(0,T)$ they satisfy
\begin{subequations}
\label{eq:maxwell-adjoint}
\begin{align}
  -\ip{\varepsilon\dot\lambda_E}{v}
  +\ip{\mu^{-1}dv}{\lambda_B}
  +w_E\ip{O_E(E)-E_d}{O_Ev}_{Y_E}
  &=0
  &&\forall v\in V^1\cap(\HH^1_{\varepsilon})^{\perp_\varepsilon},
  \\
  -\ip{\mu^{-1}\dot\lambda_B}{w}
  -\ip{\mu^{-1}w}{d\lambda_E}
  +w_B\ip{O_B(B)-B_d}{O_Bw}_{Y_B}
  &=0
  &&\forall w\in V^2\cap(\HH^2_{\mu^{-1}})^{\perp_{\mu^{-1}}},
  \\
  -M_E^H\dot\xi_E
  +w_E H_E^*O_E^*(O_E(E)-E_d)
  +w_{\Pi,E}(c_E-c_{E,d})
  &=0,
  \\
  -M_B^H\dot\xi_B
  +w_B H_B^*O_B^*(O_B(B)-B_d)
  +w_{\Pi,B}(c_B-c_{B,d})
  &=0.
\end{align}
\end{subequations}
Here $E=e+H_Ec_E$ and $B=b+H_Bc_B$.  The terminal conditions are obtained from the terminal objective.  Explicitly,
\begin{subequations}
\begin{align}
  \varepsilon\lambda_E(T)
  &=-\beta_E O_{E,T}^*\big(O_{E,T}E(T)-E_T\big)
  &&\text{on }(\HH^1_{\varepsilon})^{\perp_\varepsilon},\\
  \mu^{-1}\lambda_B(T)
  &=-\beta_B O_{B,T}^*\big(O_{B,T}B(T)-B_T\big)
  &&\text{on }(\HH^2_{\mu^{-1}})^{\perp_{\mu^{-1}}},\\
  M_E^H\xi_E(T)
  &=-\beta_E H_E^*O_{E,T}^*\big(O_{E,T}E(T)-E_T\big)
    -\beta_{\Pi,E}(c_E(T)-c_{E,T}),\\
  M_B^H\xi_B(T)
  &=-\beta_B H_B^*O_{B,T}^*\big(O_{B,T}B(T)-B_T\big)
    -\beta_{\Pi,B}(c_B(T)-c_{B,T}).
\end{align}
\end{subequations}
The notation in the first two lines means equality after testing against the corresponding harmonic-complement spaces; no strong-domain interpretation is required.

The control conditions are the variational inequalities
\begin{subequations}
\label{eq:maxwell-control-vi}
\begin{align}
  \int_0^T
  \ip{\alpha z+\mathcal B_\perp^*\lambda_E+C_E^*\xi_E}
      {\widetilde z-z}_Z\,dt
  &\ge0
  &&\forall\widetilde z\in Z_{\rm ad},
  \\
  \int_0^T
  (\alpha_Ea_E-G_E^T\xi_E)\cdot(\widetilde a_E-a_E)\,dt
  &\ge0
  &&\forall\widetilde a_E\in A_{E,{\rm ad}},
  \\
  \int_0^T
  (\alpha_Ba_B-G_B^T\xi_B)\cdot(\widetilde a_B-a_B)\,dt
  &\ge0
  &&\forall\widetilde a_B\in A_{B,{\rm ad}}.
\end{align}
\end{subequations}
Here $\mathcal B_\perp=P_E^\perp\mathcal B$.  The two terms involving the distributed control in \eqref{eq:maxwell-control-vi} show explicitly that a current actuator can affect both the local electric field and the topological electric-circulation subsystem.  If the current-control operator has no harmonic component, then $C_E=0$ and these channels decouple.

For unconstrained controls, \eqref{eq:maxwell-control-vi} reduces pointwise in time to
\begin{equation}
  z=-\alpha^{-1}(\mathcal B_\perp^*\lambda_E+C_E^*\xi_E),
  \qquad
  a_E=\alpha_E^{-1}G_E^T\xi_E,
  \qquad
  a_B=\alpha_B^{-1}G_B^T\xi_B.
\end{equation}

\begin{remark}[No strong adjoint Maxwell equations are needed]
The weak system \eqref{eq:maxwell-adjoint} is the natural adjoint formulation for FEEC.  Rewriting it formally using $d^*$ would require additional assumptions such as $\lambda_E\in\operatorname{dom}(d^*)$ or $\lambda_B\in\operatorname{dom}(d^*)$, which are neither necessary for the control analysis nor guaranteed by the finite element spaces.  The weak formulation is therefore the primary mathematical statement.
\end{remark}

\subsection{Time-dependent Maxwell control on a spherical shell}
\label{sec:maxwell-spherical-shell}

We next specialize the topology-aware Maxwell formulation to a
spherical shell
\begin{equation}
  \Omega
  =
  \left\{
    x\in\mathbb R^3:
    r_{\rm in}<|x-c|<r_{\rm out}
  \right\},
  \qquad c=(1/2,1/2,1/2).
\end{equation}
The domain is homotopy equivalent to $S^2$, and therefore
\[
  b_1(\Omega)=0,
  \qquad
  b_2(\Omega)=1.
\]
Thus the shell has no nontrivial degree-one circulation coordinate,
but it has one degree-two harmonic coordinate.  In the Maxwell
interpretation this coordinate is a global magnetic flux through a
closed surface surrounding the excluded cavity.

\subsubsection{Local and harmonic Maxwell fields}

Let $\mathfrak H_E^1$ and $\mathfrak H_B^2$ denote the harmonic spaces
associated with the electric and magnetic spaces and the chosen
boundary complex.  On the shell,
\[
  \mathfrak H_E^1=\{0\},
  \qquad
  \dim\mathfrak H_B^2=1.
\]
Accordingly, the electric field has no topological harmonic coordinate,
whereas the magnetic field is decomposed as
\begin{equation}
  B(t)=b(t)+h_B(t),
  \qquad
  b(t)\perp\mathfrak H_B^2,
  \qquad
  h_B(t)\in\mathfrak H_B^2.
\end{equation}
Choose $h_2\in\mathfrak H_B^2$ with unit period,
\begin{equation}
  \Pi_B h_2
  :=
  \int_\Sigma h_2\cdot n\,dS
  =
  1,
  \label{eq:max-shell-unit-flux}
\end{equation}
where $\Sigma\subset\Omega$ is any closed oriented surface surrounding
the cavity once.  Then
\begin{equation}
  h_B(t)=c_B(t)h_2,
  \qquad
  c_B(t)=\Pi_B B(t).
\end{equation}
As in the stationary shell example, a convenient continuous
representative is proportional to
\[
  \frac{x-c}{|x-c|^3},
\]
restricted to $\Omega$ and normalized by
\eqref{eq:max-shell-unit-flux}.

The physical significance of $c_B$ is global rather than local:
for any two homologous surfaces $\Sigma_1,\Sigma_2$ surrounding the
cavity,
\[
  \int_{\Sigma_1}B\cdot n\,dS
  =
  \int_{\Sigma_2}B\cdot n\,dS
\]
whenever $\div B=0$ in the material domain.  The flux cannot be
removed by a local deformation of the integration surface without
crossing the excluded cavity.

\subsubsection{Maxwell evolution and conservation of the topological flux}

For clarity, consider the standard first-order Maxwell system in weak
form with electric field $E$ and magnetic field $B$:
\begin{subequations}
\begin{align}
  \ip{\varepsilon E_t}{v}
  -
  \ip{\mu^{-1}B}{\curl v}
  &=
  -\ip{j+B_z z}{v},
  && \forall v\in H(\curl;\Omega),
  \\
  \ip{\mu^{-1}B_t}{q}
  +
  \ip{\mu^{-1}\curl E}{q}
  &=
  \langle \mathcal F_{\rm top}a_B,q\rangle,
  && \forall q\in H(\divergence;\Omega).
  \label{eq:max-shell-faraday}
\end{align}
\end{subequations}
Here $z(t)$ is a distributed electric/current control.  The term
$\mathcal F_{\rm top}a_B$ is written explicitly because a distributed
current control in Amp\`ere's equation does not, by itself, change the
degree-two cohomology class of $B$.

Projecting \eqref{eq:max-shell-faraday} onto the normalized harmonic
field $h_2$ gives
\begin{equation}
  m_B \dot c_B(t)
  =
  g_B a_B(t),
  \qquad
  m_B:=\ip{\mu^{-1}h_2}{h_2}>0,
  \label{eq:max-shell-top-dynamics}
\end{equation}
after defining the scalar actuator gain by
\[
  g_B a_B(t)
  :=
  \langle\mathcal F_{\rm top}a_B(t),h_2\rangle.
\]
The curl term drops out because $h_2$ is harmonic.

Equation \eqref{eq:max-shell-top-dynamics} exposes the topological
content of the Maxwell problem.  If no topological or boundary
magnetic actuator is present, $a_B\equiv0$, then
\begin{equation}
  \dot c_B(t)=0,
  \qquad
  c_B(t)=c_B(0).
\end{equation}
Thus the cavity flux is a conserved cohomological quantity.  A
time-dependent desired flux cannot be reached by the ordinary
distributed current control alone.  To make the flux itself a
controllable output, one must include a mechanism that acts on this
global degree of freedom, represented abstractly here by $a_B$.
Depending on the application, this can model boundary excitation,
linked external circuitry, imposed global flux, or another actuator
whose action is not contained in the local exact/coexact Maxwell
dynamics.

\subsubsection{Flux-tracking optimal-control problem}

Let $B_d(t)$ and $E_d(t)$ denote desired local fields and let
$\Phi_d(t)$ denote a desired cavity-flux trajectory.  We consider
\begin{equation}
\begin{aligned}
  J(E,B,z,a_B)
  &:=
  \frac{w_E}{2}
  \int_0^T
  \norm{O_EE(t)-E_d(t)}^2\,dt
  \\
  &\quad+
  \frac{w_B}{2}
  \int_0^T
  \norm{O_BB(t)-B_d(t)}^2\,dt
  \\
  &\quad+
  \frac{w_\Phi}{2}
  \int_0^T
  \left|
    \Pi_B B(t)-\Phi_d(t)
  \right|^2\,dt
  \\
  &\quad+
  \frac{\alpha}{2}
  \int_0^T\norm{z(t)}^2\,dt
  +
  \frac{\alpha_{\rm top}}{2}
  \int_0^T|a_B(t)|^2\,dt,
\end{aligned}
\end{equation}
with $\alpha>0$ and $\alpha_{\rm top}>0$, subject to the Maxwell
evolution, Gauss constraints, initial conditions, and
\eqref{eq:max-shell-top-dynamics}.  A terminal flux penalty
\begin{equation}
  \frac{w_{\Phi,T}}{2}
  |c_B(T)-\Phi_T|^2
  \label{eq:max-shell-terminal-flux}
\end{equation}
may be added when the control objective is to establish a prescribed
final flux rather than track a complete trajectory.

Using $B=b+c_Bh_2$, the explicitly topological part of the problem is
the scalar linear-quadratic control problem
\begin{equation}
  \min_{c_B,a_B}
  \frac{w_\Phi}{2}
  \int_0^T|c_B-\Phi_d|^2\,dt
  +
  \frac{\alpha_{\rm top}}{2}
  \int_0^T|a_B|^2\,dt,
  \qquad
  m_B\dot c_B=g_Ba_B,
  \qquad
  c_B(0)=c_{B,0},
\end{equation}
up to additional coupling caused by an observation $O_B$ that acts on
the full magnetic field.  This scalar subsystem is absent on a
topologically trivial domain and appears here solely because
$b_2(\Omega)=1$.

\subsubsection{Weak topological adjoint}

Introduce a scalar adjoint $\xi_B(t)$ for
\eqref{eq:max-shell-top-dynamics}.  With the sign convention
\[
  \int_0^T
  \xi_B(t)
  \bigl(
    m_B\dot c_B(t)-g_Ba_B(t)
  \bigr)\,dt,
\]
variation with respect to $c_B$ and integration by parts in time gives
\begin{equation}
  -m_B\dot\xi_B
  +
  w_\Phi(c_B-\Phi_d)
  +
  \mathcal C_B^\ast
  \bigl(O_BB-B_d\bigr)
  =
  0,
  \label{eq:max-shell-top-adjoint}
\end{equation}
where $\mathcal C_B:c\mapsto O_B(ch_2)$ records any coupling of the
full-field observation to the harmonic mode.  If there is no terminal
penalty, then
\[
  \xi_B(T)=0;
\]
with \eqref{eq:max-shell-terminal-flux},
\[
  m_B\xi_B(T)
  =
  -w_{\Phi,T}\bigl(c_B(T)-\Phi_T\bigr)
\]
for the stated Lagrangian sign convention.

The topological control condition is
\begin{equation}
  \alpha_{\rm top}a_B-g_B\xi_B=0
  \label{eq:max-shell-top-control-condition}
\end{equation}
in the unconstrained case, or the corresponding variational
inequality if $a_B$ is bounded.  Equations
\eqref{eq:max-shell-top-dynamics},
\eqref{eq:max-shell-top-adjoint}, and
\eqref{eq:max-shell-top-control-condition} form the
one-dimensional topology-dependent optimality subsystem.

\subsection{Interpretation for the examples}
\label{sec:maxwell-interpretation}

The topology-dependent extension has two particularly clean electromagnetic interpretations.
\begin{enumerate}
\item On a solid torus with $b_1=1$, $c_E(t)$ is a global electric circulation around the hole.  A distributed current with nonzero harmonic projection or a dedicated loop actuator can steer this circulation.  A target $c_{E,d}(t)$ is therefore physically meaningful.
\item On a spherical shell with $b_2=1$, $c_B(t)$ is magnetic flux through the cavity.  Without magnetic or boundary topological actuation, \eqref{eq:maxwell-magnetic-flux-conservation} makes this flux an invariant.  With a coil or boundary actuator represented by $G_Ba_B$, the same quantity becomes a genuine controlled topological state.
\end{enumerate}
For a domain with several independent tunnels or cavities, the vectors $c_E$ and $c_B$ contain one coordinate per cohomology generator.  Rank-deficient actuator matrices provide a direct numerical experiment in which some global modes are reachable and others are not, even though the local Maxwell solver remains well posed in all cases.

\section{Solution Methods}
\label{sec:solution-methods}

The stationary reduced problem is posed in the two control variables
\[
  \min_{(z,a)\in Z_{\rm ad}\times A_{\rm ad}}\widehat J(z,a),
\]
where the local mixed state is obtained from the distributed control $z$ and the harmonic state from the finite-dimensional actuator $a$.  The quadratic objective and linear state maps make the unconstrained problem a symmetric positive-definite reduced system.  Pointwise bounds or other closed convex constraints enter only through the control variational inequalities \eqref{eq:zvi}--\eqref{eq:avi}.

\subsection{Reduced gradient and projection}
\label{sec:reduced-gradient-methods}

Let $g_z$ and $g_a$ denote the reduced gradients in the natural control inner products.  The weak optimality system gives
\[
  g_z=\alpha z+B^*\mu,
  \qquad
  g_a=\alpha_{\rm top}a-G^T\xi.
\]
Consequently the constrained optimum is characterized by the projected fixed-point relations
\[
  z=P_{Z_{\rm ad}}(z-c_z g_z),
  \qquad
  a=P_{A_{\rm ad}}(a-c_a g_a),
  \qquad c_z,c_a>0.
\]
The first projection is an infinite-dimensional control-space projection before discretization; for pointwise box constraints it becomes pointwise clipping.  The second acts only in the finite-dimensional actuator space and is therefore inexpensive.  In the numerical experiments below the reported stationary optima are interior, so these projections are inactive and the reduced problem is solved by conjugate gradients.  We retain the constrained formulation because it is the natural one for engineering actuator limits.

\subsection{Primal--dual active sets and semismooth Newton}
\label{sec:active-set}

For illustration, suppose
\[
  Z_{\rm ad}=\{z:z_\ell\le z\le z_u\},
  \qquad
  A_{\rm ad}=\{a:a_\ell\le a\le a_u\}.
\]
Introduce normal-cone multipliers $\eta_z\in N_{Z_{\rm ad}}(z)$ and $\eta_a\in N_{A_{\rm ad}}(a)$.  The control conditions are equivalently
\[
  g_z+\eta_z=0,
  \qquad
  g_a+\eta_a=0.
\]
A primal--dual active-set iteration identifies the lower and upper active sets from the signs of a projected trial step, fixes the control on those sets, and solves the linearized state--adjoint system on the inactive set.  Because the topological actuator dimension is at most the number of chosen actuator channels, its active-set update can be performed exactly at negligible cost.

The same equations can be written as a semismooth system by using the projection residual
\[
  F(z,a)=
  \begin{pmatrix}
    z-P_{Z_{\rm ad}}(z-c_zg_z)\\
    a-P_{A_{\rm ad}}(a-c_ag_a)
  \end{pmatrix}=0.
\]
The projection onto box constraints is semismooth in the standard function-space setting.  A generalized Newton derivative therefore gives the familiar primal--dual active-set/semismooth-Newton method: on active degrees of freedom the control is fixed to its bound, while on inactive degrees of freedom one solves the reduced Newton equation.  The topological block contributes only a small dense block or a low-rank coupling to that equation.  Standard local superlinear convergence results apply under the usual strict-complementarity and second-order assumptions; see \cite{ulbrich2002semismooth,de2015numerical}.

\subsection{Reduced Hessian and Hessian-vector products}
\label{sec:hessian-methods}

For unconstrained controls, or on the inactive subspace of a constrained problem, the reduced Hessian is the block operator \eqref{eq:reduced-hessian}.  Its coercivity \eqref{eq:hessian-lower} permits conjugate gradients in the product control inner product.  No full Hessian matrix is required.

Given a direction $(\delta z,\delta a)$, first solve the linearized local mixed state with forcing $B\delta z$.  Since the state equation is linear, this is the same mixed Hodge--Laplace operator as in the forward solve.  The harmonic-state variation is explicit,
\[
  \delta h=HG\,\delta a.
\]
Next solve the weak adjoint system with right-hand sides generated by the observation variations $O(\delta u+\delta h)$ and $R\delta\sigma$.  If $(\delta\mu,\delta\xi)$ denotes the resulting adjoint variation, the Hessian action is
\[
  \nabla^2\widehat J(z,a)
  \begin{pmatrix}\delta z\\ \delta a\end{pmatrix}
  =
  \begin{pmatrix}
    \alpha\delta z+B^*\delta\mu\\
    \alpha_{\rm top}\delta a-G^T\delta\xi
  \end{pmatrix}.
\]
Thus one Hessian-vector product requires one incremental state solve and one incremental adjoint solve, together with operations in a $b_k$-dimensional harmonic reconstruction.  The same factorization or preconditioner used for the state and adjoint equations can be reused at every Krylov iteration.

The topology-dependent block is small.  If desired, $a$ may be eliminated exactly by a Schur complement.  In the unconstrained case, for fixed $z$ the topological balance law of Proposition~\ref{prop:balance} is a linear system of dimension $m$, and substitution of its solution produces a reduced problem in $z$ alone.  Conversely, retaining $(z,a)$ exposes the cross term $K_z^*K_a$ and is preferable when studying how a non-global observation couples local and topological controls.

\subsection{Mesh-independent iteration structure}
\label{sec:mesh-independent-methods}

At fixed positive regularization, the continuous reduced Hessian is a coercive identity perturbation by observation--solution operators.  A stable FEEC discretization approximates these operators without introducing mesh-dependent null directions.  The finite-dimensional topological block has fixed dimension $b_k$ under refinement.  Consequently, with a uniformly stable realization of the control inner product and the mixed solves, the spectrum relevant to the reduced Krylov iteration approaches that of the continuous reduced operator rather than deteriorating simply because the mesh is refined.  The stationary experiments test this claim directly by reporting conjugate-gradient counts under refinement and under sweeps in the regularization parameters.

For the time-dependent Maxwell problem the same organization is used after time discretization: forward integration of the local and topological states, backward integration of the weak adjoints, and a reduced gradient in $(z,a_E,a_B)$.  The harmonic states and adjoints remain $b_1$- and $b_2$-dimensional.  The Maxwell experiments in Sections~\ref{sec:maxwell-lshape-numerics}--\ref{sec:maxwell-shell-numerics} are designed to check gradient consistency, re-entrant-corner behavior, reachability, conservation, and mesh/time-step independence of this coupled solver.

\section{Discretization}\label{s:disc}

\subsection{Background: Finite Element Exterior Calculus}\label{s:feec}

We recap basic notions of differential forms.
For $k \in \mathbb Z$, we let $\Lambda^k$ be the set of alternating $k$-linear maps on $\mathbb R^n$.
The vector space $\Lambda^k$ has dimension $\binom{n}{k}$.
It is non-zero only for $0 \leq k \leq n$.
Here, we identify $\Lambda^0 = \mathbb R$.

Let $\Omega \subseteq \mathbb R^n$ be a domain.
A differential $k$-form over $\Omega$ is a function $\Omega \rightarrow \Lambda^k$.
We $C^{\infty}\Lambda^{k}(\Omega)$ be the space of smooth differential $k$-forms over $\Omega$, and $L^{2}\Lambda^{k}(\Omega)$ for the Hilbert space of differential $k$-forms over $\Omega$ with coefficients in $L^2(\Omega)$.
The exterior product $\omega \wedge \eta$ of a $k$-form $\omega$ and an $l$-form $\eta$ is bilinear in each argument and satisfies the identity $\omega \wedge \eta = (-1)^{kl} \eta \wedge \omega$.
Whenever $\omega \in C^{\infty}\Lambda^k$ and $X$ is a vector field,
we let $X \llcorner \omega \in \Lambda^{k-1}$ denote the contraction of $\omega$ with $X$. This is uniquely defined by the identity
$X \llcorner \omega( w_1, \ldots, w_{k-1} ) = \omega( X, w_1, \ldots, w_{k-1} )$.
We introduce the exterior derivative as
\begin{align}
	d\omega = \sum_{i=1}^{n} d x_{i} \wedge \partial_{i} \omega,
	\quad
	\omega \in C^{\infty}\Lambda^{k}(\Omega).
\end{align}
We introduce the exterior derivative of differential forms with coefficients in Lebesgue spaces defined in the sense of distributions.
We define the following Hilbert spaces of differential $k$-forms:
\begin{align}
	H\Lambda^{k}(\Omega)
	:=
	\left\{\;
	\omega \in L^{2}\Lambda^{k}(\Omega)
	\big|
	d\omega \in L^{2}\Lambda^{k+1}(\Omega)
	\;\right\}
	.
\end{align}
We can assemble the differential complex
\begin{align}
	\begin{CD}
		H\Lambda^{0}
		@>d>>
		H\Lambda^{1}
		@>d>>
		\cdots
		@>d>>
		H\Lambda^{n},
	\end{CD}
\end{align}
For each $r\in\mathbb N_0$, we write
$\mathcal P_r\Lambda^k$ for the space of polynomial differential
$k$-forms of degree at most $r$. For $r\geq1$, the trimmed polynomial
space is defined by
\[
\mathcal P_r^{-}\Lambda^k
=
\mathcal P_{r-1}\Lambda^k
+
X\llcorner\mathcal P_{r-1}\Lambda^{k+1},
\]
where $X$ is the source vector field. The standard inclusions are
\[
\mathcal P_{r-1}\Lambda^k
\subset
\mathcal P_r^{-}\Lambda^k
\subset
\mathcal P_r\Lambda^k.
\]
In particular,
\[
\mathcal P_r^{-}\Lambda^0
=
\mathcal P_r\Lambda^0,
\qquad
\mathcal P_r^{-}\Lambda^n
=
\mathcal P_{r-1}\Lambda^n.
\]

We consider polynomial de~Rham complexes.
For each $r \in \mathbb N_0$ there are two important examples:
the full polynomial de~Rham complex
\begin{equation*}
	\begin{CD}
		\mathcal P_{r}\Lambda^{0}
		@>d>>
		\mathcal P_{r-1}\Lambda^{1}
		@>d>>
		\cdots
		@>d>>
		\mathcal P_{r-n}\Lambda^{n},
	\end{CD}
\end{equation*}
and the trimmed polynomial complex
\begin{equation*}
	\begin{CD}
		\mathcal P_{r+1}^{-}\Lambda^{0}
		@>d>>
		\mathcal P_{r+1}^{-}\Lambda^{1}
		@>d>>
		\cdots
		@>d>>
		\mathcal P_{r+1}^{-}\Lambda^{n}.
	\end{CD}
\end{equation*}
More examples are found by combination of different arrows:
\begin{equation*}
	\begin{CD}
		\mathcal P_{r  }^{ }\Lambda^{k}
		@>d>>
		\mathcal P_{r-1}^{ }\Lambda^{k+1},
		@.
		\quad
		@.
		\mathcal P_{r+1}^{-}\Lambda^{k}
		@>d>>
		\mathcal P_{r+1}^{-}\Lambda^{k+1},
		\\
		\mathcal P_{r  }^{ }\Lambda^{k}
		@>d>>
		\mathcal P_{r  }^{-}\Lambda^{k+1},
		@.
		\quad
		@.
		\mathcal P_{r+1}^{-}\Lambda^{k}
		@>d>>
		\mathcal P_{r  }^{ }\Lambda^{k+1}
		.
	\end{CD}
\end{equation*}
In that manner, we can build $2^{n-1}$ different types of polynomial de~Rham complexes, each type parameterized over the polynomial degree $r$ and
having the form
\begin{equation*}
	\begin{CD}
		\mathcal P\Lambda^{0}
		@>d>>
		\mathcal P\Lambda^{1}
		@>d>>
		\cdots
		@>d>>
		\mathcal P\Lambda^{n},
	\end{CD}
\end{equation*}
where each $\mathcal P\Lambda^{k}$ is either a full polynomial space or a trimmed one.

Let $V_h^\bullet\subset V^\bullet$ be a finite-dimensional subcomplex with uniformly bounded commuting projection $\pi_h$. Let $\HH_h^k$ denote the discrete harmonic space,
\[
 \HH_h^k=\{v_h\in V_h^k:dv_h=0,\ \ip{v_h}{d\tau_h}=0\ \forall\tau_h\in V_h^{k-1}\}.
\]
The FEEC mixed state seeks
\[
 (\sigma_h,u_h,p_h)\in V_h^{k-1}\times V_h^k\times\HH_h^k
\]
satisfying the discrete analogue of \eqref{eqn:mixed-hodge-laplacian}. Standard FEEC theory gives uniform stability and quasi-optimality under the usual subcomplex and cochain-projection hypotheses \cite{arnold2010finite,arnold2010finite}.

\subsection{Discrete period map and harmonic reconstruction}
Let $b_k=\dim\HH^k$. On a topology-compatible mesh, $\dim\HH_h^k=b_k$. Choose discrete cycles or surfaces representing the same homology basis and define
\begin{equation}
 \Piop_h:\HH_h^k\to\R^{b_k}.
\end{equation}
A discrete period-normalized harmonic basis $h_{1,h},\dots,h_{b_k,h}$ satisfies
\[
 \Piop_hh_{i,h}=e_i.
\]
The reconstruction map $H_h:\R^{b_k}\to\HH_h^k$ is
\[
 H_hc=\sum_{i=1}^{b_k}c_i h_{i,h}.
\]
The discrete topological state equation is
\begin{equation}
 \Piop_hh_h=Ga_h+c_0,
 \qquad h_h=H_h(Ga_h+c_0).
\end{equation}
Thus the continuous and discrete models use the same physical period coordinates even though their harmonic basis fields differ.

\begin{assumption}[Uniform period stability]\label{ass:period}
There exists $C_\Pi>0$, independent of $h$, such that
\[
 \norm{\Piop_h}_{\mathcal L(\HH_h^k,\R^{b_k})}
 +\norm{\Piop_h^{-1}}_{\mathcal L(\R^{b_k},\HH_h^k)}\le C_\Pi.
\]
Moreover, after identification by the commuting projection,
\[
 \norm{H-H_h}_{\mathcal L(\R^{b_k},W^k)}\to0.
\]
\end{assumption}

The discrete objective is obtained by replacing $(\sigma,u,h)$ with $(\sigma_h,u_h,h_h)$. The discrete adjoint equations are precisely the transpose mixed system together with a $b_k$-dimensional topological adjoint block. No discrete strong adjoint operator is required.

In coefficient form, if $H_h$ is the matrix whose columns are a mass-orthonormal or period-normalized harmonic basis, the physical state vector is
\[
 y_h=u_h+H_hc_h,\qquad c_h=Ga_h+c_0.
\]
The state saddle matrix retains the harmonic border that enforces $u_h\perp\HH_h^k$, while the objective and the topological control act on $H_hc_h$, not on the compatibility coefficient $p_h$.

\begin{proposition}[Discrete period functionals]\label{prop:discrete-period}
Let $V_h^{\bullet}\subset V^{\bullet}$ be a subcomplex with discrete harmonic
space $\HH_h^{k}$, and let $\tilde{J}_h$ satisfy
\begin{equation}\label{eq:Jh-coclosed}
 P_{\HH_h}\tilde J_h=P_{\HH_h} J_h\text{ and }\ip{d\tau_h}{\tilde J_h}=0\qquad\text{for all }\tau_h\in V_h^{k-1}.
\end{equation}
Then for every discretely closed $v_h$ it holds that 
\begin{equation}\label{eq:discrete-period-both}
    \ip{v_h}{\tilde J_h} = \ip{v_h}{J_h}
\end{equation}
Moreover $\Piop_{\tilde J_h}(v_h)=\ip{v_h}{\tilde J_h}$ annihilates $dV_h^{k-1}$ exactly.
\end{proposition}
\begin{proof}

Condition \eqref{eq:Jh-coclosed} is the statement that $\Piop_{\tilde J_h}$ annihilates
$\f B_h^{k}=dV_h^{k-1}$. A discretely closed $v_h$ splits as
$v_h=P_{\HH_h}v_h+b_h$ with $b_h\in\f B_h^{k}$, and
$(I-P_{\HH_h})J_h\perp\HH_h^{k}$. 
Thus both functionals annihilate $\f B_h^k=dV^{k-1}_h$ and their harmonic forms coincide, so they agree on $\f Z^k_h=\f B^k_h\oplus \HH^k_h$.
\end{proof}

\begin{remark}[Enforcing \eqref{eq:Jh-coclosed}]\label{rem:leray}
Condition \eqref{eq:Jh-coclosed} is not inherited from a continuous $J$
satisfying the hypotheses of Corollary~\ref{cor:period-derham}. The discrete
boundary is not the boundary of $\Omega$, and interpolation of a rational field
does not commute with the exterior derivative, so the assembled functional
leaks onto $\f B_h^{k}$ at the size of the geometric error. The remedy costs
nothing. Let $\phi_h\in V_h^{k-1}$ solve
\[
 \ip{d\phi_h}{d\tau_h}=\ip{J_h}{d\tau_h}\qquad\forall\tau_h\in V_h^{k-1},
\]
gauge-fixed on $\ker d$, and replace $J_h$ by $J_h-d\phi_h$. The replacement
satisfies \eqref{eq:Jh-coclosed} by construction, and since
$d\phi_h\in\f B_h^{k}\perp\HH_h^{k}$ its harmonic part is unchanged, so by
Proposition~\ref{prop:discrete-period} the functional it induces on
$\HH_h^{k}$ is the same. In matrix form, with $D$ the incidence matrix of
$d^{k-1}$ and $b_J$ the assembled load vector, the condition is
$D^{\top}b_J=0$ and the correction is $b_J\mapsto b_J-MD\chi$ with
$(D^{\top}MD)\chi=D^{\top}b_J$.
\end{remark}

\subsection{Error estimates}
\label{sec:error-estimates}

We record the abstract estimate needed to transfer FEEC consistency to the optimal controls.  Let
\[
  X:=Z\times\R^m,
  \qquad
  x:=(z,a),
  \qquad
  \|x\|_X^2:=\|z\|_Z^2+|a|^2,
\]
and let $\widehat J$ and $\widehat J_h$ be the continuous and discrete reduced objectives.  Write
\[
  \gamma_0:=\min\{\alpha,\alpha_{\rm top}\}>0.
\]
For simplicity of presentation the theorem is stated first for variational discretization, so the admissible control set is the same in the continuous and discrete problems.  The usual best-approximation terms are then added for discretized control spaces.

To compare the continuous and discrete operators in common spaces, let
$P_h^W$ denote the bounded FEEC data projection and let
$\iota_h^X$ denote the natural reconstruction of the discrete mixed
state in the continuous mixed-state space.  Likewise, let
$P_h^{X^*}$ and $\iota_h^W$ denote the corresponding adjoint data
projection and reconstruction.  Define
\[
\widetilde S_h:=\iota_h^X S_hP_h^W,
\qquad
\widetilde S_h^*:=\iota_h^W S_h^*P_h^{X^*},
\qquad
\widetilde H_h:=\iota_h^{\HH}H_h,
\]
where
\[
\iota_h^{\HH}:\HH_h^k\hookrightarrow W^k
\]
is the natural inclusion.
Thus all operator differences below are taken between maps with the
same domains and codomains.

\begin{theorem}[Control error from reduced-gradient consistency]
\label{thm:control-error}
Assume the hypotheses of Theorem~\ref{thm:existence}, uniform FEEC stability of the mixed state and adjoint operators, and uniform period stability in Assumption~\ref{ass:period}.  Let $x^*=(z^*,a^*)$ and $x_h^*=(z_h^*,a_h^*)$ minimize $\widehat J$ and $\widehat J_h$ over the same nonempty closed convex admissible set.  Then
\begin{equation}\label{eq:control-error-abstract}
  \gamma_0\|x^*-x_h^*\|_X
  \le
  \|\nabla\widehat J(x_h^*)-\nabla\widehat J_h(x_h^*)\|_{X^*}.
\end{equation}
Moreover, on every bounded set containing the optimizers,
\begin{equation}\label{eq:gradient-consistency}
\begin{aligned}
\|\nabla\widehat J(x)-\nabla\widehat J_h(x)\|_{X^*}
\le C\big(
&\|S-\widetilde S_h\|
+\|S^*-\widetilde S_h^*\|
+\|H-\widetilde H_h\|
\big).
\end{aligned}
\end{equation}
where $S$ and $S^*$ denote the continuous mixed state and weak adjoint
solution operators, respectively, while $\widetilde S_h$ and
$\widetilde S_h^*$ denote their lifted discrete counterparts acting
between the corresponding continuous spaces. Likewise, $H$ is the
continuous period-normalized harmonic reconstruction map and
$\widetilde H_h$ is its lifted discrete counterpart. Consequently,
\begin{equation}\label{eq:control-error-operator}
  \|z^*-z_h^*\|_Z+|a^*-a_h^*|
  \le
  \frac{C}{\gamma_0}
  \big(
    \|S-\widetilde S_h\|
    +\|S^*-\widetilde S_h^*\|
    +\|H-\widetilde H_h\|
  \big).
\end{equation}
If the controls are also discretized, the right-hand side acquires the standard approximation terms
\[
  \operatorname{dist}(z^*,Z_{{\rm ad},h})
  +\operatorname{dist}(a^*,A_{{\rm ad},h}).
\]
\end{theorem}

\begin{proof}
Strong convexity of the reduced objective follows from the positive regularization terms and gives
\[
  \ip{\nabla\widehat J(x_h^*)-\nabla\widehat J(x^*)}{x_h^*-x^*}
  \ge \gamma_0\|x_h^*-x^*\|_X^2.
\]
The continuous and discrete variational inequalities, tested with $x_h^*$ and $x^*$ respectively and then added, imply
\[
  \gamma_0\|x_h^*-x^*\|_X^2
  \le
  \ip{\nabla\widehat J(x_h^*)-\nabla\widehat J_h(x_h^*)}{x_h^*-x^*},
\]
which yields \eqref{eq:control-error-abstract} by duality.  The reduced gradients are obtained by composing the observation operators with the state and weak adjoint solution maps, while the topological component is composed with the harmonic reconstruction.  Boundedness of $O$, $R$, $B$, and $G$, together with uniform boundedness of the state, adjoint, and period maps on the relevant bounded set, gives \eqref{eq:gradient-consistency} by adding and subtracting the continuous and discrete compositions.  Substitution yields \eqref{eq:control-error-operator}.  If the admissible controls are discretized, comparison with suitable admissible approximants gives the stated additional best-approximation terms.
\end{proof}
Define the gap between the continuous and discrete harmonic spaces by
\[
\delta_h
:=
\operatorname{gap}(\HH^k,\HH_h^k)
:=
\max\left\{
\sup_{\substack{q\in\HH^k\\\|q\|=1}}
\inf_{q_h\in\HH_h^k}\|q-q_h\|,
\;
\sup_{\substack{q_h\in\HH_h^k\\\|q_h\|=1}}
\inf_{q\in\HH^k}\|q_h-q\|
\right\}.
\]
The state error follows from FEEC quasi-optimality plus the control error.  In particular,
\begin{equation}
\begin{aligned}
&\|\sigma^*-\sigma_h^*\|_V
+\|u^*-u_h^*\|_V
+\|p^*-p_h^*\|
+\|h^*-h_h^*\|
\\
&\quad\le C\Big(
\inf_{\tau_h}\|\sigma^*-\tau_h\|_V
+
\inf_{v_h}\|u^*-v_h\|_V
+
\inf_{q_h}\|p^*-q_h\|
\\
&\hspace{18mm}
+\delta_h\|f+Bz^*\|
+\|H-\widetilde H_h\|\,|Ga^*+c_0|
+\|z^*-z_h^*\|_Z
+|a^*-a_h^*|
\Big).
\end{aligned}
\end{equation}
The same argument applies to the weak adjoint variables.  A rate statement is therefore obtained by inserting the regularity-dependent FEEC approximation estimates for the selected family and the approximation order of the period-normalized harmonic basis.  The numerical examples deliberately include both a smooth/topologically trivial case and a geometry in which topology forces a singular harmonic representative, so that these regularity-dependent rates can be observed rather than hidden inside an operator norm.

\subsection{Cavity Shell FEEC Discretization}
\label{sec:shell-feec}
 
Let $\mathcal T_h$ be a tetrahedral mesh of the shell. At degree two we
use the compatible pair
\[
  \sigma_h\in\mathcal P^-_r\Lambda^1(\mathcal T_h),\qquad
  u_h\in\mathcal P^-_r\Lambda^2(\mathcal T_h),
\]
N\'ed\'elec and Raviart--Thomas in vector proxies, with $r=1$ in the
experiments below and the control in the piecewise-constant vector
space $\mathrm{DG}_0$. Let $\HH^2_h$ be the discrete harmonic space;
for topology-faithful meshes $\dim\HH^2_h=1$. The discrete physical
harmonic state is
\begin{equation}\label{eq:shell-discrete-h}
  h_h=c_h\,h_{2,h},\qquad c_h=g\,a_h+c_0,
\end{equation}
so that the complete physical state is $y_h=u_h+h_h$ by
\eqref{eq:shell-discrete-h}, the local mixed state
retains the harmonic border enforcing $u_h\perp\HH^2_h$, and the
objective acts on $y_h$ and on the period $c_h$.
 
The generator $h_{2,h}$ is obtained from
\eqref{eq:shell-radial-harmonic} in two projections, and neither is
optional.
 
First a Leray projection onto the discretely divergence-free
Raviart--Thomas subspace, from the saddle system pairing the mass
matrix with the divergence coupling into $\mathrm{DG}_0$. This is
required because the interpolant of \eqref{eq:shell-radial-harmonic} is
\emph{not} discretely divergence-free. The canonical Raviart--Thomas
degrees of freedom are face fluxes, and the commuting property
$\divergence(\Pi_{RT}v)=P_0(\divergence v)$ holds only when those
fluxes are evaluated exactly; a quadrature rule that is exact for
polynomials is not exact for a rational field, so the cell divergences
do not telescope to zero. The residual is large, not small: on the
meshes used below it is comparable to the norm of the field itself.
 
Second the removal of the exact part. At degree two the exact forms are
curls of N\'ed\'elec functions, so the projection uses the curl--curl
operator on $\mathcal P^-_1\Lambda^1$, which is singular on gradients;
we gauge-fix with a multiplier against the gradient coupling from
$\mathcal P_1\Lambda^0$ and subtract the discrete curl of the resulting
potential, whose Raviart--Thomas interpolant is exact.
 
The period functional is assembled from
\eqref{eq:shell-period-potential} rather than from a mesh-resolved
surface. Both hypotheses of Lemma~\ref{lem:flux-period} must hold
discretely and neither is automatic. The divergence condition is what
the Leray projection supplies. The boundary condition fails if $\psi$
is taken as the radial function
$(r_{\rm out}-|x-x_0|)/(r_{\rm out}-r_{\rm in})$, because the discrete
boundary is made of chords whose interiors do not lie on the spheres;
we therefore take $\psi_h\in\mathcal P_1\Lambda^0$ with vertex values
pinned to one on the cavity wall and zero on the outer sphere, which is
exactly constant on every boundary face. The assembled functional is
then projected so that it annihilates the discrete curls, exactly as
the degree-one period functionals are projected against the discrete
gradients.
 
The defining diagnostics for the generator are
\begin{equation}\label{eq:shell-harmonic-diagnostics}
  \norm{\divergence h_{2,h}}=0,\qquad
  \ip{h_{2,h}}{\curl\tau_h}=0\ \ \forall\tau_h,\qquad
  \Piop_h h_{2,h}=1,
\end{equation}
all three holding to round-off rather than in the limit, once the two
projections are in place.

\subsection{Maxwell System FEEC semidiscretization}
\label{sec:maxwell-feec}

Let
\[
  V_h^0\xrightarrow{d}V_h^1\xrightarrow{d}V_h^2\xrightarrow{d}V_h^3
\]
be a FEEC subcomplex with a uniformly bounded commuting projection.  Let $\HH_{E,h}^1$ and $\HH_{B,h}^2$ be the corresponding weighted discrete harmonic spaces.  For topology-compatible meshes,
\[
  \dim\HH_{E,h}^1=b_1,
  \qquad
  \dim\HH_{B,h}^2=b_2.
\]
Choose discrete period-normalized harmonic bases and reconstruction maps
\[
  H_{E,h}:\R^{b_1}\to\HH_{E,h}^1,
  \qquad
  H_{B,h}:\R^{b_2}\to\HH_{B,h}^2,
\]
with
\[
  \Piop_{E,h}H_{E,h}=I_{b_1},
  \qquad
  \Piop_{B,h}H_{B,h}=I_{b_2}.
\]
The semidiscrete physical fields are
\begin{equation}
  E_h=e_h+H_{E,h}c_{E,h},
  \qquad
  B_h=b_h+H_{B,h}c_{B,h}.
\end{equation}
The local fields solve the Galerkin version of \eqref{eq:maxwell-local-state}, and the period vectors solve
\begin{subequations}
\begin{align}
  M_{E,h}^H\dot c_{E,h}
  &=-j_{H,h}-C_{E,h}z_h+G_Ea_{E,h}+g_E,\\
  M_{B,h}^H\dot c_{B,h}
  &=G_Ba_{B,h}+g_B.
\end{align}
\end{subequations}
The same time integrator should be used consistently for the local and topological blocks when conservation or symplectic structure is important.  The finite-dimensional topological equations may also be integrated exactly over each time step when the controls are piecewise polynomial in time.

The discrete adjoint is obtained by transposing the semidiscrete evolution system and integrating backward in time.  In particular, the harmonic adjoints remain exactly $b_1$- and $b_2$-dimensional.  No discrete approximation of a strong codifferential is needed.

\paragraph{Ball with Cavity}

Let
\[
  V_h^1\subset H(\curl;\Omega),
  \qquad
  V_h^2\subset H(\divergence;\Omega)
\]
be a compatible FEEC subcomplex, for example first-kind N\'ed\'elec
and Raviart--Thomas spaces of compatible degree.  The discrete
magnetic harmonic space is
\[
  \mathfrak H_{B,h}^2
  =
  \left\{
    q_h\in V_h^2:
    d q_h=0,\;
    q_h\perp dV_h^1
  \right\}.
\]
For a topology-faithful shell mesh,
\[
  \dim\mathfrak H_{B,h}^2=1.
\]
Choose $h_{2,h}$ normalized by
\begin{equation}
  \Pi_{B,h}h_{2,h}
  =
  \int_{\Sigma_h}h_{2,h}\cdot n_h\,dS
  =
  1.
\end{equation}
The semidiscrete magnetic field is decomposed as
\begin{equation}
  B_h(t)=b_h(t)+c_{B,h}(t)h_{2,h},
  \qquad
  b_h(t)\perp\mathfrak H_{B,h}^2.
\end{equation}
Projecting the semidiscrete Faraday equation onto $h_{2,h}$ yields
\begin{equation}
  m_{B,h}\dot c_{B,h}
  =
  g_{B,h}a_{B,h},
  \qquad
  m_{B,h}
  =
  \ip{\mu^{-1}h_{2,h}}{h_{2,h}},
\end{equation}
which is the discrete counterpart of
\eqref{eq:max-shell-top-dynamics}.

The FEEC approximation therefore preserves both pieces of the
continuous structure: the local Maxwell evolution on the
harmonic-orthogonal subspace and the finite-dimensional cohomological
flux dynamics.

\subsection{Error decomposition}
\label{sec:maxwell-error}

The total physical-field errors split into local FEEC error, topological-basis error, and period-control error:
\begin{align}
  \norm{E-E_h}
  &\le \norm{e-e_h}
  +\norm{(H_E-H_{E,h})c_E}
  +\norm{H_{E,h}(c_E-c_{E,h})},
  \\
  \norm{B-B_h}
  &\le \norm{b-b_h}
  +\norm{(H_B-H_{B,h})c_B}
  +\norm{H_{B,h}(c_B-c_{B,h})}.
\end{align}
Under the standard FEEC approximation hypotheses for the local evolution problem, uniform stability of the discrete period maps, and a stable time integrator, one obtains an estimate of the schematic form
\begin{equation}
\begin{aligned}
 &\norm{z-z_h}_{L^2(0,T;Z)}
 +\norm{a_E-a_{E,h}}_{L^2(0,T)}
 +\norm{a_B-a_{B,h}}_{L^2(0,T)}\\
 &\quad\le C\Big(
   \mathcal E_{\rm FEEC}(h)
  +\mathcal E_{\rm period}(h)
  +\mathcal E_{\rm time}(\Delta t)
  +\mathcal E_{\rm control}(h,\Delta t)
 \Big),
\end{aligned}
\end{equation}
where $\mathcal E_{\rm FEEC}$ is the local Maxwell approximation error, $\mathcal E_{\rm period}$ measures $H_E-H_{E,h}$ and $H_B-H_{B,h}$, $\mathcal E_{\rm time}$ is the time-discretization error, and $\mathcal E_{\rm control}$ is present only when the control spaces are discretized.  Positive regularization makes the reduced objective strongly convex, so state/adjoint consistency estimates transfer directly to the optimal controls.

\section{Numerical experiments - Stationary Hodge Laplacian}\label{sec:numerics}

The experiments of this section verify the two halves of the formulation
separately and then together: the FEEC discretization of the local mixed
Hodge--Laplace problem, and the finite-dimensional topological subsystem
carried by the period coordinates. Four domains are used, chosen so that the
harmonic dimension takes the values $0$, $1$, $2$ and, at the other degree,
$1$ again. On each of them Proposition~\ref{prop:balance} is checked against
computed quantities, and on the solid torus and the spherical shell it is
checked against a harmonic norm known in closed form. No exact solution of the
mixed system is available for the data used, so local state and adjoint errors
are not reported; the convergence evidence is instead the harmonic norm, the
period coordinates and the reduced objective, each against an exact or
extrapolated limit.

Throughout, $z$ is discretized in the piecewise-constant vector space
$\mathrm{DG}_{0}$, the control bounds are inactive at every optimum reported,
and the reduced problem is solved by conjugate gradients in the $M_{z}$ inner
product.

\subsection{Contractible domain: local control only}\label{sec:lshape}

On the L-shaped domain of Section~\ref{ex:lshape} the domain is contractible,
so $b_{1}=0$, the harmonic space $\HH^{1}$ is trivial, the compatibility
multiplier $p$ is absent, and the topological subsystem
\eqref{eq:intro-top-control} is empty: there is no period map, no actuator
matrix $G$ and no topological control $a$. The objective reduces to
\[
 J=\frac{w_y}{2}\norm{Ou-y_d}^2
  +\frac{w_\sigma}{2}\norm{R\sigma-r_d}^2
  +\frac{\alpha}{2}\norm{z}^2 ,
\]
and the example isolates the local part of the formulation, providing the
baseline against which the multiply connected domains are compared.

We solve the $k=1$ curl--curl system \eqref{eq:curlcurl} with the pair the
framework prescribes: the gauge-fixed state $u$ in the first-kind N\'ed\'elec
space $\mathcal{P}^{-}_{1}\Lambda^{1}$ and $\sigma$ in the Lagrange space
$\mathcal{P}_{1}\Lambda^{0}$. The scalar unknown is the co-derivative
$\sigma=d^{*}u$, equal to $-\divergence u$ for the de Rham $1$-form Hodge
Laplacian, and it is recovered weakly from the first equation of
\eqref{eq:curlcurl} together with the natural condition $u\cdot n=0$. We take
$B=I$, $O=R=I$, $w_{y}=w_{\sigma}=\alpha=1$, and $f=0$. The discrete harmonic
space is
\[
  \HH^{1}_{h}
  =\bigl\{\,u_{h}\in\mathcal{P}^{-}_{1}\Lambda^{1}:
    \ip{\curl u_{h}}{\curl v_{h}}=0,\quad
    \ip{u_{h}}{\grad\tau_{h}}=0
    \ \ \forall v_{h},\tau_{h}\,\bigr\},
\]
the discrete form of $\HH^{1}=\ker(\curl)\cap\ker(\divergence)$, orthogonality
to the range of the gradient being the discrete statement of
divergence-freeness.

\medskip
\noindent\textbf{The discrete topology.}
The mesh is generated rather than assumed, so we verify the topology by the
combinatorial computation of Section~\ref{sec:betti}. At $N=8$ it returns
$n_{V}=665$, $n_{E}=3736$, $n_{F}=5760$, $n_{T}=2688$, hence $\chi=1$, with
one connected component, one boundary surface, and all vertex and edge stars
connected, giving $(b_{0},b_{1},b_{2},b_{3})=(1,0,0,0)$. Therefore
$\HH^{1}_{h}=\{0\}$ and $p_{h}=0$. The same conclusion follows spectrally: on
that mesh the state operator has condition number $1.17\times10^{5}$, smallest
singular value $1.32\times10^{-3}$, and numerical nullity zero. The two
determinations of $\dim\HH^{1}_{h}$ are independent, one combinatorial and one
spectral.

\medskip
\noindent\textbf{Discrete operators.}
Let $\{\varphi_{j}\}$, $\{\psi_{i}\}$ and $\{\chi_{l}\}$ be the bases of
$\mathcal{P}_{1}\Lambda^{0}$, $\mathcal{P}^{-}_{1}\Lambda^{1}$ and
$\mathrm{DG}_{0}$, and let $M_{\sigma}$, $M_{u}$, $M_{z}$ be the corresponding
mass matrices, $K_{ij}=\ip{\curl\psi_{j}}{\curl\psi_{i}}$ the curl--curl
stiffness, $G_{ij}=\ip{\grad\varphi_{j}}{\psi_{i}}$ the discrete gradient
coupling, and $C_{il}=\ip{\chi_{l}}{\psi_{i}}$ the control-to-state coupling.
Since $b_{1}=0$ the harmonic border is empty and the state and adjoint systems
are
\begin{equation}\label{eq:lshape-A0}
  A_{0}\begin{pmatrix}\sigma\\u\end{pmatrix}
  =\begin{pmatrix}0\\Cz\end{pmatrix},
  \qquad
  A_{0}^{\top}\begin{pmatrix}\lambda\\\mu\end{pmatrix}
  =\begin{pmatrix}
     w_{\sigma}W_{\sigma}(\sigma-r_{d})\\
     w_{y}W_{y}(u-y_{d})\end{pmatrix},
  \qquad
  A_{0}=\begin{pmatrix}M_{\sigma}&-G^{\top}\\ G&K\end{pmatrix},
\end{equation}
with observation weights $W_{y}=O^{*}O$ and $W_{\sigma}=R^{*}R$, equal to
$M_{u}$ and $M_{\sigma}$ here. The reduced gradient in the $M_{z}$ inner
product is $\widehat g=\alpha z+M_{z}^{-1}C^{\top}\mu$, and
\[
  \widehat H=\alpha I+M_{z}^{-1}C^{\top}A_{0}^{-\top}WA_{0}^{-1}C
\]
is symmetric positive definite in that inner product, so we apply conjugate
gradients there, at one state solve and one adjoint solve per iteration. We do
not form $\widehat H$; its action on a direction $d$ is
$\alpha d+M_{z}^{-1}C^{\top}\mu_{d}$ with $\mu_{d}$ obtained from
\eqref{eq:lshape-A0} with right-hand side $(0,Cd)^{\top}$ and no target term.
A single sparse $LU$ factorization of $A_{0}$ serves both solves, and $M_{z}$
is diagonal. The reduced problem is quadratic, so its minimizer solves
$\widehat Hz=-\widehat g(0)$ and no line search is needed.

\medskip
\noindent\textbf{Results.}
The targets are
\begin{equation}\label{eq:lshape-targets}
  y_{d}=\bigl(0.1\sin(\pi x)\cos(\pi y),\
              0.1\cos(\pi x)\sin(\pi y),\ 0.05\,z\bigr)^{\top},
  \qquad
  r_{d}=0.1\sin(\pi x)\sin(\pi y).
\end{equation}
The state target is a gradient, $y_{d}=\grad\psi$ with
\begin{equation}\label{eq:psi}
  \psi(x,y,z)=-\frac{0.1}{\pi}\cos(\pi x)\cos(\pi y)+0.025\,z^{2},
\end{equation}
a fact of no consequence here, where $\HH^{1}=\{0\}$, but used in
Sections~\ref{sec:torus} and \ref{sec:fig8}, where it makes the trivial part
of the target cohomologically trivial to round-off rather than to
discretization error.

Starting from $z=0$, conjugate gradients reach relative residual $10^{-10}$ in
six iterations, the reduced gradient falling from $5.368\times10^{-3}$ to
$5.01\times10^{-13}$. On the $N=16$ mesh the objective components are
$2.585448\times10^{-3}$ for the state tracking, $1.064311\times10^{-3}$ for
the co-derivative tracking and $1.322585\times10^{-5}$ for the
regularization, totalling $3.662985\times10^{-3}$. The optimal control is
small and interior, $\max_{\Omega}|z|=9.99\times10^{-3}$, so the bound
constraints are inactive. The two blocks of \eqref{eq:lshape-A0} are solved to
relative residuals $1.9\times10^{-15}$ and $2.9\times10^{-13}$ on every mesh.

\medskip
\noindent\textbf{Verification of the adjoint.}
The reduced objective is exactly quadratic, so
\begin{equation}\label{eq:taylor}
  \widehat J(z+\varepsilon d)-\widehat J(z)
  =\varepsilon\,\ip{\widehat g(z)}{d}_{M_{z}}
  +\tfrac{\varepsilon^{2}}{2}\ip{d}{\widehat Hd}_{M_{z}}
\end{equation}
holds with no remainder, and evaluating it where the gradient does not vanish
verifies the gradient and the Hessian action together, to round-off. At $z=0$
on the $N=16$ mesh the two sides agree to relative error $3.3\times10^{-15}$
at $\varepsilon=10^{-2}$, degrading to $1.7\times10^{-13}$ and
$8.9\times10^{-11}$ at $10^{-3}$ and $10^{-4}$. The degradation is the
cancellation in differencing two nearby values of $\widehat J$, whose absolute
error is fixed while the difference scales with $\varepsilon$, so
\eqref{eq:taylor} is read at the largest $\varepsilon$; this reverses the
convention for a truncated difference quotient. At the optimum the first term
vanishes and the identity still verifies the Hessian action, but no longer the
gradient.

The same evaluation exposes the reduced Hessian. For $d$ normalized in the
$M_{z}$ inner product, $\tfrac12\ip{d}{\widehat Hd}=0.5000184$ at $\alpha=1$,
so the compact term of \eqref{eq:reduced-hessian} contributes
$\ip{d}{K_{z}^{*}K_{z}d}=3.68\times10^{-5}$ in that direction. This is a
Rayleigh quotient in one direction and not a bound on $\norm{K_{z}^{*}K_{z}}$,
which the regularization sweep below shows to be two orders larger.

\medskip
\noindent\textbf{Convergence.}
Table~\ref{tab:lshape-refine} reports five meshes. The refinement factors are
not constant, so the order is extracted from
\begin{equation}\label{eq:order}
  \frac{d_{i}}{d_{i+1}}
  =\frac{N_{i}^{-q}-N_{i+1}^{-q}}{N_{i+1}^{-q}-N_{i+2}^{-q}},
  \qquad d_{i}=\widehat J_{N_{i+1}}-\widehat J_{N_{i}},
\end{equation}
which reduces to the familiar two-term formula only when $N_{i+1}/N_{i}$ is
fixed. The three consecutive triples give $q=1.92$, $1.96$ and $1.97$, an
estimate that rises and settles, and Richardson extrapolation at $q=1.97$
gives $\widehat J_{\infty}=3.6707\times10^{-3}$. The seven digits carried by
$\widehat J$ limit these orders to two decimals. The conjugate-gradient count
is six on every mesh, so the reduced solver is mesh independent at fixed
regularization, the spectrum of $K_{z}^{*}K_{z}$ being a property of the
continuous problem.

The L-shaped domain has three reentrant edges of opening $3\pi/2$, and second
order may seem surprising in their presence. The evidence that no singular
mode is active is the order sequence itself, which rises toward $2$ rather
than falling toward the $4/3$ that an $r^{-1/3}$ contaminant would impose. A
tube probe is consistent with this: measuring the energy of the optimal state
in a tube of radius $R$ about those edges, over radii from $1.2h$ to $4h$ at
$N=20$, gives $E(R)\sim R^{2.087}$, against $R^{2}$ for a bounded field and
$R^{4/3}$ for the singular exponent $\lambda=\pi/\omega=\tfrac23$. The probe
radii lie within a few mesh widths of the edge, so it corroborates the order
sequence rather than replacing it. Section~\ref{sec:fig8}, where the harmonic
field is fixed by the topology and has no such freedom, attains the singular
rate on the same probe.

\begin{table}[htbp]
\centering
\begin{tabular}{rrrlcl}
\toprule
$N$ & cells & dofs $(\sigma,u)$ & $\widehat J$ & CG
    & $\max_{\Omega}|z|$\\
\midrule
 $8$ &  $2\,688$ &   $4\,401$ & $3.640814\times10^{-3}$ & $6$
     & $9.198\times10^{-3}$\\
$12$ &  $9\,072$ &  $13\,897$ & $3.657110\times10^{-3}$ & $6$
     & $9.800\times10^{-3}$\\
$16$ & $21\,504$ &  $31\,841$ & $3.662985\times10^{-3}$ & $6$
     & $9.995\times10^{-3}$\\
$20$ & $42\,000$ &  $60\,921$ & $3.665735\times10^{-3}$ & $6$
     & $1.007\times10^{-2}$\\
$24$ & $72\,576$ & $103\,825$ & $3.667237\times10^{-3}$ & $6$
     & $1.011\times10^{-2}$\\
\midrule
$\infty$ & & & $3.6707\times10^{-3}$ & &\\
\bottomrule
\end{tabular}
\caption{Refinement study on the L-shaped domain, with the Richardson limit at
the observed order $q=1.97$ from \eqref{eq:order}.}
\label{tab:lshape-refine}
\end{table}

\medskip
\noindent\textbf{Dependence on the regularization.}
The six iterations at $\alpha=1$ reflect a reduced Hessian dominated by its
regularization term. Fixing $N=16$ and weakening the regularization through
$\alpha=10^{0},10^{-1},\dots,10^{-6}$ gives counts $6$, $10$, $21$, $47$,
$115$, $293$, $819$ at tolerance $10^{-10}$. The growth per decade is not
constant: the successive values of $\log_{10}(n_{i+1}/n_{i})$ are $0.22$,
$0.32$, $0.35$, $0.39$, $0.41$ and $0.45$, so a single power law fitted across
the range, which returns $-0.36$, averages two regimes and is not a rate.

The behaviour is that of eigenvalue counting rather than of the worst-case
bound. Conjugate gradients resolve only the eigenvalues of $K_{z}^{*}K_{z}$
exceeding $\alpha$, so for an algebraically decaying spectrum
$\lambda_{j}\sim Cj^{-s}$ the count grows like $\alpha^{-1/s}$, giving a
per-decade exponent $1/s$. The observed limit $0.45$ corresponds to
$s\approx2.2$. The worst-case bound
$n\lesssim\tfrac12\kappa^{1/2}\log(2/\mathrm{tol})$ with
$\kappa=1+\norm{K_{z}}^{2}/\alpha$ is consistent with the data but is not
sharp anywhere in the range, and it is loosest at small $\alpha$: inverting it
at each value gives implied $\norm{K_{z}}^{2}$ of $2.1\times10^{-2}$,
$1.5\times10^{-2}$, $9.3\times10^{-3}$, $6.1\times10^{-3}$ and
$4.8\times10^{-3}$ as $\alpha$ falls from $10^{-2}$ to $10^{-6}$. That this
sequence decreases shows that the $\sqrt{\kappa}$ regime is not being
approached, and that the shallow growth at moderate $\alpha$ is the decay of
the compact spectrum and not a preasymptotic artefact.

\subsection{Solid torus: one circulation mode}\label{sec:torus}

For the solid torus of Section~\ref{ex:torus} we have $b_{1}=1$, and we solve
\eqref{eq:pdeconlaplace}--\eqref{eq:pdeconlaplace:constraints} with a scalar
topological actuator, in the discretization of Section~\ref{sec:lshape}.

Three things become visible that were absent on a contractible domain. The
harmonic border is required for the discrete problem to be well posed, and the
requirement is fixed by the topology alone. The compatibility multiplier $p$
is inert under control, as Proposition~\ref{prop:p-inert} asserts. And the
physical harmonic state $h$ responds to the objective, including when the
objective contains no period target.

\medskip
\noindent\textbf{The mesh.}
The domain is not a union of background cells, so the mesh must be
constructed. Carving, in which a background cell is retained when its centroid
lies in $\Omega$, is not adequate for a refinement study: the staircase
boundary has a volume that oscillates about the true volume as the resolution
varies, and on the sequence $N=16,20,24$ the discrete cross-sectional area
took the values $0.9753$, $1.0095$ and $0.9982$ in units of $\pi\rho^{2}$,
changing sign twice, while the objective rose and then fell. This is the
geometric variational crime of \cite{HolstStern2012a}, a property of the
domain approximation and not of the solver, whose residuals stayed at
$10^{-15}$ throughout.

We therefore sweep a triangulated cross-sectional disk around the azimuth.
The disk is discretized with $n_{r}$ concentric rings, ring $j$ carrying $6j$
points, so its boundary is a regular $n_{c}$-gon with $n_{c}=6n_{r}$; this
section is placed at $n_{\varphi}$ equally spaced angles, and each prism is
split into three tetrahedra by a rule depending only on the sorted vertex
indices of its base triangle, which is the same in every layer and therefore
conforming. Every boundary vertex lies exactly on the torus surface, so with
$\lambda(n)=\frac{n}{2\pi}\sin\frac{2\pi}{n}$ the area ratio of an inscribed
regular $n$-gon,
\begin{equation}\label{eq:torus-volume}
  \frac{|\Omega_{h}|}{|\Omega|}
  =\lambda(n_{c})\,\lambda(n_{\varphi})+O(h^{4}),
  \qquad |\Omega|=2\pi^{2}R\rho^{2}.
\end{equation}
At $n_{r}=2,3,4$, with $1800$, $6156$ and $14400$ cells and $n_{\varphi}$
chosen for isotropy, the computed deficits are $5.5092\%$, $2.4643\%$ and
$1.3984\%$, matching \eqref{eq:torus-volume} to three decimals; the azimuthal
factor carries a steady nineteen percent of each. Both factors are $O(h^{2})$
and both are deficits, so the discrete domain approaches $\Omega$ from inside
and the geometric error cancels in differences. The computation of
Section~\ref{sec:betti} returns $\chi=0$ with one boundary surface and all
stars connected at every level, giving $b_{1}=1$.

\medskip
\noindent\textbf{The harmonic generator.}
Let $\theta_{A}$ and $\theta_{B}$ be the branches of the azimuth cut at
$\varphi=\pi$ and $\varphi=0$, interpolated as $\mathcal{P}_{1}\Lambda^{0}$
functions. The gradient of a $\mathcal{P}_{1}$ function has continuous
tangential trace, so its N\'ed\'elec interpolant is exact and each degree of
freedom is exactly the increment along that edge. Selecting, per degree of
freedom, the branch whose increment does not exceed $\pi$ in modulus therefore
yields the cochain $h^{0}$ of wrapped azimuthal increments with no appeal to
the orientation convention of the element. Since the axis lies in the hole, no
triangle encloses it, the three wrapped increments around any triangle sum to
zero, and $h^{0}$ is a discrete gradient on each tetrahedron. Its curl
vanishes elementwise, not to a projection tolerance: at $n_{r}=4$ the
quadratic form $\ip{\curl h^{0}}{\curl h^{0}}$ evaluates to
$-3.9\times10^{-14}$, negative and hence zero to round-off. Its circulation
around a loop threading the hole is $2\pi$, so it represents the nontrivial
class. A Helmholtz solve removes the exact part, reducing
$\max_{\tau_{h}}|\ip{h}{\grad\tau_{h}}|/\norm{\grad\tau_{h}}$ from
$2.6\times10^{-3}$ to $1.2\times10^{-16}$. Normalizing to unit period gives
the discrete generator $h_{1,h}$.

\medskip
\noindent\textbf{The period map.}
By Corollary~\ref{cor:period-derham} the period at degree one is the $L^{2}$
pairing against any divergence-free field tangent to $\partial\Omega$ with
unit flux through a cross-section, and on the solid torus
$J=\hat e_{\varphi}/(\pi\rho^{2})$ serves. No mesh-resolved cycle is required.

For the discrete map to be a functional on cohomology it must annihilate
discrete gradients exactly, which is the discrete coclosedness hypothesis of
Proposition~\ref{prop:discrete-period}. The assembled vector does not satisfy
it, since the discrete boundary is not a surface to which the smooth $J$ is
tangent, and the period of a discrete gradient of unit norm is
$2.7\times10^{-3}$; left uncorrected this displaces the normalization by about
one percent. With $D_{0}$ the vertex-to-edge incidence, so that
$G=M_{u}D_{0}$, the requirement is $G^{\top}M_{u}^{-1}b_{J}=0$, and we enforce
it by solving $S\chi=G^{\top}M_{u}^{-1}b_{J}$ with $S=D_{0}^{\top}M_{u}D_{0}$
and replacing $b_{J}$ by $b_{J}-G\chi$. After this projection the leak is
$1.3\times10^{-18}$, and the Helmholtz projection of the generator changes the
period by exactly zero, which is the invariance of the functional under a
change of representative asserted in Proposition~\ref{prop:discrete-period}.
The defect $\norm{\Piop_{h}H_{h}-I}$ vanishes by construction at rank one.

\medskip
\noindent\textbf{Analytic reference.}
The field $\hat e_{\varphi}/r$ is curl-free and divergence-free on $\Omega$
and tangent to $\partial\Omega$, so it lies in $\HH^{1}$; since that space is
one-dimensional it spans it. Its circulation about a loop threading the hole
is $2\pi$, so the period-normalized generator is
\begin{equation}\label{eq:torus-h-exact}
  h_{1}=\frac{\hat e_{\varphi}}{2\pi r},\qquad \Piop h_{1}=1 .
\end{equation}
This is the harmonic form and not an approximation to it. Hence
$m:=\norm{h_{1}}^{2}$, which enters the optimality law, is known exactly:
using $\int_{0}^{2\pi}d\phi/(R+s\cos\phi)=2\pi/\sqrt{R^{2}-s^{2}}$ in the
torus coordinates $r=R+s\cos\phi$,
\begin{equation}\label{eq:m-exact}
  m=\frac{1}{4\pi^{2}}\int_{\Omega}\frac{dV}{r^{2}}
   =R-\sqrt{R^{2}-\rho^{2}}
   =4.019237886\times10^{-2}.
\end{equation}
We write $m_{h}=\norm{h_{1,h}}^{2}$ for its discrete counterpart.

\medskip
\noindent\textbf{The bordered system.}
The discrete state system carries a rank-one harmonic border,
\begin{equation}
  A\begin{pmatrix}\sigma\\u\\p\end{pmatrix}
  =\begin{pmatrix}0\\Cz\\0\end{pmatrix},
  \qquad
  A=\begin{pmatrix}
      M_{\sigma}&-G^{\top}&0\\ G&K&b_{h}\\ 0&b_{h}^{\top}&0
    \end{pmatrix},
  \qquad (b_{h})_{i}=\ip{h_{1,h}}{\psi_{i}},
\end{equation}
the last row enforcing $\ip{u}{h_{1,h}}=0$. Deleting the border leaves
$(0,h_{1,h})$ in the nullspace. At $n_{r}=2$, where a dense computation is
affordable, $\operatorname{cond}(A)=7.5\times10^{5}$, the ordinary
conditioning of a mixed saddle operator, while
$\operatorname{cond}(A_{0})=5.1\times10^{17}$ with smallest singular value
$4.4\times10^{-16}$ and numerical nullity one. The gap is twelve orders, and
the nullity agrees with the Betti number obtained combinatorially, so
$\dim\HH^{1}_{h}$ is determined twice by independent routes. The border is
required for well-posedness, and the requirement is set by the topology and
not by the objective.

\medskip
\noindent\textbf{Results.}
We take $w_{y}=w_{\sigma}=\alpha=\alpha_{\rm top}=1$, $g=1$, $c_{0}=0$, the
scalar target $r_{d}$ of \eqref{eq:lshape-targets}, and
\[
  y_{d}=\grad\psi+\gamma\,\frac{\hat e_{\varphi}}{r},
\]
with $\psi$ as in \eqref{eq:psi}, so that the gradient part reproduces the
state target of Section~\ref{sec:lshape}. That part is cohomologically
trivial, and exactly so rather than to discretization error: interpolating
$\grad\psi$ into $\mathcal{P}^{-}_{1}\Lambda^{1}$ gives the discrete gradient
of the $\mathcal{P}_{1}$ interpolant of $\psi$, which the period functional
annihilates by construction, and the computed $d_{1}$ at $\gamma=0$ is
$2.3\times10^{-18}$. The second term supplies a class, and by
\eqref{eq:torus-h-exact} its content is $d_{1}=2\pi\gamma m$ up to the
geometric error. Table~\ref{tab:torus-runs} reports three configurations at
$n_{r}=4$, a mesh of $14\,400$ cells with $3\,050$ scalar and $18\,650$ vector
state degrees of freedom. In each of them Proposition~\ref{prop:balance} is
reproduced to relative error $1.1\times10^{-16}$ or better against the
computed $m_{h}$.

\begin{table}[htbp]
\centering
\begin{tabular}{clllll}
\toprule
 & $\gamma$ & $w_{\Pi}$ & $d_{1}$ & $c$ & $\widehat J$\\
\midrule
A & $0$   & $1$ & $2.3\times10^{-18}$ & $0.14698568$
  & $2.358198\times10^{-2}$\\
B & $0.2$ & $1$ & $5.086\times10^{-2}$ & $0.17190676$
  & $4.701849\times10^{-2}$\\
C & $0.2$ & $0$ & $5.086\times10^{-2}$ & $0.04886029$
  & $3.093384\times10^{-2}$\\
\bottomrule
\end{tabular}
\caption{Three configurations on the solid torus at $n_{r}=4$ with $O=I$ and
$\pi_{d}=0.30$ where present. Conjugate gradients converge in four iterations
throughout.}
\label{tab:torus-runs}
\end{table}

Configuration~C is the substantive one, and it is the only place in this
section where a claim of the formulation is tested rather than restated. The
objective contains no period target, no reference to topology and no term in
$p$, yet the optimal state carries circulation $0.04886029$ about the hole,
matching $w_{y}d_{1}/(\alpha_{\rm top}+w_{y}m_{h})$ to sixteen digits. The
tracking term asks the physical state to resemble a field carrying
circulation, the gauge-fixed component $u$ is constrained orthogonal to
$\HH^{1}$ and can supply none, and only $h$ can respond. This is what the
splitting \eqref{eq:intro-split} is for, and it is unavailable in a
formulation placing the harmonic slot at $p$, where by
Proposition~\ref{prop:p-inert} the harmonic content of a target is unreachable
for every positive $\alpha$. Comparing A with B isolates the same effect
against a period target, the circulation rising by $16.95\%$ when the target
acquires its class. In all three runs $f=0$ and the computed $p$ is of order
$10^{-16}$, including in configuration~C, which drives $h$ to a value of order
$10^{-1}$.

A fourth configuration exercises the coupling. Restricting the observation to
the half of the tube with $x<0.5$ gives $\kappa=5.735\times10^{-5}$ against
$10^{-22}$ under $O=I$, and Proposition~\ref{prop:balance} reproduces
$c=0.16103617$ to relative error $1.7\times10^{-16}$; the control grows from
$\max_{\Omega}|z|=2.9\times10^{-3}$ to $5.4\times10^{-2}$, since the
unobserved half may be sacrificed. The circulation moves by six percent
relative to configuration~B, so the observation window changes the topological
state appreciably. What remains small is the coupling itself: $\kappa$ is the
inner product of $Ou$ with $Oh_{1,h}$, and the plane $x=\tfrac12$ is a
symmetry plane of the tube, so restricting to one half nearly halves the
vanishing $\ip{u}{h_{1,h}}$ instead of exposing a generic misalignment. The
uncoupled law, which sets $\kappa=0$, remains accurate to $0.018\%$ for that
reason and not because the formulation decouples. A window that is not a
symmetry half would give a $\kappa$ of the order of the other terms.

\medskip
\noindent\textbf{Convergence.}
Table~\ref{tab:torus-refine} reports the refinement study. The period and the
harmonic norm converge to $\Piop h_{1}=1$ and \eqref{eq:m-exact} at second
order, inheriting their error from the volume deficit. The objective is
monotone with observed order $q=1.86$ and Richardson limit
$4.7299\times10^{-2}$, and the optimal circulation converges to $0.171801$ by
the same extrapolation. The balance law holds to round-off and the
compatibility multiplier stays at $10^{-16}$ at every level, neither depending
on the geometry, and the conjugate-gradient count is four throughout, so the
rank-one border costs nothing in the reduced solver.

\begin{table}[htbp]
\centering
\begin{tabular}{rrllllr}
\toprule
$n_{r}$ & cells & $\widehat J$ & $c$ & $m_{h}$
 & $|m_{h}/m-1|$ & $|\Piop(h)/2\pi-1|$\\
\midrule
$2$ & $1\,800$  & $4.627993\times10^{-2}$ & $0.17221344$
    & $4.356979\times10^{-2}$ & $8.40\times10^{-2}$
    & $5.01\times10^{-2}$\\
$3$ & $6\,156$  & $4.681984\times10^{-2}$ & $0.17198139$
    & $4.165250\times10^{-2}$ & $3.63\times10^{-2}$
    & $2.24\times10^{-2}$\\
$4$ & $14\,400$ & $4.701849\times10^{-2}$ & $0.17190676$
    & $4.101517\times10^{-2}$ & $2.05\times10^{-2}$
    & $1.27\times10^{-2}$\\
\midrule
\multicolumn{2}{l}{order} & $1.86$ & $1.86$ & & $2.07,\ 1.99$
    & $1.98,\ 1.98$\\
\multicolumn{2}{l}{limit} & $4.7299\times10^{-2}$ & $0.171801$
    & $4.019238\times10^{-2}$ & &\\
\bottomrule
\end{tabular}
\caption{Refinement on the boundary-fitted torus, configuration~B. The limit
for $m_{h}$ is the exact value \eqref{eq:m-exact}; those for $\widehat J$ and
$c$ are Richardson extrapolations at $q=1.86$.}
\label{tab:torus-refine}
\end{table}

The limit for $c$ is not merely an extrapolation. Setting $O=I$,
$w_{y}=w_{\Pi}=\alpha_{\rm top}=1$, $g=1$ and $c_{0}=0$ in
Proposition~\ref{prop:balance}, and using $d_{1}=2\pi\gamma m$ with the exact
value \eqref{eq:m-exact}, the optimality law predicts
\[
  c^{\star}=\frac{2\pi\gamma\,m+\pi_{d}}{2+m}
  =\frac{0.05050724+0.30}{2.04019238}
  =0.1718011 ,
\]
which the Richardson limit reproduces to seven digits. The balance law holds
to round-off on each mesh against the computed $m_{h}$, so the entire
discretization error in $c$ is inherited from $m_{h}\to m$, and the agreement
above confirms that inheritance at the extrapolated limit.

The solid torus also exhibits $\dim\HH^{k}=b_{k}$ across two degrees on one
mesh. Since $b_{2}=0$, the degree-two div--div problem here carries no
harmonic block: with $\sigma$ in $\mathcal{P}^{-}_{1}\Lambda^{1}$ and $u$ in
the Raviart--Thomas space $\mathcal{P}^{-}_{1}\Lambda^{2}$, the unbordered
state operator is nonsingular and solves to machine precision. The border thus
has rank one at degree one and rank zero at degree two on the same
triangulation. The spherical shell of Section~\ref{sec:shell} is the mirror
image.

\subsection{Figure eight: two circulation modes}\label{sec:fig8}

On the slab with two square holes of Section~\ref{ex:dtorus} we have
$b_{1}=2$, and we solve
\eqref{eq:pdeconlaplace}--\eqref{eq:pdeconlaplace:constraints} with a
two-dimensional topological actuator. The finite element pair, the control
space and the reduced solver are those of Section~\ref{sec:lshape}, with
$B=I$, $O=R=I$ and $f=0$; the harmonic border is that of
Section~\ref{sec:torus}, now of rank two. The mesh is carved from a uniform
triangulation of the slab with $N$ cells across each horizontal direction and
$N/4$ through the thickness, and since every corner of $\Omega$ is a multiple
of $\tfrac18$, the carve is exact whenever $N$ is a multiple of eight.

Three things appear that rank one cannot show. The period matrix has
off-diagonal entries that no normalization forces. The harmonic basis is not
orthogonal, so the balance law is a linear system coupling the two classes.
And the actuator matrix can be rank deficient, which is the only setting in
which Proposition~\ref{prop:reachability} says anything.

\medskip
\noindent\textbf{Topology.}
The computation of Section~\ref{sec:betti} returns, at $N=16$, $n_{V}=1355$,
$n_{E}=7564$, $n_{F}=11584$ and $n_{T}=5376$, hence $\chi=-1$ with one
connected component and one boundary surface, giving $b_{1}=2$. The value
$\chi=-1$ against $0$ for the solid torus and $1$ for the L-shaped domain is
the Euler characteristic counting tunnels. The manifold test matters more here
than on the previous two domains: two lobes joined at a single vertex satisfy
the facet condition without being a manifold, so Alexander duality would not
apply, and the same configuration is fatal for the discretization, since
N\'ed\'elec degrees of freedom live on edges and a vertex-connected mesh
decouples the discrete $H(\curl)$ spaces of the two lobes. All vertex and edge
stars are connected at every refinement level.

\medskip
\noindent\textbf{Two generators.}
The construction of Section~\ref{sec:torus} is applied once per hole, giving
cochains $h^{0}_{i}$ of wrapped increments of the azimuth $\theta_{i}$ about
the axis through $(c_{i},\tfrac12)$. Each axis lies in the hole of its own
square and outside the material of the other, so both are discretely closed on
all of $\Omega$: the quadratic forms $\ip{\curl h^{0}_{i}}{\curl h^{0}_{i}}$
are $8.6\times10^{-13}$ and $8.0\times10^{-13}$ at $N=16$, zero to round-off.
Helmholtz projection leaves
$\max_{\tau_{h}}|\ip{h_{i}}{\grad\tau_{h}}|/\norm{\grad\tau_{h}}$ at
$3.2\times10^{-14}$ and $7.4\times10^{-15}$.

\medskip
\noindent\textbf{An exact period map.}
On the torus the period functional was built from a smooth divergence-free
field whose tangency to the discrete boundary was only approximate, and it
carried a geometric error of order $10^{-2}$ before projection. On an extruded
domain it can be made exact. Taking $a_{i}$ discretely harmonic, equal to one
on the lateral boundary of hole $i$ and zero on the other hole and the outer
boundary, the field $J_{i}=\curl(a_{i}\hat e_{z})$ satisfies the hypotheses of
Corollary~\ref{cor:period-derham} at the discrete level as well, and
\[
  \Piop_{i}(v)=\tfrac1H\ip{v}{J_{i}}
\]
is the circulation about hole $i$ with the loops oriented by
$t=\hat e_{z}\times n$. As on the torus the assembled functional must be
Leray-projected before it satisfies the discrete coclosedness hypothesis of
Proposition~\ref{prop:discrete-period}; after projection the period of a
discrete gradient of unit norm is $2.1\times10^{-16}$ and
$1.5\times10^{-16}$, against $10^{-2}$ before.

Two checks confirm the construction. The raw period matrix of the two
generators, before normalization and in units of $2\pi$, is the identity to
six decimals, with off-diagonals below $5\times10^{-7}$ and condition number
$1.0000$, so $\theta_{1}$ winds once about hole one and not at all about hole
two. Nothing forces this, since at rank two the off-diagonals are not fixed by
the normalization. Second, replacing the harmonic potentials $a_{i}$ by
solutions of $(-\Delta+25)a_{i}=0$ with the same boundary constants, which are
entirely different functions, leaves the matrix unchanged to
$2.5\times10^{-16}$. This is the invariance of the functional under a change
of representative asserted in Proposition~\ref{prop:discrete-period}, and it
is the reason harmonicity of $a_{i}$ is not among the hypotheses. After
normalization $\norm{\Piop_{h}H_{h}-I}=2.5\times10^{-16}$.

\medskip
\noindent\textbf{A non-orthogonal basis and the balance law.}
With the period-normalized basis the Gram matrix at $N=16$ is
\[
  M=\bigl[\ip{h_{i}}{h_{j}}\bigr]
  =\begin{pmatrix}
     3.960250\times10^{-2} & 8.023\times10^{-3}\\
     8.023\times10^{-3} & 3.960250\times10^{-2}
   \end{pmatrix},
\]
symmetric with equal diagonals as the reflection $x\mapsto1-x$ requires, and
with an off-diagonal that is $20.3\%$ of the diagonal: on a slab each harmonic
field spreads across the cross-section instead of localizing near its own
hole. The basis cannot be orthogonalized without destroying the period
normalization, and it is the period coordinates that carry the physical
meaning. Proposition~\ref{prop:balance} therefore appears here in its full
form, with $O=I$ and $c=Ga+c_{0}$,
\begin{equation}\label{eq:fig8-balance}
  \Bigl[\alpha_{\rm top}I+G^{\top}\bigl(w_{y}M+w_{\Pi}I\bigr)G\Bigr]a
  =G^{\top}\Bigl(w_{y}d+w_{\Pi}\pi_{d}
      -\bigl(w_{y}M+w_{\Pi}I\bigr)c_{0}\Bigr),
  \qquad d_{i}=\ip{y_{d}}{h_{i}} .
\end{equation}
A componentwise reduction is valid only for an $L^{2}$-orthogonal harmonic
basis, which a geometry of two well-separated lobes happens to supply; here it
gives $(0.166544,-0.117516)$ against the correct $(0.167010,-0.118173)$,
errors of $0.28\%$ and $0.56\%$.

\medskip
\noindent\textbf{Reachability.}
All weights are one, $c_{0}=0$, and
$y_{d}=\grad\psi+\gamma\hat e_{\varphi,1}/r_{1}-\gamma\hat e_{\varphi,2}/r_{2}$
with $\gamma=0.2$, so the target is antisymmetric under $x\mapsto1-x$ and
$d=(0.039686,-0.039686)$ at $N=16$. The period content of that target is
recovered exactly: $M^{-1}d$ equals
$(2\pi\gamma,-2\pi\gamma)=(1.256637,-1.256637)$ to six digits, so the harmonic
projection preserves the periods of the two circulating fields, which is the
property required of a functional on cohomology and a sharper check than
antisymmetry alone. Table~\ref{tab:fig8-runs} reports four configurations;
\eqref{eq:fig8-balance} reproduces every computed $c$ to relative error
$6.4\times10^{-15}$ or better.

\begin{table}[htbp]
\centering
\begin{tabular}{clllll}
\toprule
 & $G$ & $\pi_{d}$ & $a$ & $c$ & $\widehat J$\\
\midrule
A & $I_{2}$ & $(0.30,-0.20)$ & $(0.167010,-0.118173)$
  & $(0.167010,-0.118173)$ & $8.83937\times10^{-2}$\\
B & $(1,1)^{\top}$ & $(0.30,-0.20)$ & $0.0323076$
  & $(0.032308,\ 0.032308)$ & $1.29306\times10^{-1}$\\
C & $(1,-1)^{\top}$ & $(0.25,-0.25)$ & $0.1891422$
  & $(0.189142,-0.189142)$ & $7.36296\times10^{-2}$\\
D & $(1,1)^{\top}$ & $(0.25,-0.25)$ & $4.30\times10^{-9}$
  & $(4.3\times10^{-9},\ 4.3\times10^{-9})$
  & $1.28421\times10^{-1}$\\
\bottomrule
\end{tabular}
\caption{Four actuator configurations at $N=16$ with $O=I$. Runs C and D share
an actuator rank and a target, differing only in the alignment of
$\range(G)$.}
\label{tab:fig8-runs}
\end{table}

Run~D is the sharpest form of Proposition~\ref{prop:reachability}. Both $d$
and $\pi_{d}$ are antisymmetric while $G=(1,1)^{\top}$ spans the symmetric
direction, so the right-hand side of \eqref{eq:fig8-balance} vanishes in the
continuum and the optimal topological control is zero: nothing the actuator
can do improves the objective, and the period misfit $6.250\times10^{-2}$ is
the full cost of the target. The computed $4.30\times10^{-9}$ is stable to
seven digits as the conjugate-gradient tolerance is tightened from $10^{-10}$
to $10^{-14}$, so it is a genuine stationary point of the discrete problem. It
is nonzero because the discrete antisymmetry of the target is not exact: the
residual $d_{1}+d_{2}$ is $1.33\times10^{-8}$ at $N=16$ and converges at order
$4.2$, while the Gram defect $M_{11}-M_{22}$ stays at $10^{-17}$, so the mesh
respects the reflection exactly and the defect lies in the interpolation of
the target, at a rate faster than anything else in the problem. Dividing it by
$\alpha_{\rm top}+G^{\top}(w_{y}M+w_{\Pi}I)G=3.09525$ reproduces the computed
$a$ to six digits.

Comparing C with D, an actuator of the same rank recovers $76\%$ of the target
in one case and nothing in the other. Comparing B with D separates rank from
alignment: writing $\pi_{d}=(0.30,-0.20)$ as $0.05\,(1,1)+0.25\,(1,-1)$ in
run~B, the achieved $0.0323$ is $65\%$ of the reachable component while the
unreachable component is untouched. We note also that
$\max_{\Omega}|z|=1.045716\times10^{-1}$ in all four runs, agreeing to seven
digits: with $O=I$ the reduced Hessian \eqref{eq:reduced-hessian} is block
diagonal and the distributed control is blind to $G$, $\pi_{d}$ and $w_{\Pi}$.
The compatibility multiplier is of order $10^{-15}$ throughout, as
Proposition~\ref{prop:p-inert} requires, and the state is orthogonal to both
harmonic directions to $5\times10^{-16}$.

\medskip
\noindent\textbf{Convergence and the reentrant corners.}
Table~\ref{tab:fig8-refine} reports configuration~A on four meshes, with the
geometry exact at every level. Five quantities converge at the same rate:
$M_{11}$ at $1.31$, $M_{12}$ at $1.18$, $c_{1}$ and $c_{2}$ at $1.33$, and
$d_{1}$ at $1.33$. The value is $4/3$ and it is set by the geometry. Each hole
has four corners of interior angle $3\pi/2$, so the singular exponent is
$\lambda=\pi/\omega=\tfrac23$; the harmonic potential behaves like
$r^{\lambda}$ there and the harmonic $1$-form, being its gradient, like
$r^{-1/3}$, square integrable and not in $H^{1}$, which is precisely the
regularity for which the mixed formulation exists. Quantities quadratic in the
field, which is what $M$, $d$ and through \eqref{eq:fig8-balance} the period
coordinates are, converge at $O(h^{2\lambda})=O(h^{4/3})$. The exponent is
confirmed directly: the energy of $h_{1}$ inside a cylinder of radius $R$
about a corner edge, over radii from $0.64h$ to $2.9h$ on the $N=32$ mesh,
fits $E(R)\sim R^{1.330}$ against $R^{2}$ for a bounded field, giving a field
exponent of $-0.335$.

\begin{table}[htbp]
\centering
\begin{tabular}{rrlllr}
\toprule
$N$ & cells & $\widehat J$ & $c_{1}$ & $M_{11}$ & CG\\
\midrule
 $8$ &    $672$ & $8.947738\times10^{-2}$ & $0.168091$
     & $4.1854\times10^{-2}$ & $7$\\
$16$ & $5\,376$ & $8.839372\times10^{-2}$ & $0.167010$
     & $3.9602\times10^{-2}$ & $7$\\
$24$ & $18\,144$ & $8.810166\times10^{-2}$ & $0.166712$
     & $3.8980\times10^{-2}$ & $7$\\
$32$ & $43\,008$ & $8.801430\times10^{-2}$ & $0.166580$
     & $3.8701\times10^{-2}$ & $7$\\
\midrule
\multicolumn{2}{l}{order} & --- & $1.33$ & $1.31$ & \\
\multicolumn{2}{l}{limit} & --- & $0.166297$ & $3.8089\times10^{-2}$ & \\
\bottomrule
\end{tabular}
\caption{Refinement on the figure-eight domain, configuration~A. Each limit is
extrapolated at its own observed order, $1.33$ for $c_{1}$ and $1.31$ for
$M_{11}$. No order is quoted for $\widehat J$; see the text.}
\label{tab:fig8-refine}
\end{table}

The contrast with Section~\ref{sec:lshape} is worth drawing, since that domain
has three reentrant edges of the same opening and converges at order $1.97$. A
reentrant feature does not by itself produce a singular field; it produces one
when the field is compelled to admit the singular mode. On the L-shaped domain
nothing compels it, the state solving the mixed system with the natural
condition $u\cdot n=0$ under smooth data, and the same energy probe applied
there returns $E(R)\sim R^{2.087}$. Here the harmonic field is not driven by
data at all: it is fixed by the cohomology, and the harmonic $1$-form of a
planar domain with a reentrant corner is the gradient of a potential behaving
like $r^{\lambda}$, with no freedom in the matter. Topology forces a
singularity that boundary conditions and smooth data do not, and the
topological subsystem inherits its rate while the harmonic dimension, the
period matrix and the balance law remain exact throughout.

No order is quoted for the total objective. Its four components move in
opposite directions under refinement: the state tracking term falls through
$5.588$, $5.486$, $5.459$, $5.451$ in units of $10^{-2}$, the period misfit
rises through $1.196$, $1.219$, $1.226$, $1.229$ in the same units, and the
topological regularization $\tfrac{\alpha_{\rm top}}{2}|a|^{2}$, which is
$2.09\times10^{-2}$ at $N=16$ and larger than the period misfit, falls with
$c$. Partial cancellation between sequences of opposite sign corrupts any
fitted rate, successive triples returning $1.37$ and $2.43$.

\subsection{Spherical shell: cavity-flux control}
\label{sec:shell-numerics}



The preceding examples realize nontrivial topology at degree $k=1$,
where the topological observables are circulations along
noncontractible loops. We now take an example at degree $k=2$, to show
that the same control mechanism applies to fluxes through nontrivial
two-cycles. The domain is the spherical shell  described in \ref{ex:Ball}
\[
  \Omega=\bigl\{\,x\in\R^3:\ r_{\rm in}<|x-x_0|<r_{\rm out}\,\bigr\},
  \qquad x_0=(\tfrac12,\tfrac12,\tfrac12),
\]
with $b_0=1$, $b_1=0$, $b_2=1$, so that $\dim\HH^2=1$.

 
\noindent\textbf{Mesh.}
Carving the shell from a background triangulation gives a staircase
boundary whose volume oscillates under refinement, which destroys any
convergence statement. We build instead a boundary-fitted mesh, as for
the torus one dimension over: an icosahedron subdivided $n_{\rm sub}$
times and projected to the sphere, extruded radially through $n_r$
layers, each prism split into three tetrahedra by a rule keyed to the
sorted vertex indices of its base triangle. That rule is the same in
every layer, so the diagonal chosen on each shared quadrilateral face
agrees between neighbours and the mesh is conforming; both spherical
triangulations appear as unions of mesh faces. Every boundary vertex
lies exactly on one of the two spheres and the radial edges are exact
segments along rays, so the geometric error is the inscribed-polyhedron
error in the angular directions alone, independent of $n_r$: the volume
deficits at $n_{\rm sub}=1,\dots,4$ are $12.65\%$, $3.384\%$,
$0.861\%$ and $0.216\%$, with successive ratios $3.74$, $3.93$ and
$3.98$ approaching four. Figure~\ref{fig:shell} shows the mesh with one
octant removed. The computation of Section~\ref{sec:betti} returns
$\chi=2$ with one connected component, two boundary surfaces and all
vertex and edge stars connected, hence $b_0=1$, $b_1=0$, $b_2=1$ at
every level.
 
Refinement must act in both directions. The harmonic $2$-form
\eqref{eq:shell-h2-closed} varies only radially, by a factor
$(r_{\rm out}/r_{\rm in})^2=6.5$ across the wall, and subdividing the
sphere while advancing $n_r$ through $2,3,4$ leaves the radial spacing
sixty percent coarser than the angular one by the third level; the
observed order in $m_h$ then falls to $1.24$ and drifts. Taking $n_r$
proportional to $2^{n_{\rm sub}}$ restores second order, and the
sequence below does so.
 
\medskip
\noindent\textbf{The discrete generator.}
At $n_{\rm sub}=3$, $n_r=8$ the divergence falls from
$1.08\times10^1$ to $-7.6\times10^{-13}$ under the Leray projection,
and $\max_{\tau_h}|\ip{h}{\curl\tau_h}|$ from $1.2\times10^{-2}$ to
$1.2\times10^{-15}$ under the removal of the exact part, so
\eqref{eq:shell-harmonic-diagnostics} holds to round-off. The period of
a random discrete curl is $4.3\times10^{-17}$ after the functional is
projected, and $|\Piop_h h_{2,h}-1|$ is at round-off by construction.
 
The construction does not depend on the field it starts from, which is
the check that it reaches $\HH^2_h$ rather than merely approximating
it. Starting from $(x-x_0)/|x-x_0|^2$, or from the linear field
$x-x_0$ whose divergence is three, returns the same normalized
generator to $1.5\times10^{-14}$ and the same $m_h$ to ten digits: the
two projections land in the one-dimensional discrete harmonic space
exactly and the period normalization fixes the scale. The ratio
$\Piop(h)/4\pi$, the discrete flux of the projected radial field before
normalization, takes the values $1.004856$, $0.999338$ and $0.999692$
on the three meshes below.
 
\medskip
\noindent\textbf{Results.}
We take $w_y=w_\sigma=\alpha=\alpha_{\rm top}=1$, $g=1$, $c_0=0$, and
the target \eqref{eq:shell-manufactured-target} with
$u_d=\curl A$, $A=(0,0,\ 0.1\sin\pi x\,\sin\pi y)$. The computed
$d=\ip{y_d}{h_2}$ at $\gamma=0$ is $-2.19\times10^{-7}$. This is small
but not round-off, unlike the corresponding quantity on the torus:
there the trivial part was a gradient and the N\'ed\'elec interpolant
of a gradient is exactly the discrete gradient of the $\mathcal P_1$
interpolant, whereas here $\curl A$ is interpolated into
Raviart--Thomas directly rather than through the canonical N\'ed\'elec
interpolant of $A$, so a small component outside the range of the
discrete curl survives; it displaces the achieved flux in the sixth
digit. At $\gamma=0.3$ the computed $d$ is $8.192619\times10^{-2}$
against $\gamma m_h=8.195167\times10^{-2}$, agreeing to
$3.1\times10^{-4}$.
 
Table~\ref{tab:shell-runs} reports three configurations at
$n_{\rm sub}=3$, $n_r=8$: a mesh of $30\,720$ cells with $37\,776$
N\'ed\'elec and $62\,720$ Raviart--Thomas state degrees of freedom and
$92\,160$ control degrees of freedom. Equation
\eqref{eq:shell-exact-flux} is reproduced to relative error
$3.3\times10^{-16}$ or better in all three, conjugate gradients
converge in four iterations, the compatibility multiplier is of order
$10^{-15}$ as Proposition~\ref{prop:p-inert} requires, and the
orthogonality defect $|\ip{u_h}{h_{2,h}}|/\norm{u_h}$ is at most
$10^{-14}$.
 
\begin{table}[htbp]
\centering
\begin{tabular}{clllll}
\toprule
 & $\gamma$ & $w_\Pi$ & $d$ & $\Phi_h$ & $\widehat J$\\
\midrule
A & $0$   & $1$ & $-2.193\times10^{-7}$ & $0.13197406$
  & $3.531757\times10^{-2}$\\
B & $0.3$ & $1$ & $\phantom{-}8.192619\times10^{-2}$ & $0.16801463$
  & $3.531530\times10^{-2}$\\
C & $0.3$ & $0$ & $\phantom{-}8.192619\times10^{-2}$ & $0.06434808$
  & $1.976400\times10^{-2}$\\
\bottomrule
\end{tabular}
\caption{Three configurations on the spherical shell at
$n_{\rm sub}=3$, $n_r=8$, with $O=I$ and $\Phi_d=0.30$ where present.
In A the target carries no cohomology and the achieved flux is set by
$\Phi_d$ alone; in B the two drives compete; in C there is no flux
target at all.}
\label{tab:shell-runs}
\end{table}
 
Configuration C is the substantive one. The objective contains no flux
target, no reference to topology and no term in $p$, yet the optimal
state carries a cavity flux of $6.434808\times10^{-2}$, matching
$w_y d/(\alpha_{\rm top}+w_y m_h)$ to six digits. The tracking term
asks the physical state to resemble a field carrying flux, the
gauge-fixed component satisfies $u\perp\HH^2$ and can supply none, and
only $h$ can respond. By Proposition~\ref{prop:p-inert} this is
unreachable in a formulation that assigns the tracking target to $p$.
Comparing A with B isolates the same effect against a flux target, the
achieved flux rising by $27.31\%$ when the target acquires its
cohomology class. In all three runs
$\max_\Omega|z|=4.32\times10^{-3}$, agreeing to three digits between B
and C: with $O=I$ the reduced Hessian is block diagonal and the
distributed control is blind to $w_\Pi$ and $\Phi_d$.
 
\medskip
\noindent\textbf{Conditioning, and the degree contrast.}
Omitting the harmonic border leaves $(0,h_{2,h})$ in the nullspace of
the degree-two state operator. On the coarsest mesh, where a dense
computation is affordable, the bordered operator has condition number
$6.40\times10^5$ while the unbordered one has $9.56\times10^{18}$ with
smallest singular value $1.56\times10^{-15}$ and numerical nullity one,
a gap of thirteen orders; the nullity agrees with the Betti number
obtained combinatorially.
 
This yields the sharpest available statement of $\dim\HH^k=b_k$. On
this shell $b_1=0$, so the curl--curl problem needs no border at all
and its unbordered operator is nonsingular, while the div--div problem
needs a rank-one border. The solid torus of Section~\ref{sec:torus} is
the mirror image, carrying a border at degree one and none at degree
two. Between the two examples the harmonic content appears at exactly
one degree per domain, its dimension equal to the corresponding Betti
number in each case, and neither example alone can say this.

 \begin{figure}[htbp]
  \centering
  \includegraphics[width=\textwidth]{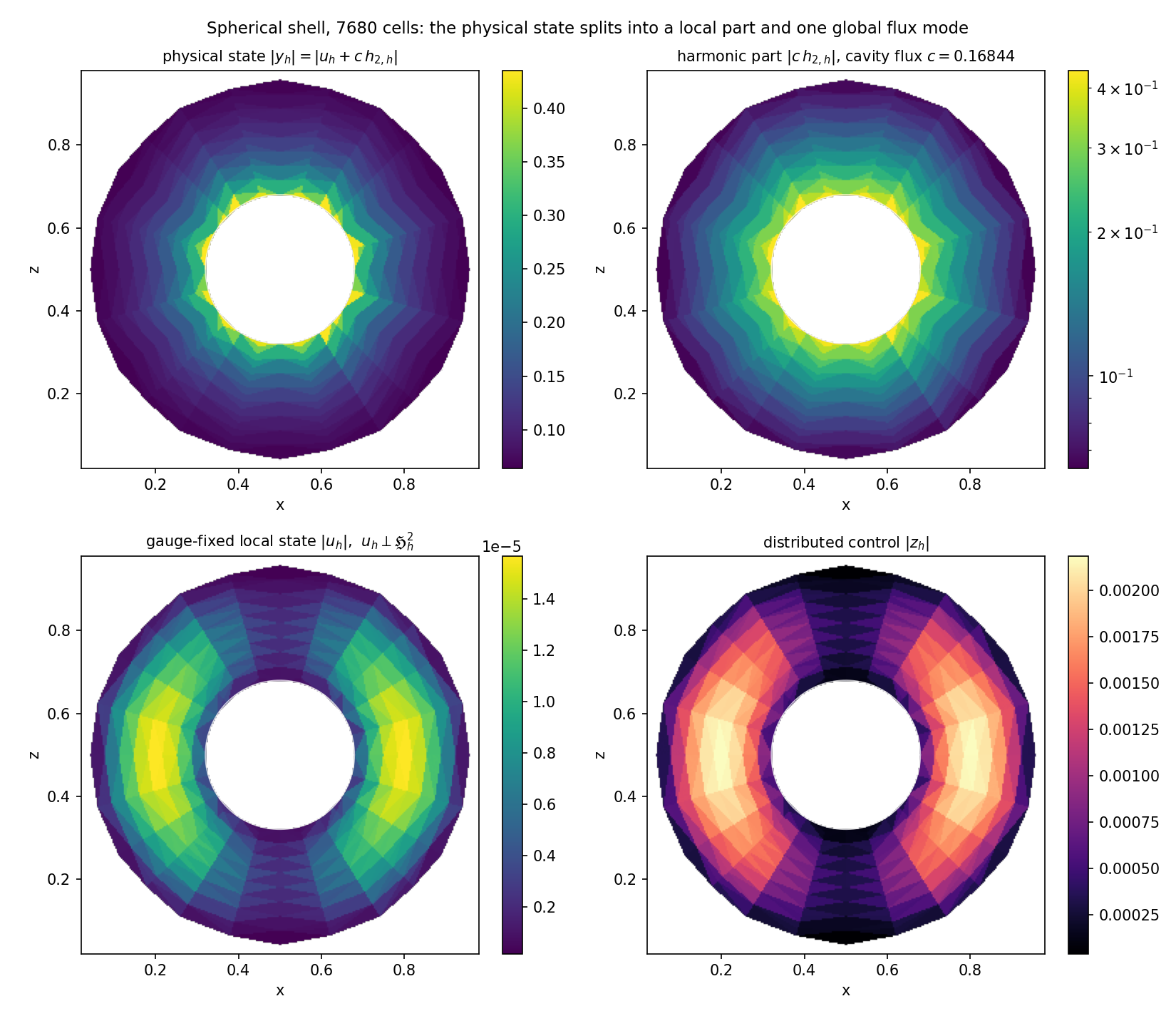}
  \caption{Spherical shell, $7\,680$ cells, meridional plane $y=1/2$.
  The physical state $y_h=u_h+c\,h_{2,h}$ is dominated by its harmonic
  component, which carries the entire cavity flux $c=0.16844$ and decays as
  $1/(4\pi r^2)$ across the wall; the gauge-fixed local state $u_h$, shown on
  its own scale, is four orders smaller because it is constrained orthogonal
  to $\mathfrak H^2_h$ and can supply no flux. The distributed control $z_h$
  acts only on $u_h$, so the two parts of the state respond to different
  controls.}
  \label{fig:shell-fields}
\end{figure}
\medskip
\noindent\textbf{Convergence.}
Table~\ref{tab:shell-refine} reports configuration B on three isotropic
levels. The harmonic norm converges to \eqref{eq:shell-m-closed} at
orders $2.05$ and $2.01$, and Richardson extrapolation at order two
gives $0.2690649$ against the exact $0.2691026$, a relative error of
$1.4\times10^{-4}$: the discrete harmonic space converges to the
continuous one at the expected rate, verified against an analytic value
rather than against itself. The achieved cavity flux converges at order
$2.16$ with limit $0.16781$. Equation \eqref{eq:shell-exact-flux} holds
to round-off against the computed $m_h$ at every level, the
compatibility multiplier stays at $10^{-15}$, and the
conjugate-gradient count does not grow, being $5$, $5$ and $4$. No
order is quoted for the objective, which mixes terms converging at
different rates as on the figure-eight domain of
Section~\ref{sec:fig8}; its single available triple gives $1.20$.
 
\begin{table}[htbp]
\centering
\begin{tabular}{rrrrllll}
\toprule
$n_{\rm sub}$ & $n_r$ & cells & dofs & $\widehat J$ & $\Phi_h$ & $m_h$
 & $|m_h/m-1|$\\
\midrule
$1$ & $2$ &    $480$ &   $1\,724$ & $3.495473\times10^{-2}$
    & $0.17183100$ & $3.369371\times10^{-1}$ & $2.52\times10^{-1}$\\
$2$ & $4$ &  $3\,840$ &  $12\,968$ & $3.520568\times10^{-2}$
    & $0.16871234$ & $2.854942\times10^{-1}$ & $6.09\times10^{-2}$\\
$3$ & $8$ & $30\,720$ & $100\,496$ & $3.531530\times10^{-2}$
    & $0.16801463$ & $2.731722\times10^{-1}$ & $1.51\times10^{-2}$\\
\midrule
\multicolumn{4}{l}{order} & --- & $2.16$ & $2.05,\ 2.01$ & \\
\multicolumn{4}{l}{limit} & --- & $0.16781$
    & $2.691026\times10^{-1}$ & \\
\bottomrule
\end{tabular}
\caption{Refinement on the boundary-fitted shell, configuration B, with
$n_r$ doubling at each level. The limit for $m_h$ is
\eqref{eq:shell-m-closed}; that for $\Phi_h$ is a Richardson
extrapolation.}
\label{tab:shell-refine}
\end{table}

\medskip
\noindent\textbf{Diagnostics.}
The cell and space counts are given with Table~\ref{tab:shell-runs};
$\dim\HH^2_h=1$ is obtained combinatorially from $\chi=2$ and
spectrally from the nullity of the unbordered operator; the harmonic
residuals \eqref{eq:shell-harmonic-diagnostics} are
$-7.6\times10^{-13}$, $1.2\times10^{-15}$ and round-off; the
orthogonality defect is $2.5\times10^{-15}$; the achieved flux is the
column $\Phi_h$; the agreement with \eqref{eq:shell-exact-flux} is
$3.3\times10^{-16}$; the reduced gradient and the Hessian action are
verified together in both control directions by the exact quadratic
Taylor identity ~\ref{eq:taylor}, to $3.3\times10^{-15}$;
and the reduced solver count is $5$, $5$, $4$ under refinement. We do
not report local state and adjoint errors: no exact solution of
\eqref{eq:shell-mixed} is available for this data, and the convergence
evidence is instead $m_h\to m$ and $\Phi_h$ against their limits.
 
The shell completes the pattern of the elliptic examples. On the solid
torus the nontrivial mode lies in $\HH^1$ and is a line circulation; on
the figure-eight domain there are two such modes; here it lies in
$\HH^2$ and is a surface flux. The same scalar optimality law governs
the torus and the shell, with only the value of $m$ changing, so the
formulation is tied to cohomology degree rather than to a particular
vector-calculus operator or geometric picture of a hole.


\section{Maxwell experiment I: re-entrant L-shaped domain}
\label{sec:maxwell-lshape-numerics}

The first time-dependent experiment is deliberately topologically trivial.  We reuse the L-shaped three-dimensional domain introduced in Section~\ref{sec:lshape}.  Here $b_1=b_2=0$, so there are no electric or magnetic harmonic coordinates and no topological actuators.  The purpose is to isolate the local Maxwell discretization, the effect of the re-entrant edges, the space--time convergence of the state and adjoint, and mesh independence of the reduced optimization method before introducing cohomology.

The computation uses the same compatible FEEC subcomplex as in
Section~\ref{s:feec} and a time integrator of the same formal order for the
forward and backward equations.  A manufactured local Maxwell solution is used,
which permits separate \(E\), \(B\), and control errors.  The table reports the
discrete charge-continuity and Gauss-law defects.  The manufactured data are
smooth and are not designed to excite the singular re-entrant edge mode
strongly; the local energy probe therefore documents bounded-field behaviour
near the re-entrant edge for this data set.

\begin{figure}[htbp]
\centering
\IfFileExists{maxwell_lshape_reentrant.pdf}{\includegraphics[width=.84\linewidth]{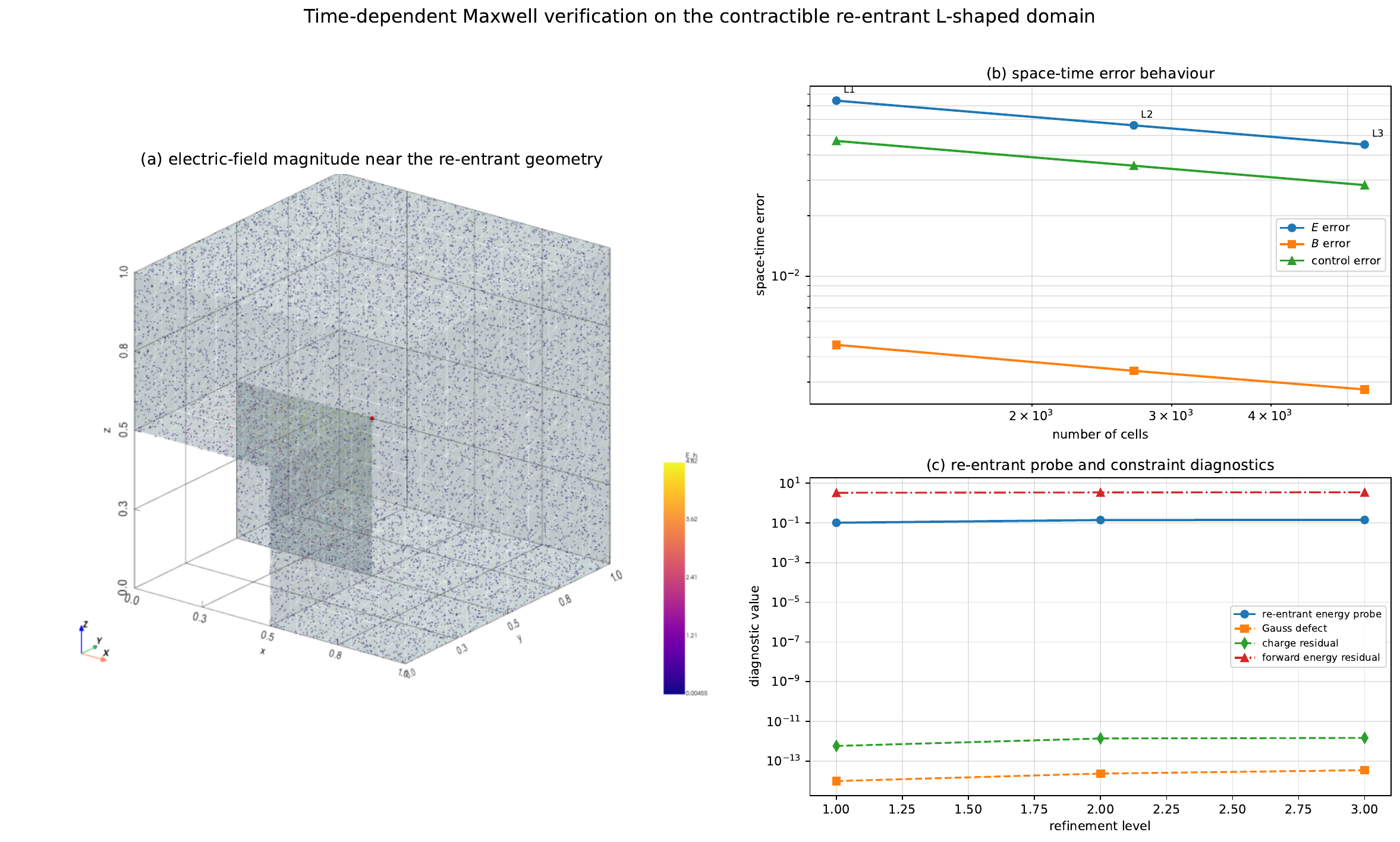}}{%
\fbox{\parbox[c][5.3cm][c]{.80\linewidth}{\centering
\textbf{Figure file unavailable in this compiled draft}\\[2mm]
\texttt{maxwell\_lshape\_reentrant.pdf} was not found.}}}
\caption{Time-dependent Maxwell computation on the contractible
L-shaped domain.  The field plot shows how the electric response is
distributed near the re-entrant geometry, while the accompanying error
and diagnostic panels show that the approximation errors decrease under
space--time refinement and that the Gauss-law and charge-continuity
defects remain small.  Since the domain has \(b_1=b_2=0\), the observed
behaviour is due to the local FEEC discretization and the re-entrant
geometry, not to harmonic or topological degrees of freedom.}
\label{fig:maxwell-lshape-reentrant}
\end{figure}

\begin{table}[htbp]
\centering
\small
\begin{tabular}{rrrrrrrr}
\toprule
level & cells & $\Delta t$ & $E$ error & $B$ error & control error & Gauss defect & $N_t$\\
\midrule
1 & 1134 & $2.50\times10^{-3}$ & $7.43\times10^{-2}$ & $4.58\times10^{-3}$ & $4.69\times10^{-2}$ & $9.89\times10^{-15}$ & 8\\
2 & 2688 & $1.67\times10^{-3}$ & $5.60\times10^{-2}$ & $3.41\times10^{-3}$ & $3.54\times10^{-2}$ & $2.36\times10^{-14}$ & 12\\
3 & 5250 & $1.25\times10^{-3}$ & $4.50\times10^{-2}$ & $2.76\times10^{-3}$ & $2.83\times10^{-2}$ & $3.49\times10^{-14}$ & 16\\
\bottomrule
\end{tabular}
\caption{Space--time convergence and Maxwell solver diagnostics on the
re-entrant L-shaped domain.  The decreasing \(E\), \(B\), and control
errors show the consistency of the local Maxwell discretization over
the refinement sequence, while the Gauss-law defect remains at
round-off level.  The final column reports the prescribed number of
time steps \(N_t=T/\Delta t\).}
\label{tab:maxwell-lshape-results}
\end{table}

The L-shaped experiment is the Maxwell analogue of Section~\ref{sec:lshape}: it provides a baseline in which any loss of regularity is geometric rather than cohomological. This baseline is useful for comparing the later multiply connected experiments, where additional harmonic and topological effects are present.

\section{Maxwell experiment II: solid-torus electric circulation}
\label{sec:maxwell-torus-numerics}

On the solid torus, $b_1=1$ and $b_2=0$.  We use a period-normalized electric harmonic field $h_{E,h}$ satisfying $\Piop_{E,h}h_{E,h}=1$.  The scalar $c_{E,h}(t)$ is therefore the electric circulation around the noncontractible loop.  We prescribe the smooth target
\(c_{E,d}(t)=c_{\max}\sin^2(\pi t/T)\) and compare two actuator configurations: a distributed current whose harmonic projection is nonzero, $C_{E,h}\neq0$, and a current control with $C_{E,h}=0$ supplemented by a scalar topological actuator $a_E$.  Table~\ref{tab:maxwell-torus-results} reports the circulation-tracking error, distributed and topological control energies, Gauss-law defect, and agreement between adjoint and finite-difference gradients.

\begin{figure}[htbp]
\centering
\IfFileExists{maxwell_torus_circulation.pdf}{\includegraphics[width=.82\linewidth]{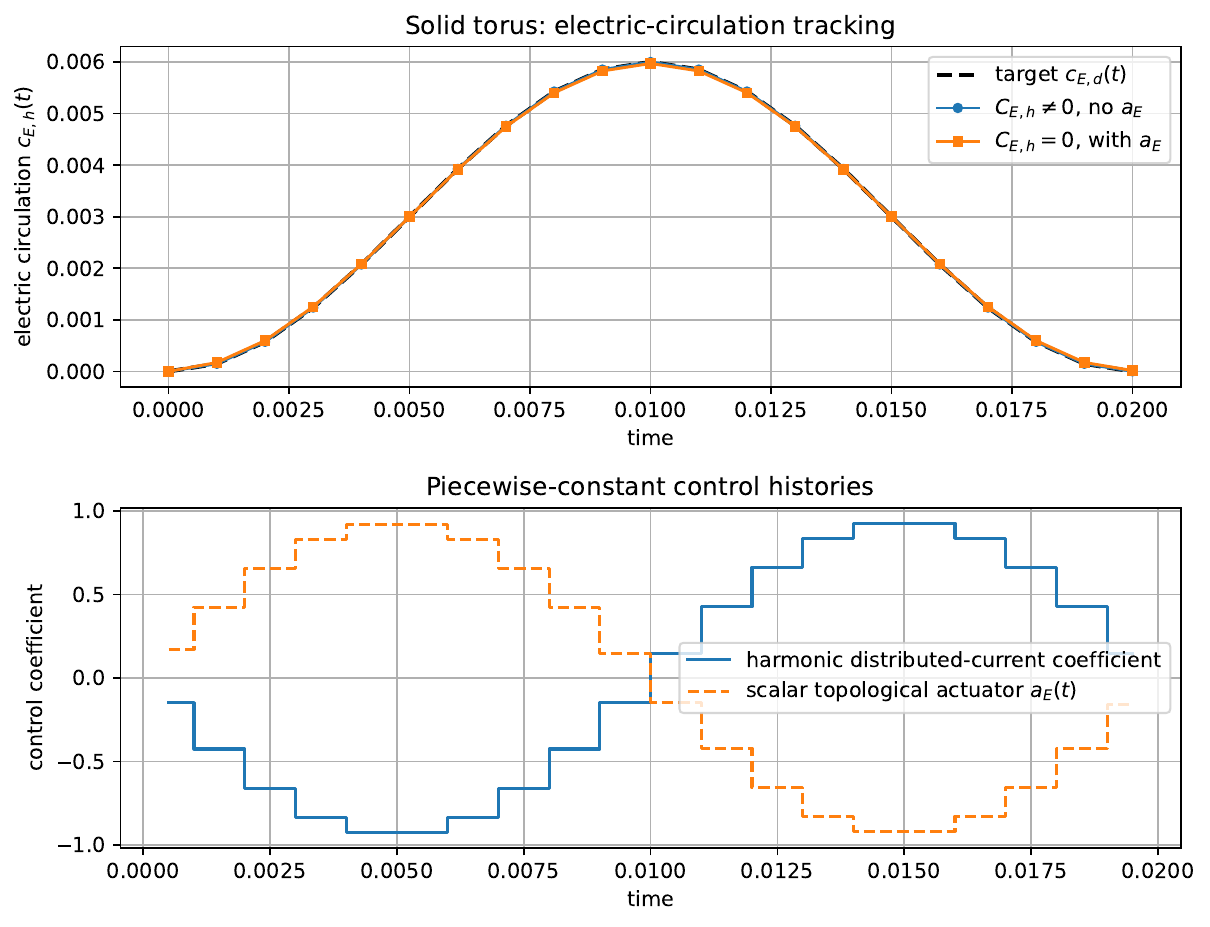}}{%
\fbox{\parbox[c][5.0cm][c]{.78\linewidth}{\centering
\textbf{Figure file unavailable in this compiled draft}\\[2mm]
\texttt{maxwell\_torus\_circulation.pdf} was not found.}}}
\caption{Time-dependent electric-circulation control on the solid torus.
The two runs compare actuation through a harmonic component of the
distributed current with a dedicated finite-dimensional topological
actuator.}
\label{fig:maxwell-torus-circulation}
\end{figure}

\begin{table}[htbp]
\centering
\small
\begin{tabular}{lrrrrr}
\toprule
configuration
& $\|c_E-c_{E,d}\|_{L^2(0,T)}$
& $\|z\|_{L^2(0,T;L^2)}$
& $\|a_E\|_{L^2(0,T)}$
& Gauss defect
& grad. check\\
\midrule
$C_{E,h}\neq0$, no $a_E$
& $1.13\times10^{-7}$
& $2.04\times10^{-2}$
& --
& $2.08\times10^{-17}$
& $1.23\times10^{-17}$\\
$C_{E,h}=0$, with $a_E$
& $2.76\times10^{-6}$
& $8.00\times10^{-3}$
& $9.30\times10^{-2}$
& $2.08\times10^{-17}$
& $9.59\times10^{-17}$\\
\bottomrule
\end{tabular}
\caption{Maxwell torus diagnostics for electric-circulation control.
The first configuration uses a distributed current with nonzero
electric harmonic projection and no scalar topological actuator.  The
second configuration removes the harmonic contribution from the
distributed current and uses the scalar topological actuator \(a_E\).
In both cases the local Maxwell block is active, so the electric field
is reconstructed as \(E_h=e_h+c_{E,h}h_{E,h}\).  The Gauss-law
diagnostic remains at round-off level, and the reduced-gradient checks
agree with finite differences to near machine precision.}
\label{tab:maxwell-torus-results}
\end{table}
This experiment tests the dynamical counterpart of the stationary rank-one torus problem: the topology contributes one physical electric-circulation state, while the manner in which it is actuated determines whether that state is reached through the harmonic component of the distributed current or through a dedicated finite-dimensional control.

\section{Maxwell experiment III: two-hole circulation reachability}
\label{sec:maxwell-fig8-numerics}

On the connected two-hole geometry, \(b_1=2\).  We construct
period-normalized fields \(h_{E,1,h}\) and \(h_{E,2,h}\) and compare a
full-rank actuator \(G_E=I_2\) with the rank-one actuator
\(G_E=(1,1)^T\).  With \(G_E=I_2\), the two loop-circulation coordinates are
independently reachable.  With the rank-one actuator, only the one-dimensional
subspace generated by \((1,1)^T\) is reachable, so the optimal trajectory is
the least-cost reachable projection of the desired two-mode trajectory.  This
is the time-dependent numerical test of Proposition~\ref{prop:maxwell-reachability}.

The experiment records the two modal tracking errors, the projected-target
residual, and the reduced-gradient finite-difference check.  The common
space--time refinement table in Section~\ref{sec:maxwell-refinement} reports
the corresponding period-normalization diagnostics for the two-hole geometry.
A fitted singular rate and a local re-entrant energy probe are not reported
for this time-dependent reachability run, so no stationary \(4/3\)-type
singular convergence rate is assumed here.

\begin{figure}[htbp]
\centering
\IfFileExists{maxwell_fig8_reachability.pdf}{\includegraphics[width=.82\linewidth]{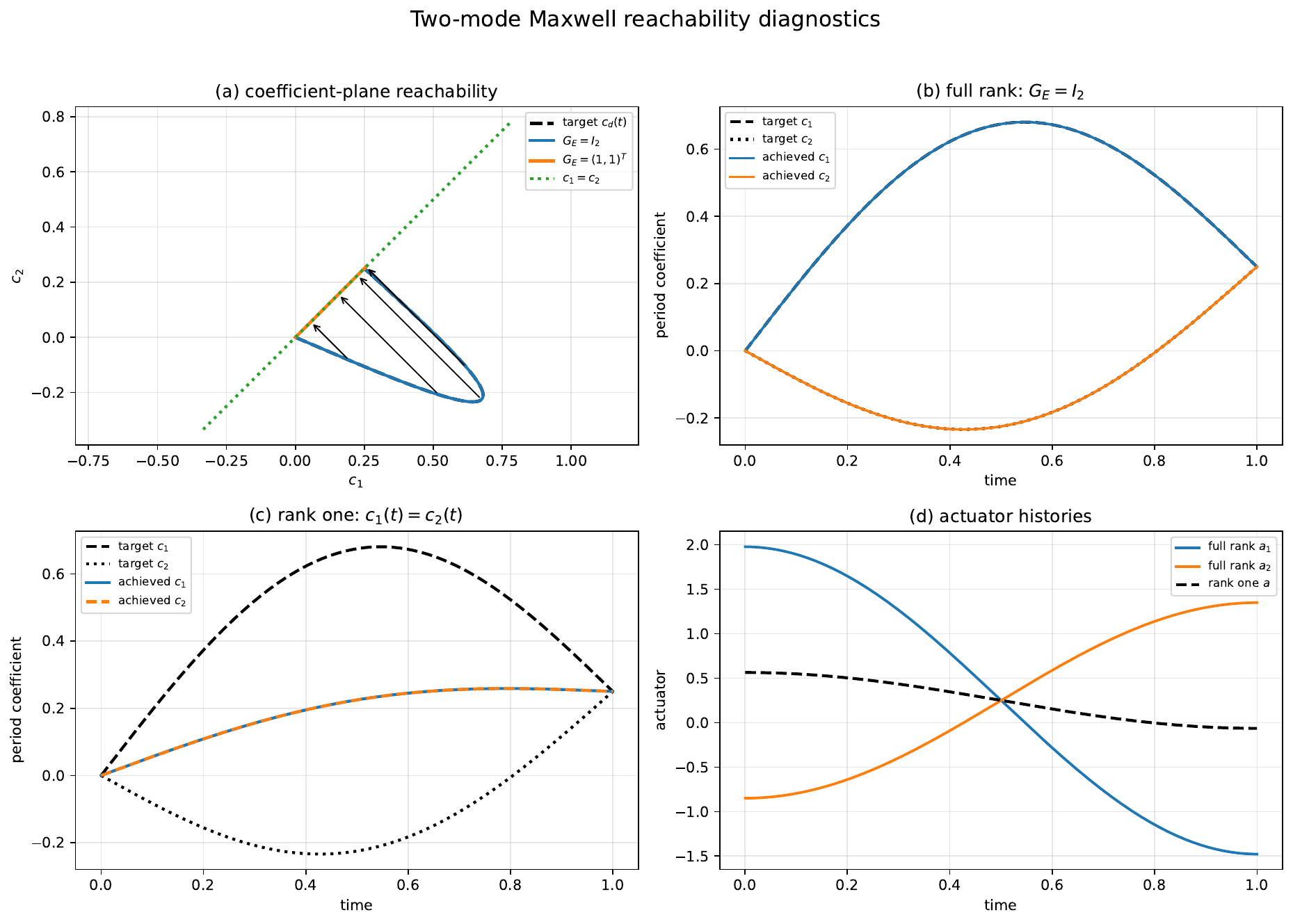}}{%
\fbox{\parbox[c][5.2cm][c]{.78\linewidth}{\centering
\textbf{Figure file unavailable in this compiled draft}\\[2mm]
\texttt{maxwell\_fig8\_reachability.pdf} was not found.}}}
\caption{Two-mode Maxwell reachability diagnostics. Panel (a) compares the desired trajectory in the harmonic coefficient plane with the full-rank and rank-one reachable paths. Panel (b) shows that the full-rank actuator \(G_E=I_2\) tracks both circulation coefficients, while panel (c) shows that the rank-one actuator \(G_E=(1,1)^T\) constrains the response to \(c_1(t)=c_2(t)\). Panel (d) compares the corresponding actuator histories.}
\label{fig:maxwell-fig8-reachability}
\end{figure}

\begin{table}[htbp]
\centering
\small
\begin{tabular}{lrrrrr}
\toprule
\(G_E\) & rank & mode-1 error & mode-2 error & projected-target residual & grad. check\\
\midrule
\(I_2\) & 2 & \(0\) & \(0\) & \(0\) & \(6.86\times 10^{-8}\)\\
\((1,1)^T\) & 1 & \(3.18\times 10^{-1}\) & \(3.18\times 10^{-1}\) & \(4.50\times 10^{-1}\) & \(3.73\times 10^{-7}\)\\
\bottomrule
\end{tabular}
\caption{Two-mode Maxwell reachability diagnostics. The full-rank actuator resolves both harmonic coordinates, while the rank-one actuator can only reach the one-dimensional subspace generated by \((1,1)^T\). Consequently, the rank-one case satisfies \(c_1(t)=c_2(t)\) and leaves a nonzero projected-target residual when the desired two-mode trajectory has unequal components.}
\label{tab:maxwell-fig8-results}
\end{table}
\FloatBarrier
The key comparison is between actuator rank and actuator alignment.  A rank-one actuator is not intrinsically ineffective: it can control one topological direction exactly, but it cannot alter the complementary harmonic coordinate.  This is the dynamical version of the stationary reachability experiment and directly exposes the role of \(\operatorname{range}(G_E)\).
\section{Maxwell experiment IV: spherical-shell magnetic flux}
\label{sec:maxwell-shell-numerics}

On the spherical shell, \(b_1=0\) and \(b_2=1\).  We normalize the magnetic
harmonic representative \(h_{B,h}\) by imposing
\(\Piop_{B,h}h_{B,h}=1\).  Two runs are considered.  In the conservation run,
\(G_B=0\), and the magnetic flux coefficient \(c_{B,h}(t)\) remains equal to
its initial value.  In the controlled run, \(G_B=1\), and the scalar
topological actuator \(a_B\) drives the magnetic flux toward the prescribed
terminal target \(c_{B,T}=0.65\).  Thus the two runs separate the uncontrolled
magnetic-flux invariant from the controlled flux state introduced by the
global actuator.

\begin{figure}[htbp]
\centering
\IfFileExists{maxwell_shell_flux.pdf}{\includegraphics[width=.82\linewidth]{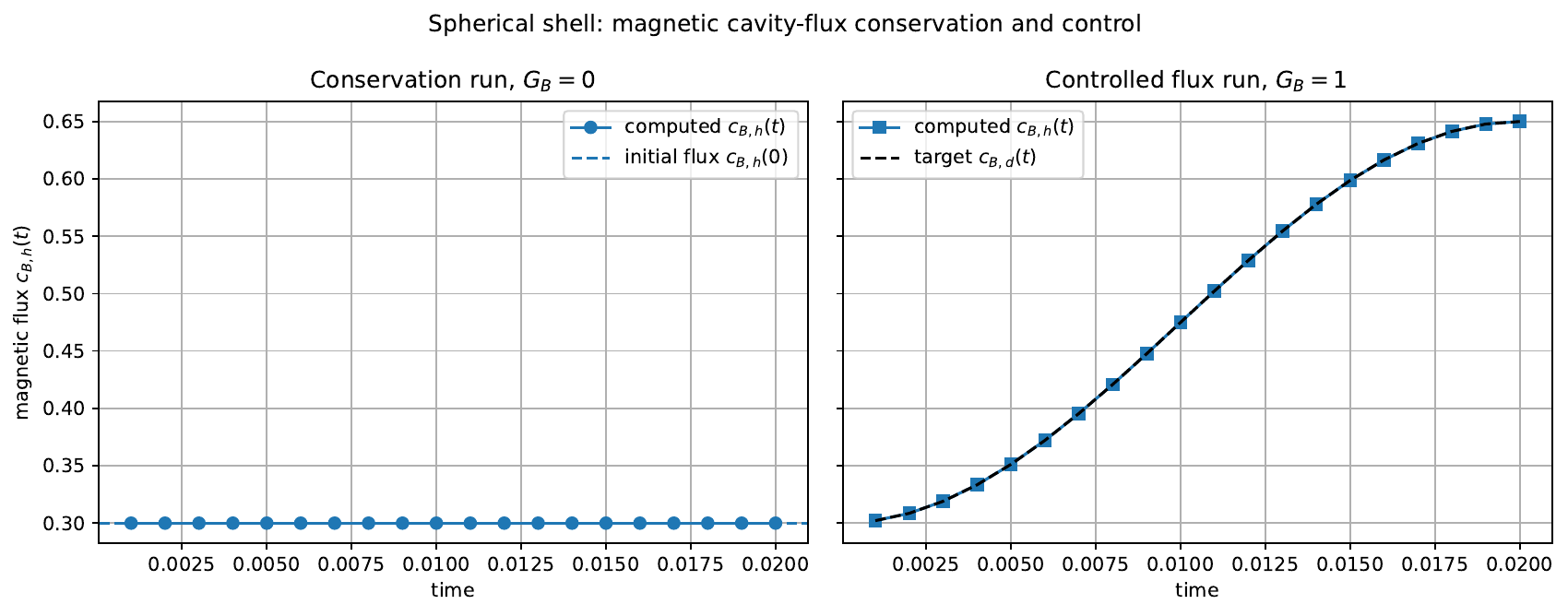}}{%
\fbox{\parbox[c][5.2cm][c]{.78\linewidth}{\centering
\textbf{Figure file unavailable in this compiled draft}\\[2mm]
\texttt{maxwell\_shell\_flux.pdf} was not found.}}}
\caption{Magnetic cavity flux on the spherical shell.  Without global
actuation the magnetic flux is conserved; with a topological flux actuator it
becomes a controlled scalar state.}
\label{fig:maxwell-shell-flux}
\end{figure}

\begin{table}[htbp]
\centering
\small
\begin{tabular}{lrrrrr}
\toprule
run
& $\max_t|c_B(t)-c_B(0)|$
& flux error
& $\|a_B\|_{L^2(0,T)}$
& Gauss defect
& grad. check\\
\midrule
conservation, $G_B=0$
& $0$
& --
& --
& $7.40\times10^{-3}$
& --\\
controlled flux, $G_B=1$
& --
& $1.43\times10^{-6}$
& $5.14\times10^{-1}$
& $1.60\times10^{-2}$
& $2.44\times10^{-10}$\\
\bottomrule
\end{tabular}
\caption{Spherical-shell Maxwell diagnostics.  In the conservation run,
\(G_B=0\), the magnetic flux remains equal to its initial value
\(c_{B,h}(0)=0.30\).  In the controlled run, \(G_B=1\), the scalar
topological actuator drives the flux to the terminal target
\(c_{B,T}=0.65\).  The reported flux error for the controlled run is the
terminal mismatch \(|c_{B,h}(T)-c_{B,T}|\).  The Gauss defect is reported as
a diagnostic quantity and is not at round-off level in this shell experiment.}
\label{tab:maxwell-shell-results}
\end{table}

The conservation run is as important as the controlled run.  It verifies that
the compatible discretization does not intentionally actuate the nontrivial
magnetic-flux coordinate when \(G_B=0\).  The controlled run then tests the
distinct physical assumption encoded by the global or boundary flux actuator.
The nonzero Gauss defects in Table~\ref{tab:maxwell-shell-results} are retained
as diagnostic limitations of the shell computation.
\section{Maxwell space--time refinement and solver diagnostics}
\label{sec:maxwell-refinement}

The four Maxwell experiments are accompanied by the common refinement
summary in Table~\ref{tab:maxwell-refinement}.  For each geometry, the spatial
mesh and time step are refined simultaneously.  The table reports the state
error, the topological period or flux check error where applicable, the
control error, and the recorded solver iteration count.  For the L-shaped
domain no period or flux error is reported because the domain is simply
connected.  For the torus and double-torus geometries, the period-error column
records the independent period-check error after normalization of the discrete
harmonic basis.  For the spherical shell, the same column records the
independent magnetic-flux check error.
\begin{table}[htbp]
\centering
\small
\resizebox{\textwidth}{!}{%
\begin{tabular}{llrrrrrr}
\toprule
geometry & level & cells & $\Delta t$ & state error & period error & control error & iterations\\
\midrule
L-shape & 1 & 196  & $2.00\times10^{-2}$ & $1.982316\times10^{-1}$ & -- & $2.379492\times10^{-2}$ & 29\\
L-shape & 2 & 328  & $1.00\times10^{-2}$ & $1.712027\times10^{-1}$ & -- & $1.869428\times10^{-2}$ & 30\\
L-shape & 3 & 592  & $5.00\times10^{-3}$ & $1.455778\times10^{-1}$ & -- & $1.050521\times10^{-2}$ & 29\\
L-shape & 4 & 597  & $2.50\times10^{-3}$ & $1.457797\times10^{-1}$ & -- & $1.099208\times10^{-2}$ & 30\\
L-shape & 5 & 1435 & $1.00\times10^{-3}$ & $1.062155\times10^{-1}$ & -- & $5.921656\times10^{-3}$ & 31\\
\midrule
torus & 1 & 260  & $2.00\times10^{-2}$ & $2.874825\times10^{-2}$ & $1.723937\times10^{-4}$ & $2.873814\times10^{-2}$ & 44\\
torus & 2 & 427  & $1.00\times10^{-2}$ & $2.768007\times10^{-2}$ & $1.688053\times10^{-4}$ & $2.767034\times10^{-2}$ & 37\\
torus & 3 & 803  & $5.00\times10^{-3}$ & $2.646860\times10^{-2}$ & $1.001147\times10^{-4}$ & $2.645930\times10^{-2}$ & 34\\
torus & 4 & 1370 & $2.50\times10^{-3}$ & $2.257529\times10^{-2}$ & $6.475845\times10^{-5}$ & $2.256735\times10^{-2}$ & 34\\
torus & 5 & 3112 & $1.00\times10^{-3}$ & $1.652963\times10^{-2}$ & $3.656942\times10^{-5}$ & $1.652383\times10^{-2}$ & 35\\
\midrule
double torus & 1 & 1126  & $2.00\times10^{-2}$ & $7.988629\times10^{-1}$ & $2.620304\times10^{-4}$ & $2.273720\times10^{-1}$ & 38\\
double torus & 2 & 1469  & $1.00\times10^{-2}$ & $7.503069\times10^{-1}$ & $5.667683\times10^{-4}$ & $2.138072\times10^{-1}$ & 37\\
double torus & 3 & 2512  & $5.00\times10^{-3}$ & $5.460463\times10^{-1}$ & $5.175228\times10^{-4}$ & $1.558284\times10^{-1}$ & 35\\
double torus & 4 & 4492  & $2.50\times10^{-3}$ & $4.032653\times10^{-1}$ & $4.038762\times10^{-4}$ & $1.163758\times10^{-1}$ & 37\\
double torus & 5 & 10103 & $1.00\times10^{-3}$ & $3.051726\times10^{-1}$ & $1.407430\times10^{-4}$ & $8.687740\times10^{-2}$ & 37\\
\midrule
spherical shell & 1 & 214  & $2.00\times10^{-2}$ & $1.475494\times10^{-1}$ & $9.532631\times10^{-3}$ & $4.694993\times10^{-2}$ & 36\\
spherical shell & 2 & 359  & $1.00\times10^{-2}$ & $1.383180\times10^{-1}$ & $1.101327\times10^{-2}$ & $4.401253\times10^{-2}$ & 40\\
spherical shell & 3 & 883  & $5.00\times10^{-3}$ & $8.744211\times10^{-2}$ & $5.712641\times10^{-3}$ & $2.782391\times10^{-2}$ & 43\\
spherical shell & 4 & 1596 & $2.50\times10^{-3}$ & $7.964111\times10^{-2}$ & $3.216381\times10^{-3}$ & $2.534165\times10^{-2}$ & 40\\
spherical shell & 5 & 3626 & $1.00\times10^{-3}$ & $5.767274\times10^{-2}$ & $7.815220\times10^{-5}$ & $1.835135\times10^{-2}$ & 43\\
\bottomrule
\end{tabular}%
}
\caption{Space--time refinement diagnostics for the Maxwell optimal-control experiments on non-convex and multiply connected geometries. The L-shaped domain is used as a non-convex simply connected baseline, so no topological period is reported there. For the torus and double-torus cases, the period error denotes the independent period-check error after normalization of the discrete harmonic basis. For the spherical-shell case, the same column records the independent flux-check error, since the relevant topological invariant is magnetic flux rather than a one-cycle period. The state error denotes the combined state error for the L-shaped, torus, and double-torus cases, and the magnetic-field projection error for the spherical-shell case.}
\label{tab:maxwell-refinement}
\end{table}

Taken together, the Maxwell experiments close the same loop as the stationary computations.  Topology is first identified combinatorially, represented by period-normalized harmonic bases, inserted as finite-dimensional physical states, and finally tested through conservation or reachability.  The L-shaped baseline is deliberately included so that error behaviour caused by re-entrant geometry can be separated from effects associated with nontrivial topology.

The refinement levels in Table~\ref{tab:maxwell-refinement} refine the mesh size and the time step simultaneously.  Therefore, the observed decay represents a coupled space--time refinement effect rather than a purely spatial or purely temporal convergence order.  Separate spatial rates would require fixing \(\Delta t\) and refining only the mesh, while separate temporal rates would require fixing the mesh and refining only \(\Delta t\).  The present table is therefore best interpreted as a robustness and consistency test across increasingly resolved space--time discretizations.

\section{Conclusion}
We have formulated optimal control on Hilbert complexes by separating the gauge-fixed local state from the physical harmonic state. The mixed Hodge--Laplace variable $p$ remains a compatibility component of the forcing, while the additional state $h\in\HH^k$ carries physically observable circulation or flux. The period map $\Piop$ provides finite-dimensional coordinates for these global modes, and the actuator relation $\Piop h=Ga+c_0$ separates the topological dimension $b_k$ from the controllable dimension $\operatorname{rank}(G)$. The resulting weak optimality system stays in the natural Hilbert-complex domains, and the FEEC discretization preserves both the local mixed PDE and the period-coordinate subsystem.

The stationary experiments show complementary aspects of this structure. The contractible L-shaped domain isolates local FEEC convergence and mesh-independent reduced iterations. The torus realizes one circulation coordinate, the two-hole domain realizes a non-orthogonal two-dimensional harmonic space and makes actuator rank and alignment observable, and the spherical shell realizes the analogous construction at degree two through cavity flux. The comparison between the L-shaped and two-hole geometries further shows that reentrant geometry alone does not determine the observed rate: topology can force the harmonic representative into a singular regularity class and thereby control convergence of the topological subsystem.

For time-dependent Maxwell control, harmonic electric circulation and magnetic flux become finite-dimensional dynamical states coupled to the local FEEC evolution. The electric coordinates are driven by the harmonic component of current and by optional loop actuators, whereas magnetic cavity flux is conserved unless a physically distinct global or boundary flux actuator is present. This makes conservation and reachability of cohomological modes explicit control-theoretic properties rather than artifacts of the linear solver. Future work will address nonlinear constitutive laws and actuator models in which the finite-dimensional coupling matrices are derived directly from electromagnetic, fluid-mechanical, or other application-specific hardware.






\bibliographystyle{plain}
\bibliography{refs,refs-mjh}

\begin{thebibliography}{10}

\bibitem{antil2026structure}
Harbir Antil, Yaw Owusu-Agyemang, Rohit Khandelwal, Jimmie Adriazola, and Denis
  Ridzal.
\newblock Structure-preserving optimal control of maxwell's equations with
  applications to source cloaking.
\newblock {\em arXiv preprint arXiv:2605.00212}, 2026.

\bibitem{AFW2010}
D.N. Arnold, R.S. Falk, and R.~Winther.
\newblock {Finite element exterior calculus: from {H}odge theory to numerical
  stability}.
\newblock {\em Bulletin of the American Mathematical Society}, 47(2):281--354,
  2010.

\bibitem{arnold2010finite}
Douglas Arnold, Richard Falk, and Ragnar Winther.
\newblock Finite element exterior calculus: from hodge theory to numerical
  stability.
\newblock {\em Bulletin of the American mathematical society}, 47(2):281--354,
  2010.

\bibitem{Babuska.I1971}
I.~Babu{\v s}ka.
\newblock Error bounds for the finite element method.
\newblock {\em Numerische Mathematik}, 16:322--333, 1971.

\bibitem{brizitskii2010inverse}
RV~Brizitskii and AS~Savenkova.
\newblock Inverse extremum problems for maxwell’s equations.
\newblock {\em Computational Mathematics and Mathematical Physics},
  50(6):984--992, 2010.

\bibitem{bru1992hilbert}
J~Bru, Matthias Lesch, et~al.
\newblock Hilbert complexes.
\newblock {\em Journal of Functional Analysis}, 108(1):88--132, 1992.

\bibitem{BrLe92}
J.~Br\"uning and M.~Lesch.
\newblock {Hilbert Complexes}.
\newblock {\em J. Funct. Anal.}, 108(1):88--132, August 1992.

\bibitem{CASTILLOREYES2022105030}
Octavio Castillo-Reyes, David Modesto, Pilar Queralt, Alex Marcuello, Juanjo
  Ledo, Adrian Amor-Martin, Josep {de la Puente}, and Luis~Emilio
  García-Castillo.
\newblock {3D} magnetotelluric modeling using high-order tetrahedral
  {N\'{e}d\'{e}lec} elements on massively parallel computing platforms.
\newblock {\em Computers and Geosciences}, 160:105030, 2022.

\bibitem{de2015numerical}
Juan~Carlos De~los Reyes.
\newblock {\em Numerical PDE-constrained optimization}.
\newblock Springer, 2015.

\bibitem{Ev98}
Lawrence~C. Evans.
\newblock {\em Partial Differential Equations}.
\newblock {Graduate Studies in Mathematics}. American Mathematical Society,
  Providence, RI, 1998.

\bibitem{gawlik2021local}
Evan Gawlik, Michael~J Holst, and Martin~W Licht.
\newblock Local finite element approximation of sobolev differential forms.
\newblock {\em ESAIM: Mathematical Modelling and Numerical Analysis},
  55(5):2075--2099, 2021.

\bibitem{hinze2008optimization}
Michael Hinze, Ren{\'e} Pinnau, Michael Ulbrich, and Stefan Ulbrich.
\newblock {\em Optimization with PDE constraints}, volume~23.
\newblock Springer Science \& Business Media, 2008.

\bibitem{HolstStern2012a}
Michael Holst and Ari Stern.
\newblock Geometric variational crimes: {Hilbert} complexes, finite element
  exterior calculus, and problems on hypersurfaces.
\newblock {\em Foundations of Computational Mathematics}, 12(3):263--293, 2012.

\bibitem{komkov2006optimal}
Vadim Komkov.
\newblock {\em Optimal control theory for the damping of vibrations of simple
  elastic systems}, volume 253.
\newblock Springer, 2006.

\bibitem{leopardi2016abstract}
Paul Leopardi and Ari Stern.
\newblock The abstract hodge--dirac operator and its stable discretization.
\newblock {\em SIAM Journal on Numerical Analysis}, 54(6):3258--3279, 2016.

\bibitem{li2020improved}
Changwei Li, Wenwu Wan, and Weiwen Song.
\newblock An improved nodal finite-element method for magnetotelluric modeling.
\newblock {\em IEEE Journal on Multiscale and Multiphysics Computational
  Techniques}, 5:265--272, 2020.

\bibitem{licht2024symmetry}
Martin~W Licht.
\newblock Symmetry and invariant bases in finite element exterior calculus.
\newblock {\em Foundations of Computational Mathematics}, 24(4):1185--1224,
  2024.

\bibitem{licht2017complexes}
Martin~Werner Licht.
\newblock Complexes of discrete distributional differential forms and their
  homology theory.
\newblock {\em Foundations of Computational Mathematics}, 17(4):1085--1122,
  2017.

\bibitem{nicaise2017optimal}
Serge Nicaise and Fredi Tr{\"o}ltzsch.
\newblock Optimal control of some quasilinear maxwell equations of parabolic
  type.
\newblock {\em Discrete \& Continuous Dynamical Systems-S}, 10(6):1375, 2017.

\bibitem{schiela2020optimal}
Anton Schiela and Matthias Stoecklein.
\newblock Optimal control of static contact in finite strain elasticity.
\newblock {\em ESAIM: Control, Optimisation \& Calculus of Variations}, 26,
  2020.

\bibitem{schwedes2017mesh}
Tobias Schwedes, David~A Ham, Simon~W Funke, and Matthew~D Piggott.
\newblock {\em Mesh Dependence in PDE-Constrained Optimisation: An Application
  in Tidal Turbine Array Layouts}.
\newblock Springer, 2017.

\bibitem{tran2022optimal}
Quyen Tran, Harbir Antil, and Hugo D{\'\i}az.
\newblock Optimal control of parameterized stationary maxwell's system: Reduced
  basis, convergence analysis, and a posteriori error estimates.
\newblock {\em Mathematical Control \& Related Fields}, 2022.

\bibitem{ulbrich2002semismooth}
Michael Ulbrich.
\newblock Semismooth newton methods for operator equations in function spaces.
\newblock {\em SIAM Journal on Optimization}, 13(3):805--841, 2002.

\bibitem{yousept2012optimal}
Irwin Yousept.
\newblock Optimal control of maxwell’s equations with regularized state
  constraints.
\newblock {\em Computational Optimization and Applications}, 52(2):559--581,
  2012.

\bibitem{yousept2017optimal}
Irwin Yousept.
\newblock Optimal control of non-smooth hyperbolic evolution maxwell equations
  in type-ii superconductivity.
\newblock {\em SIAM Journal on Control and Optimization}, 55(4):2305--2332,
  2017.

\end{thebibliography}

\clearpage
\appendix

\end{document}